\documentclass{lmcs} %%% last changed 2014-08-20
\pdfoutput=1

\usepackage[utf8]{inputenc}

\usepackage{lastpage}
\lmcsdoi{17}{3}{23}
\lmcsheading{}{\pageref{LastPage}}{}{}%
{Nov.~26,~2020}{Sep.~07,~2021}{}


\newcommand{\cact}[1]{\overline{#1}}
\newcommand{\D}[0]{\mathsf{D}}
\newcommand{\dd}[0]{{\boldsymbol\partial}}
\newcommand{\pair}[2]{\left\langle{#1},{#2}\right\rangle}
\newcommand{\four}[4]{\pair{\pair{#1}{#2}}{\pair{#3}{#4}}}

\newcommand{\pa}[1]{\left({#1}\right)}

\newcommand{\A}[1]{\mathbf{#1}}

\newcommand{\CdCax}[1]{{\bf [C$\dd$.{#1}]}}
\newcommand{\CdlCax}[1]{{\bf [C$\dd\lambda$.{#1}]}}
\newcommand{\CDCax}[1]{{\bf [CD.{#1}]}}
\newcommand{\CADax}[1]{{\bf [CAD.{#1}]}}
\newcommand{\Eax}[1]{{\bf [E.{#1}]}}
\newcommand{\CAax}[1]{{\bf [CA.{#1}]}}

\newcommand\Item[1][]{%
      \ifx\relax#1\relax  \item \else \item[#1] \fi
      \abovedisplayskip=0pt\abovedisplayshortskip=0pt~\vspace*{-\baselineskip}
}

\begin{document}

\title[Cartesian Difference Categories]{Cartesian Difference Categories}
%\titlecomment{{\lsuper*}OPTIONAL comment concerning the title, \eg,
 % if a variant or an extended abstract of the paper has appeared elsewhere.}

\author[M. Alvarez-Picallo]{Mario Alvarez-Picallo\rsuper{a}}	%required
\address{\lsuper{a}Department of Computer Science, University of Oxford}
\curraddr{Edinburgh Research Center, Huawei Technologies Research \& Development}	%required
\email{mario.alvarez-picallo@cs.ox.ac.uk}  %optional

\author[J-S. P. Lemay]{Jean-Simon Pacaud Lemay\rsuper{b}}	%optional
\address{\lsuper{b}Department of Computer Science, University of Oxford}	%optional
\email{jean-simon.lemay@cs.ox.ac.uk}  %optional
\thanks{The second author is financially supported by Kellogg College, the Oxford-Google Deep Mind Graduate Scholarship, and the Clarendon Fund.}	%optional

%% etc.

%% required for running head on odd and even pages, use suitable
%% abbreviations in case of long titles and many authors:

%%%%%%%%%%%%%%%%%%%%%%%%%%%%%%%%%%%%%%%%%%%%%%%%%%%%%%%%%%%%%%%%%%%%%%%%%%%

%% the abstract has to PRECEDE the command \maketitle:
%% be sure not to issue the \maketitle command twice!

\begin{abstract}
Cartesian differential categories are categories equipped with a differential combinator which axiomatizes the directional derivative. Important models of Cartesian differential categories include classical differential calculus of smooth functions and categorical models of the differential $\lambda$-calculus. However, Cartesian differential categories cannot account for other interesting notions of differentiation of a more discrete nature such as the calculus of finite differences. On the other hand, change action models have been shown to capture these examples as well as more ``exotic'' examples of differentiation. But change action models are very general and do not share the nice properties of Cartesian differential categories. In this paper, we introduce Cartesian difference categories as a bridge between Cartesian differential categories and change action models. We show that every Cartesian differential category is a Cartesian difference category, and how certain well-behaved change action models are Cartesian difference categories. In particular, Cartesian difference categories model both the differential calculus of smooth functions and the calculus of finite differences. Furthermore, every Cartesian difference category comes equipped with a tangent bundle monad whose Kleisli category is again a Cartesian difference category.
\end{abstract}

\maketitle

%% start the paper here:
\allowdisplaybreaks%

\section*{Acknowledgment}
  \noindent The authors would like to thank the anonymous reviewers for their comments and suggestions which helped improve this paper, as well as suggestions of potential interesting examples of Cartesian difference categories to study in the future (which we discuss in the conclusion).

  \section{Introduction}

In the early 2000's, Ehrhard and Regnier introduced the differential $\lambda$-calculus~\cite{ehrhard2003differential}, an extension of the $\lambda$-calculus equipped with a differential combinator capable of taking the derivative of arbitrary higher-order functions. This development, based on models of linear logic equipped with a natural notion of ``derivative''~\cite{ehrhard2018introduction}, sparked a wave of research into categorical models of differentiation.
One of the most notable developments in the area are Cartesian differential categories~\cite{blute2009cartesian}, introduced by Blute, Cockett and Seely, which provide an abstract categorical axiomatization of the directional derivative from differential calculus. The relevance of Cartesian differential categories lies in their ability to model both ``classical'' differential calculus (with the canonical example being the category of Euclidean spaces and smooth functions between) and the differential $\lambda$-calculus (as every categorical model for it gives rise to a Cartesian differential category~\cite{manzonetto2012categorical}). However, while Cartesian differential categories have proven to be an immensely successful formalism, they have, by design, some limitations. Firstly, they cannot account for certain ``exotic'' notions of derivative, such as the difference operator from the calculus of finite differences~\cite{richardson1954introduction}. This is because the axioms of a Cartesian differential category stipulate that derivatives should be linear in their second argument (in the same way that the directional derivative is), whereas these aforementioned discrete sorts of derivatives need not be. Additionally, every Cartesian differential category is equipped with a tangent bundle monad~\cite{cockett2014differential,manzyuk2012tangent} whose Kleisli category can be intuitively understood as a category of generalized vector fields.

More recently, discrete derivatives have been suggested as a semantic framework for understanding incremental computation. This led to the development of change structures~\cite{cai2014theory} and change actions~\cite{alvarez2019change}. Change action models have been successfully used to provide a model for incrementalizing Datalog programs~\cite{alvarez2019fixing}, and have also been shown to model the calculus of finite differences. Change action models, however, are very general, lacking many of the nice properties of Cartesian differential categories (for example, addition in a change action model is not required to be commutative), even though they are verified in most change action models. As a consequence of this generality, the tangent bundle endofunctor in a change action model can fail to be a monad.

In this work, we introduce Cartesian difference categories (Section~\ref{CdCsec}), whose key ingredients are an infinitesimal extension operator and a difference combinator, whose axioms are a generalization of the differential combinator axioms of a Cartesian differential category. In Section~\ref{CDCisCdCsec}, we show that every Cartesian differential category is a Cartesian difference category whose infinitesimal extension operator is zero, and conversely how every Cartesian difference category admits a full subcategory which is a Cartesian differential category. In Section~\ref{CdCisCAsec}, we show that every Cartesian difference category is a change action model, and conversely how a full subcategory of suitably well-behaved objects and maps of a change action model is a Cartesian difference category. In Section~\ref{EXsec} we provide some examples of Cartesian difference categories; notably, the calculus of finite differences and the stream calculus. In Section~\ref{monadsec}, we show that every Cartesian difference category comes equipped with a monad whose Kleisli category (Section~\ref{Kleislisec}) and a certain full sub-category of its Eilenberg-Moore category (Section~\ref{EMsec}) are again both Cartesian difference categories. Lastly, in Section~\ref{sec:differential-lambda-categories}, we briefly discuss difference $\lambda$-categories, which are Cartesian difference categories that are Cartesian closed. We conclude with a discussion on future work regarding Cartesian difference categories.

This paper is an extended version of a conference paper~\cite{alvarez2020cartesian} for Foundations of Software Science and Computation Structures: 23rd International Conference (FOSSACS 2020). While the overall structure and story of both versions are very similar, this version includes additional details, proofs, and results, as well as some important corrections. We highlight the main differences between this paper and the conference paper. Probably the most important change to point out is the correction made in Section~\ref{Kleislisec}. In this version, we correct the Cartesian difference structure of the Kleisli category of the tangent bundle monad. The proposed infinitesimal extension and difference combinator for said Kleisli category in the conference paper~\cite{alvarez2020cartesian} was based on a result from another paper~\cite{alvarez2019change}. Unfortunately, we have found that the result in said other paper is incorrect, and therefore so were the proposed infinitesimal extension and difference combinator for the Kleisli category in the conference paper. Section~\ref{CDCsec} and Section~\ref{CAsec}, the background sections, are mostly the same with some added background details. In Section~\ref{sec:CdC}, more detailed proofs and computations are provided throughout, and both Lemma~\ref{lem:d-epsilon} and Lemma~\ref{epsilon-nilponent} are new results. In Section~\ref{CDCisCdCsec}, we give a specific name to objects whose infinitesimal extensions is zero, which in Definition~\ref{def:ep-vanishing} we call $\varepsilon$-vanishing. In Section~\ref{CdCisCAsec} we introduce flat objects, whose definition (Definition~\ref{def:flat}) is slightly more general than the original one. Section~\ref{EXsec}, the example section, is mostly unchanged with some added details and proofs. In Section~\ref{monadsec}, we provide full detailed computations for all the necessary structure relating to the tangent bundle monad. Section~\ref{EMsec} and Section~\ref{sec:differential-lambda-categories} are both new sections.

  \section*{Conventions:} In an arbitrary category, we use the classical notation for composition as opposed to diagrammatic order which was used in~\cite{blute2009cartesian}: this means that the composite map $g \circ f: A \to C$ is the map which first does $f: A\to B$ then $g: B \to C$. By a Cartesian category we mean a category $\mathbb{X}$ with chosen finite products where we denote the binary product of objects $A$ and $B$ by $A \times B$ with projection maps $\pi_0: A \times B \to A$ and $\pi_1: A \times B \to B$ and pairing operation $\langle -, - \rangle$, so the product of maps is $f \times g = \langle f \circ \pi_0, g \circ \pi_1 \rangle$, and the chosen terminal object as $\top$ with unique terminal maps $!_A: A \to \top$.

\section{Cartesian Differential Categories}\label{CDCsec}

In this section, we very briefly review Cartesian differential categories, so that the reader may compare Cartesian differential categories with the new notion of Cartesian \emph{difference} categories which we introduce in Section~\ref{CdCsec}. For a full detailed introduction on Cartesian differential categories, we refer the reader to the original paper~\cite{blute2009cartesian}.

\subsection{Cartesian Left Additive Categories}

Here we recall the definition of Cartesian left additive categories~\cite{blute2009cartesian}, which is the underlying structure of both Cartesian differential categories and Cartesian difference categories. Here ``left additive'' is meant being skew enriched~\cite{Campbell2018} over commutative monoids, so in particular, this means that we do not assume the existence of additive inverses, i.e., ``negative elements''. This left additive structure allows one to sum maps and have zero maps while allowing for maps that do not preserve the sum or zero. Maps that do preserve the additive structure are called additive maps.

\begin{defiC}[{\cite[Definition 1.1.1]{blute2009cartesian}}] A \textbf{left additive category} is a category $\mathbb{X}$ such that each hom-set $\mathbb{X}(A,B)$ is a commutative monoid with addition operation $+: \mathbb{X}(A,B) \times \mathbb{X}(A,B) \to \mathbb{X}(A,B)$ and zero element (called the zero map) $0 \in \mathbb{X}(A,B)$, such that pre-composition preserves the additive structure:
\begin{align*}
(f + g) \circ h = f \circ h + g \circ h && 0 \circ f = 0
\end{align*}
A map $k$ is \textbf{additive} if post-composition by $k$ preserves the additive structure:
\begin{align*}
k \circ (f +g) = k \circ f + k \circ g && k \circ 0=0
\end{align*}
\end{defiC}

\begin{defiC}[{\cite[Definition 2.3]{lemay2018tangent}}]\label{CLACdef} A \textbf{Cartesian left additive category} is a Cartesian category $\mathbb{X}$ which is also a left additive category such all projection maps $\pi_0: A \times B \to A$ and $\pi_1: A \times B \to B$ are additive.
\end{defiC}

Examples of Cartesian left additive categories can be found in Section~\ref{EXsec}. We note that the definition given here of a Cartesian left additive category is slightly different from the one found in~\cite[Definition 1.2.1]{blute2009cartesian}, which also assumed that the pairing of additive maps be additive, but it is indeed equivalent as explained in~\cite{lemay2018tangent}. Furthermore, by~\cite[Proposition 1.2.2]{blute2009cartesian}, an equivalent axiomatization of a Cartesian left additive category is that of a Cartesian category where every object comes equipped with a commutative monoid structure such that the projection maps are monoid morphisms, which we discuss further in Lemma~\ref{caelem}. Here are some useful identities on how the additive and product structures are compatible:

\begin{lemC}[{\cite[Lemma 2.4]{lemay2018tangent}}]\label{claclemma} In a Cartesian left additive category $\mathbb{X}$:
\begin{enumerate}[\emph{(i)}]
\item $\langle f ,g \rangle + \langle h, k \rangle = \langle f + h, g+k \rangle$ and $\langle 0, 0 \rangle = 0$;
\item If $f: C \to A$ and $g: C \to B$ are additive then $\langle f,g \rangle: C \to A \times B$ is additive;
\item If $h: A \to B$ and $k: C \to D$ are additive then $h \times k: A \times C \to B \times D$ is additive;
\item For any object $A$, the unique map to the terminal object $\top$ is the zero map, $!_A = 0$.
\end{enumerate}
\end{lemC}

\subsection{Cartesian Differential Categories}

Cartesian differential categories are Cartesian left additive categories that come equipped with a differential combinator, which in turn is axiomatized by the basic properties of the directional derivative from multivariable differential calculus. In the following definition, it is important to note that here, following the more recent work on Cartesian differential categories, we've flipped the convention found in~\cite{blute2009cartesian}, so that the linear argument is in the second argument rather than in the first argument.

\begin{defiC}[{\cite[Definition 2.1.1]{blute2009cartesian}}]\label{cartdiffdef} A \textbf{Cartesian differential category} is a Cartesian left additive category equipped with a \textbf{differential combinator} $\mathsf{D}$ of the form
\[ \frac{f : A \to B}{\mathsf{D}[f]: A \times A \to B} \]
 verifying the following coherence conditions:
 \begin{enumerate}[{\CDCax{\arabic*}},ref={\CDCax{\arabic*}},align=left]
\item\label{CDCax1} $\mathsf{D}[f+g] = \mathsf{D}[f] + \mathsf{D}[g]$ and $\mathsf{D}[0]=0$
\item\label{CDCax2} $\mathsf{D}[f] \circ \langle x, y+z \rangle= \mathsf{D}[f] \circ \langle x, y \rangle + \mathsf{D}[f] \circ \langle x, z \rangle$ and $\mathsf{D}[f] \circ \langle x, 0 \rangle=0$
\item\label{CDCax3} $\mathsf{D}[1_A]=\pi_1$ and $\mathsf{D}[\pi_0] = \pi_0 \circ \pi_1$ and $\mathsf{D}[\pi_1] = \pi_1 \circ \pi_1$
\item\label{CDCax4} $\mathsf{D}[\langle f, g \rangle] = \langle  \mathsf{D}[f], \mathsf{D}[g] \rangle$ and $\mathsf{D}[!_A]=!_{A \times A}$
\item\label{CDCax5} $\mathsf{D}[g \circ f] = \mathsf{D}[g] \circ \langle f \circ \pi_0, \mathsf{D}[f] \rangle$
\item\label{CDCax6} $\mathsf{D}\left[\mathsf{D}[f] \right] \circ \left \langle \langle x, y \rangle, \langle 0, z \rangle \right \rangle=  \mathsf{D}[f] \circ \langle x, z \rangle$
\item\label{CDCax7} $\mathsf{D}\left[\mathsf{D}[f] \right] \circ \left \langle \langle x, y \rangle, \langle z, 0 \rangle \right \rangle = \mathsf{D}\left[\mathsf{D}[f] \right] \circ \left \langle \langle x, z \rangle, \langle y, 0 \rangle \right \rangle$
\end{enumerate}
\end{defiC}

\noindent
As we will see in Section~\ref{CDCisCdCsec}, a Cartesian difference category is a generalization of a Cartesian differential category. Therefore, we leave the discussion of the intuition of these axioms for later in Section~\ref{CdCsec} below. We also refer to~\cite[Section 4]{blute2009cartesian} for a term calculus which may help better understand the axioms of a Cartesian differential category. We highlight that the last two axioms~\ref{CDCax6} and~\ref{CDCax7} can equivalently be expressed as follows:

\begin{lemC}[{\cite[Proposition 4.2]{cockett2014differential}}] In the presence of the other axioms,~\ref{CDCax6} and~\ref{CDCax7} are equivalent to the following two axioms:
  \begin{enumerate}[\CDCax{\arabic*.{a}},ref={\CDCax{\arabic*.{a}}},align=left]
 \setcounter{enumi}{5}
\item\label{CDCax6a} $\mathsf{D}\left[\mathsf{D}[f] \right] \circ \left \langle \langle x, 0 \rangle, \langle 0, y \rangle \right \rangle= \mathsf{D}[f] \circ  \langle x, y \rangle$
\item\label{CDCax7a} $\mathsf{D}\left[\mathsf{D}[f] \right] \circ \left \langle \langle x, y \rangle, \langle z, w \rangle \right \rangle= \mathsf{D}\left[\mathsf{D}[f] \right] \circ \left \langle \langle x, z \rangle, \langle y, w \rangle \right \rangle$
\end{enumerate}
\end{lemC}

\noindent
The canonical example of a Cartesian differential category is the category of real smooth functions, which we will discuss in Section~\ref{smoothex}. Other interesting examples of Cartesian differential categories can be found throughout the literature such as categorical models of the differential $\lambda$-calculus~\cite{ehrhard2003differential,manzonetto2012categorical}, the subcategory of differential objects of a tangent category~\cite{cockett2014differential}, and the coKleisli category of a differential category~\cite{blute2006differential,blute2009cartesian}.

\section{Change Action Models}\label{CAsec}

Change actions~\cite{alvarez2019fixing,alvarez2019change} have recently been proposed as a setting for reasoning about higher-order incremental computation, based on a discrete notion of differentiation. Together with Cartesian differential categories, they provide the core ideas behind Cartesian difference categories. In this section, we quickly review change actions and change action models, in particular, to highlight where some of the axioms of a Cartesian difference category come from. For more details on change actions, we invite readers to see the original paper~\cite{alvarez2019change}.

\subsection{Change Actions}
The basic intuition for a change action is an object $A$ equipped with a monoid $\Delta A$ which represents the possible ``changes'' or ``updates'' that might be applied to $A$, with the monoid structure on $\Delta A$ representing the capability to compose updates.

\begin{defiC}[{\cite[Section 2]{alvarez2019change}}]
    A \textbf{change action} $\cact{A}$ in a Cartesian category $\mathbb{X}$ is a quintuple:
    \[\cact{A} \equiv (A, \Delta A, \oplus_A, +_A, 0_A)\]
    consisting of two objects $A$ and $\Delta A$, and three maps:
    \begin{align*}
     \oplus_A: A \times \Delta A \to A && +_A: \Delta A \times \Delta A \to \Delta A && 0_A: \top \to \Delta A
    \end{align*}
such that:
  \begin{enumerate}[{\CAax{\arabic*}},ref={\CAax{\arabic*}},align=left]
        \item\label{CAax1} $(\Delta A, +_A, 0_A)$ is a monoid, that is, the following equalities hold:
\[+_A \circ \langle 1_{\Delta A}, 0_{\Delta A} \circ{} !_{\Delta A} \rangle = 1_{\Delta A} = +_{\Delta A} \circ \langle 0_{\Delta A} \circ{} !_{\Delta_A}, 1_{\Delta A} \rangle\]
\begin{align*}
+_A \circ (1_{\Delta A} \times +_A) = +_A \circ \pair{+_A \circ (1_{\Delta A} \times \pi_0)}{\pi_1 \circ \pi_1}
\end{align*}
%\begin{align*}
%+_A \circ (1_{\Delta A} \times +_A) = +_A \circ (+_A \times 1_{\Delta A})
%\end{align*}
  \item\label{CAax2} $\oplus_A : A \times \Delta A \to A$ is an action of $\Delta A$ on $A$, that is, the following equalities hold:
  \begin{align*}
\oplus_A \circ \langle 1_A, 0_A \circ{} !_A \rangle = 1_A
\end{align*}
\begin{align*}
\oplus_A \circ (1_A \times +_A) = \oplus_A \circ \pair{\oplus_A \circ (1_A \times \pi_0)}{\pi_1 \circ \pi_1}
\end{align*}
\end{enumerate}
\end{defiC}

\noindent
In~\ref{CAax1}, the first line says that $0_A$ is a unit for $+_A$ while the second line is the associativity of $+_A$. In~\ref{CAax2}, the first line is the unit law of the module action, while the second line is the associativity law of the module action. These are highlighted in Lemma~\ref{lem:CAadd} below. If $\cact{A} \equiv (A, \Delta A, \oplus_A, +_A, 0_A)$ is a change action, then for a pair of parallel maps $h: C \to \Delta A$ and $k: C \to \Delta A$, we define their $\Delta A$-sum $h +_{\cact{A}} k: C \to \Delta A$ as follows:
\[ h +_{\cact{A}} k = +_A \circ \langle h, k \rangle\]
Similarly, for a pair of maps $f: C \to A$ and $g: C \to \Delta A$, we define $f \oplus_{\cact{A}} g: C \to A$ as:
\[f \oplus_{\cact{A}} g = \oplus_A \circ \langle f, g \rangle \]
Here, $f \oplus_{\cact{A}} g$ should be thought of as $g$ acting on $f$.

\begin{lem}\label{lem:CAadd}
In a Cartesian category $\mathbb{X}$, let $\cact{A} \equiv (A, \Delta A, \oplus_A, +_A, 0_A)$ be a change action. Then for any suitable maps:
\begin{enumerate}[(\roman{enumi}),ref={\thelem.\roman{enumi}}]
\item $f+_{\cact{A}} (0_A \circ{} !_C) = f = (0_A \circ{} !_C) +_{\cact{A}} f$
\item $f +_{\cact{A}} (g +_{\cact{A}} h) = (f +_{\cact{A}} g) +_{\cact{A}} h$
\item\label{lem:CAadd.0} $f \oplus_{\cact{A}} (0_A \circ{} !_C) = f$
\item\label{lem:CAadd.+}  $f\oplus_{\cact{A}} (g +_{\cact{A}} h) = (f \oplus_{\cact{A}} g) \oplus_{\cact{A}} h$
\item\label{lem:CAadd.circ} $(f \oplus_{\cact{A}} g) \circ h = (f \circ h) \oplus_{\cact{A}} (g \circ h)$
\end{enumerate}
\end{lem}
\begin{proof}We leave these as a simple exercises for the reader. Briefly, $(i)$ and $(ii)$ will follow from~\ref{CAax1}, while $(iii)$ and $(iv)$ will follow from~\ref{CAax2}, and $(v)$ follows directly from the definition.
\end{proof}

Change actions give rise to a notion of derivative, with a distinctly ``discrete'' flavour. Given change actions structure on $A$ and $B$, a map $f : A \to B$ can be said to be differentiable when changes to the input (in the sense of elements of $\Delta A$) are mapped to changes to the output (that is, elements of $\Delta B$). In the setting of incremental computation~\cite{alvarez2019fixing,cai2014theory,kelly2016evolving}, this is precisely what it means for $f$ to be incrementalizable, with the derivative of $f$ corresponding to an incremental version of $f$.

\begin{defiC}[{\cite[Definition 1]{alvarez2019change}}]\label{def:derivative}
    Let
    \[\cact{A} \!\equiv\! (A, \Delta A, \oplus_A, +_A, 0_A) \qquad\text{and}\qquad \cact{B} \!\equiv\! (B, \Delta B, \oplus_B, +_B, 0_B)\]
    be change actions in a Cartesian category $\mathbb{X}$. For a map $f: A \to B$, a map of type $\dd[f] : A \times \Delta A \to \Delta B$ is a \textbf{derivative} of $f$ whenever the following equalities hold:
 \begin{enumerate}[{\CADax{\arabic*}},ref={\CADax{\arabic*}},align=left]
        \item\label{CADax1} $f \circ (x \oplus_{\cact{A}} y) = (f \circ x) \oplus_{\cact{B}} \left(\dd[f] \circ \pair{x}{y}\right)$
        \item\label{CADax2} $\dd[f] \circ \pair{x}{y +_{\cact{A}} z} = (\dd[f] \circ \pair{x}{y}) +_{\cact{B}} (\dd[f] \circ \pair{x \oplus_{\cact{A}} y}{z})$ and \\ \noindent$\dd[f] \circ \langle x, 0_B \circ{} !_B \rangle = 0_B \circ{} !_{A \times \Delta A}$
    \end{enumerate}
\end{defiC}

\noindent
The second axiom~\ref{CADax2} is also known as regularity~\cite[Definition 2]{alvarez2019change}. The intuition for these axioms will be explained in more detail in Section~\ref{CdCsec} when we explain the axioms of a Cartesian difference category. It is important to note that there is nothing in the definition that says that derivatives are necessarily unique, and therefore a map $f$ could have multiple possible derivatives.

\subsection{Change Action Models}

The chain rule for derivatives still holds for change actions, or in other words, differentiation is compositional. Indeed by~\cite[Lemma 1]{alvarez2019change}, whenever $\dd[f]$ and $\dd[g]$ are derivatives for composable maps $f$ and $g$ respectively, then $\dd[g] \circ \langle f \circ \pi_0, \dd[f] \rangle$ is a derivative for $g \circ f$. As a corollary, change actions in $\mathbb{X}$ form a category where the maps are pairs consisting a map $f$ and a derivative of $f$. For more details on this category of change actions, see~\cite[Section 3]{alvarez2019change}.

\begin{defiC}[{\cite[Section 3]{alvarez2019change}}]
    Given a Cartesian category $\mathbb{X}$, define its change actions category $\mathsf{CAct}(\mathbb{X})$  as the category
    whose objects are change actions $\cact{A}$ in $\mathbb{X}$ and whose maps $\cact{f} : \cact{A} \to
    \cact{B}$ are the pairs $(f, \dd[f])$, where $f : A \to B$ is a map in $\mathbb{X}$ and $\dd[f] :
    A \times \Delta A \to \Delta B$ is a derivative for $f$. The identity is $(1_A, \pi_1)$, while composition of $(f, \dd[f])$ and $(g, \dd[g])$ is $(g \circ f, \dd[g] \circ \langle f \circ \pi_0, \dd[f] \rangle)$.
\end{defiC}

For a Cartesian category $\mathbb{X}$, it is straightforward to see that $\mathsf{CAct}(\mathbb{X})$ is also a Cartesian category where the terminal object and the product of objects is the same as in $\mathbb{X}$, where the projection maps are the pairs $(\pi_i, \pi_i \circ \pi_1)$, and where the pairing of maps is given point-wise $\pair{(f, \dd[f])}{(g, \dd[g])} = (\pair{f}{g}, \pair{\dd[f]}{\dd[g]})$. There is an obvious product-preserving forgetful functor $\mathcal{E} : \mathsf{CAct}(\mathbb{X}) \to \mathbb{X}$ sending every change action $(A, \Delta A, \oplus, +, 0)$ to its base object $A$ and every map $(f, \dd[f])$ to the underlying map $f$. Here is a useful lemma which describes the compatibility between change action structure and the product structure.

\begin{lem}\label{lem:CACTprod} Let $\cact{A} \!\equiv\! (A, \Delta A, \oplus_A, +_A, 0_A)$ and $\cact{B} \!\equiv\! (B, \Delta B, \oplus_B, +_B, 0_B)$ be change actions in a Cartesian category $\mathbb{X}$. Then for any suitable maps:
\begin{enumerate}[(\roman{enumi}),ref={\thelem.\roman{enumi}}]
\item\label{lem:CACTprod.oplus} $\pair{f}{g} \oplus_{\cact{A \times B}} \pair{h}{k} = \pair{f \oplus_{\cact{A}} h}{ g \oplus_{\cact{B}} k}$
\item\label{lem:CACTprod.plus} $\pair{f}{g} +_{\cact{A \times B}} \pair{h}{k} = \pair{f +_{\cact{A}} h}{ g +_{\cact{B}} k}$
\item\label{lem:CACTprod.0} $\pair{0_A}{0_B} = 0_{A \times B}$
\end{enumerate}
\end{lem}
\begin{proof} These are immediate from the definition of product structure in $\mathsf{CAct}(\mathbb{X})$, and so we leave this as excercise for the reader.
\end{proof}

As a setting for studying differentiation, the category $\mathsf{CAct}(\mathbb{X})$ is rather lacklustre, since there is no notion of higher derivatives. Instead, one works with change action models. Informally, a change action model consists of a rule which for every object $A$ of $\mathbb{X}$ associates a change action over it, and for every map a choice of a derivative.

\begin{defiC}[{\cite[Definition 5]{alvarez2019change}}]
    A \textbf{change action model} is a Cartesian category $\mathbb{X}$ with a product-preserving functor
    $\alpha : \mathbb{X} \to \mathsf{CAct}(\mathbb{X})$ that is a section of the forgetful
    functor $\mathcal{E}$, that is, $\mathcal{E} \circ \alpha = 1_{\mathbb{X}}$.
\end{defiC}

For brevity, when $A$ is an object of a change action model, we will simply write its associate change action as $\alpha(A) = (A, \Delta A, \oplus_A, +_A)$. Examples of change action models can be found in~\cite{alvarez2019change}. In particular, as was shown in~\cite[Theorem 6]{alvarez2019change}, every Cartesian differential category provides a change model action. We will generalize this result, and show in Section~\ref{CdCisCAsec} that a Cartesian difference category also always provides a change action model.

\section{Cartesian Difference Categories}\label{sec:CdC}

In this section, we introduce \emph{Cartesian difference categories}, which are generalizations of Cartesian differential categories. Examples of Cartesian difference categories can be found in Section~\ref{EXsec}.

\subsection{Infinitesimal Extensions in Left Additive Categories}

We first introduce infinitesimal extensions, which is an operator that turns a map into an ``infinitesimal'' version of itself, in the sense that every map coincides with its Taylor approximation on infinitesimal elements.

 \begin{defi} A Cartesian left additive category $\mathbb{X}$ is said to have an \textbf{infinitesimal extension} $\varepsilon$ if every pair of objects $A$ and $B$, there is a function $\varepsilon: \mathbb{X}(A,B) \to \mathbb{X}(A,B)$ such that:
  \begin{enumerate}[{\Eax{\arabic*}},ref={\Eax{\arabic*}},align=left]
        \item\label{Eax1} $\varepsilon$ is a monoid morphism, that is, $\varepsilon(f+g) = \varepsilon(f) + \varepsilon(g)$ and $\varepsilon(0)=0$
         \item\label{Eax2} $\varepsilon(g \circ f) = \varepsilon(g) \circ f$
          \item\label{Eax3} $\varepsilon(\pi_0) = \pi_0 \circ \varepsilon(1_{A\times B})$ and $ \varepsilon(\pi_1) = \pi_1 \circ \varepsilon(1_{A\times B})$
    \end{enumerate}
\end{defi}

\noindent
By~\ref{Eax1}, $\varepsilon(1_A)$ is an additive map, while by~\ref{Eax2}, it follows that $\varepsilon(f) = \varepsilon(1_B) \circ f$. In light of this, it turns out that infinitesimal extensions can equivalently be described as a class of additive maps $\varepsilon_A: A \to A$.  The equivalence is given by setting $\varepsilon(f) = \varepsilon_B \circ f$ and $\varepsilon_A = \varepsilon(1_A)$.

\begin{lem}\label{lem:epbij} For a Cartesian left additive category $\mathbb{X}$, the following are in bijective correspondence:
\begin{enumerate}[(\roman{enumi}.),ref={\thelem.\roman{enumi}}]
\item An infinitesimal extension $\varepsilon$ on $\mathbb{X}$;
\item A family of maps $\varepsilon_A: A \to A$ indexed by objects $A$, such that $\varepsilon_A$ is additive and $\varepsilon_{A \times B}=\varepsilon_A \times \varepsilon_B$.
\end{enumerate}
Therefore, for any map $f: A \to B$, $\varepsilon(f) = \varepsilon_A \circ f$.
\end{lem}
\begin{proof} Let $\varepsilon$ be an infinitesimal extension. Then for each object $A$, define $\varepsilon_A: A \to A$ as $\varepsilon_A = \varepsilon(1_A)$. Since $\varepsilon$ preserves the additive structure and $\varepsilon(g \circ f) = \varepsilon(g) \circ f$, it follows that:
\begin{align*}
\varepsilon_A \circ (f + g) &=~ \varepsilon(1_A) \circ (f + g) \\
&=~ \varepsilon\left(1_A \circ (f+g) \right) \tag*{\ref{Eax2}} \\
&=~ \varepsilon(f+g) \\
&=~\varepsilon(f) + \varepsilon(g) \tag*{\ref{Eax1}}\\
&=~\varepsilon\left(1_A \circ f \right) + \varepsilon\left(1_A \circ g \right) \\
&=~\left( \varepsilon(1_A) \circ f \right) \circ \left(\varepsilon(1_A) \circ g \right) \tag*{\ref{Eax2}} \\
&=~\left( \varepsilon_A \circ f \right) + \left( \varepsilon_A \circ g \right) \\\\
\varepsilon_A \circ 0 &=~\varepsilon(1_A) \circ 0 \\
&=~ \varepsilon\left(1_A \circ 0 \right) \tag*{\ref{Eax2}} \\
&=~\varepsilon(0) \\
&=~ 0 \tag*{\ref{Eax1}}
\end{align*}
So therefore, $\varepsilon_A$ is additive. Now for a pair of objects $A$ and $B$, we have the following:
\begin{align*}
\pi_0 \circ \varepsilon_{A \times B} &=~\pi_0 \circ \varepsilon(1_{A\times B})\\
&=~ \varepsilon(\pi_0) \tag*{\ref{Eax3}} \\
&=~\varepsilon(1_A \circ \pi_0) \\
&=~ \varepsilon(1_A) \circ \pi_0 \tag*{\ref{Eax2}} \\
&=~ \varepsilon_A \circ \pi_0
\end{align*}
So $\pi_0 \circ \varepsilon_{A \times B} = \varepsilon_A \circ \pi_0$, and similarly, $\pi_1 \circ \varepsilon_{A \times B} = \varepsilon_B \circ \pi_1$. Then by the universal property of the product, it follows that $\varepsilon_{A \times B}=\varepsilon_A \times \varepsilon_B$.

Conversely, suppose that each object $A$ comes equipped with an additive map $\varepsilon_A: A \to A$ such that $\varepsilon_{A \times B}=\varepsilon_A \times \varepsilon_B$. Define $\varepsilon: \mathbb{X}(A,B) \to \mathbb{X}(A,B)$ as $\varepsilon(f) = \varepsilon_B \circ f$. Since $\varepsilon_B$ is additive, it follows that:
\begin{align*}
\varepsilon(f + g) &=~ \varepsilon_B \circ (f + g) =~ \left( \varepsilon_B \circ f \right) + \left( \varepsilon_B \circ g \right) =~ \varepsilon(f) + \varepsilon(g) \end{align*}
\begin{align*}
\varepsilon(0) &=~\varepsilon_B \circ 0 =~0
\end{align*}
So $\varepsilon$ is a monoid morphism, and so~\ref{Eax1} holds. Next, it is straightforward by definition that $\varepsilon(g \circ f) = \varepsilon(g) \circ f$, so~\ref{Eax2} holds. So it remains to show that $\varepsilon$ is compatible with the projections. Note that $\varepsilon_{A \times B}=\varepsilon_A \times \varepsilon_B$ implies that $\pi_0 \circ \varepsilon_{A \times B} = \varepsilon_A \circ \pi_0$ and $\pi_1 \circ \varepsilon_{A \times B} = \varepsilon_B \circ \pi_1$. Therefore, we have that:
\begin{align*}
\varepsilon(\pi_0) &=~ \varepsilon_{A} \circ \pi_0 =~ \pi_0 \circ \varepsilon_{A \times B} \\
\varepsilon(\pi_1) &=~ \varepsilon_{B} \circ \pi_1 =~ \pi_1 \circ \varepsilon_{A \times B}
\end{align*}
So we have that~\ref{Eax3} holds. Therefore, $\varepsilon$ is an infinitesimal extension.

Lastly, we need to show that these constructions are inverses of each other. Starting with an infinitesimal extension $\varepsilon$, we have that:
\begin{align*}
\varepsilon(f) &=~\varepsilon(1_B \circ f) =~ \varepsilon(1_B) \circ f =~\varepsilon_B \circ f
\end{align*}
While in the other direction, it is automatic that $\varepsilon(1_A) = \varepsilon_A$. So we conclude that infinitesimal extensions are bijective correspondence with a family of additive maps $\varepsilon_A: A \to A$ such that $\varepsilon_{A \times B}=\varepsilon_A \times \varepsilon_B$.
\end{proof}

As an immediate consequence of the previous lemma, it follows that infinitesimal extensions are compatible with the product structure.

\begin{lem}\label{lem:ep-pair} Let $\mathbb{X}$ be a Cartesian left additive category with an infinitesimal extension $\varepsilon$. Then $\varepsilon(\pair{f}{g}) = \pair{\varepsilon(f)}{\varepsilon(g)}$ and $\varepsilon(h \times k) = \varepsilon(h) \times \varepsilon(k)$.
\end{lem}
\begin{proof} We compute that:
\begin{align*}
\varepsilon(\pair{f}{g}) &=~ \varepsilon_{A \times B} \circ \pair{f}{g} \tag{Lemma~\ref{lem:epbij}} \\
&=~ (\varepsilon_A \times \varepsilon_B) \circ \pair{f}{g} \tag{Lemma~\ref{lem:epbij}} \\
&=~ \pair{\varepsilon_A \circ f}{\varepsilon_B \circ g} \\
&=~ \pair{\varepsilon(f)}{\varepsilon(g)}  \tag{Lemma~\ref{lem:epbij}}
\end{align*}
So $\varepsilon(\pair{f}{g}) = \pair{\varepsilon(f)}{\varepsilon(g)}$. By similar calculations, we also have that:
\begin{align*}
\varepsilon(h \times k) &=~ \varepsilon_{A \times B} \circ (h \times k)  \tag{Lemma~\ref{lem:epbij}} \\
&=~ (\varepsilon_A \times \varepsilon_B) \circ (h \times k) \tag{Lemma~\ref{lem:epbij}} \\
&=~ (\varepsilon_A \circ h) \times (\varepsilon_B \circ k) \\
&=~ \varepsilon(h) \times \varepsilon(k) \tag{Lemma~\ref{lem:epbij}}
\end{align*}
So $\varepsilon(h \times k) = \varepsilon(h) \times \varepsilon(k)$.
\end{proof}

Infinitesimal extensions equip each object with a canonical change action structure:

\begin{lem}\label{caelem} Let $\mathbb{X}$ be a Cartesian left additive category with infinitesimal extension $\varepsilon$. For every object $A$, define the three maps $\oplus_A: A \times A \to A$,  $+_A: A \times A \to A$, and $0_A: \top \to A$ respectively as follows:
\begin{align*}
\oplus_A = \pi_0 + \varepsilon(\pi_1) && +_A = \pi_0 + \pi_1 && 0_A = 0
\end{align*}
Then $(A, A, \oplus_A, +_A, 0_A)$ is a change action in $\mathbb{X}$.
\end{lem}
\begin{proof} That $(A, +_A, 0_A)$ is a commutative monoid was shown in~\cite[Proposition 1.2.2]{blute2009cartesian}. Thus~\ref{CAax1} holds. It remains to show that $\oplus_A$ is an action, which follows directly from the fact that $\varepsilon$ preserves the additive structure. So we first that compute:
\begin{align*}
    \oplus_A \circ \langle 1_A, 0_A \circ{} !_A \rangle &=~\left( \pi_0 + \varepsilon(\pi_1) \right) \circ \langle 1_A, 0  \rangle \\
    &=~ \pi_0 \circ \langle 1_A, 0_A  \rangle + \varepsilon(\pi_1) \circ \langle 1_A, 0 \rangle \\
   &=~ \pi_0 \circ \langle 1_A, 0_A  \rangle + \varepsilon\left(\pi_1 \circ \langle 1_A, 0 \rangle \right) \tag*{\ref{Eax2}} \\
   &=~ 1_A + \varepsilon(0) \\
   &=~ 1_A + 0 \tag*{\ref{Eax1}} \\
   &=~ 1_A
   \end{align*}
So $  \oplus_A \circ \langle 1_A, 0_A \circ{} !_A \rangle = 1_A$. Next we compute that:
   \begin{align*}
&\oplus_A \circ \pair{\oplus_A \circ (1_A \times \pi_0)}{\pi_1 \circ \pi_1} =~ \left( \pi_0 + \varepsilon(\pi_1) \right) \circ  \pair{\oplus_A \circ (1_A \times \pi_0)}{\pi_1 \circ \pi_1} \\
&=~ \pi_0 \circ \pair{\oplus_A \circ (1_A \times \pi_0)}{\pi_1 \circ \pi_1} + \varepsilon(\pi_1) \circ \pair{\oplus_A \circ (1_A \times \pi_0)}{\pi_1 \circ \pi_1} \\
&=~ \pi_0 \circ \pair{\oplus_A \circ (1_A \times \pi_0)}{\pi_1 \circ \pi_1} + \varepsilon\left(\pi_1 \circ \pair{\oplus_A \circ (1_A \times \pi_0)}{\pi_1 \circ \pi_1} \right)\tag*{\ref{Eax2}} \\
&=~ \oplus_A \circ (1_A \times \pi_0) + \varepsilon\left(\pi_1 \circ \pi_1 \right) \\
&=~ \left( \pi_0 + \varepsilon(\pi_1) \right) \circ (1_A \times \pi_0) + \varepsilon\left(\pi_1 \circ \pi_1 \right) \\
&=~ \pi_0 \circ (1_A \times \pi_0) + \varepsilon(\pi_1) \circ  (1_A \times \pi_0) + \varepsilon\left(\pi_1 \circ \pi_1 \right) \\
&=~ \pi_0 \circ (1_A \times \pi_0) + \varepsilon\left(\pi_1 \circ  (1_A \times \pi_0) \right) + \varepsilon\left(\pi_1 \circ \pi_1 \right) \tag*{\ref{Eax2}} \\
&=~ \pi_0 + \varepsilon(\pi_1 \circ \pi_0) + \varepsilon\left(\pi_1 \circ \pi_1 \right) \\
&=~ \pi_0 + \varepsilon\left( \pi_1 \circ \pi_0 + \pi_1 \circ \pi_1 \right) \tag*{\ref{Eax1}} \\
&=~ \pi_0 + \varepsilon\left( \pi_1 \circ \left( \pi_0 +  \pi_1 \right) \right) \tag{$\pi_1$ is additive} \\
&=~ \pi_0 + \varepsilon\left( \pi_1 \circ +_A \right)  \\
&=~ \pi_0 \circ (1_A \times +_A) + \varepsilon\left( \pi_1 \circ (1_A \times +_A) \right) \\
&=~ \pi_0 \circ (1_A \times +_A) + \varepsilon\left( \pi_1 \right) \circ (1_A \times +_A) \tag*{\ref{Eax2}} \\
&=~ \left( \pi_0 + \varepsilon(\pi_1) \right) \circ (1_A \times +_A) \\
&=~ \oplus_A \circ (1_A \times +_A)
\end{align*}
%\begin{align*}
 %   \oplus_A \circ \pair{f}{0_A} &=~ f + \varepsilon(0) \\
 %   &=~ f + 0 \tag*{\ref{Eax1}} \\
 %   &=~ f
 %   \end{align*}
%Now by setting $f=1_A$, we have that:
%\[\oplus_A \circ \langle 1_A, 0_A \circ !_A \rangle = \oplus_A \circ \langle 1_A, 0 \rangle = 1_A \]
%We can also easily compute that:
 %   \begin{align*}
 %   \oplus_A \circ \pair{f}{g + h} &=~ f + \varepsilon(g + h)\\
 %   &=~ f + \varepsilon(g) + \varepsilon{h} \tag*{\ref{Eax1}} \\
  %  &= (f + \varepsilon(g)) + \varepsilon(h)\\
  %  &=~ \oplus_A \circ \pair{f}{g} + + \varepsilon(h)\\
  %  &= \oplus_A \circ \pair{\oplus_A \circ \pair{f}{g}}{h}
%\end{align*}
%By setting $f= \pi_0$, $g = \pi_0 \circ \pi_1$, and $h = \pi_1 \circ \pi_1$, it follows that:
%\[\oplus_A \circ (1_A \times +_A) = \oplus_A \circ \pair{\pi_0}{\pi_0 \circ \pi_1 + \pi_1 \circ \pi_1} = \oplus_A \circ \pair{\oplus_A \circ \pair{\pi_0}{\pi_0 \circ \pi_1}}{\pi_1 \circ \pi_1} = \oplus_A \circ (\oplus_A \times 1_{A}) \]
So $\oplus_A \circ (1_A \times +_A)  = \oplus_A \circ \pair{\oplus_A \circ (1_A \times \pi_0)}{\pi_1 \circ \pi_1}$. Thus~\ref{CAax2} holds. So we conclude that $(A, A, \oplus_A, +_A, 0_A)$ is a change action. \hfill
\end{proof}

It is important to note that in an arbitrary Cartesian left additive category, $\oplus_A$, $+_A$, and $0_A$ are not necessarily natural transformations. That said, $+_A$ and $0_A$ are natural with respect to additive maps, that is, if $f: A \to B$ is additive then $f \circ 0_A = 0_B$ and $f \circ +_A = +_B \circ (f \times f)$. Setting $\cact{A} \equiv (A, A, \oplus_A, +_A, 0_A)$, we note that $f \oplus_{\cact{A}} g = f + \varepsilon(g)$ and $f +_{\cact{A}} g = f + g$, and so in particular $+_{\cact{A}} = +$. Therefore, from now on we will omit the subscripts and simply write $\oplus$ and $+$.

For every Cartesian left additive category, there are always at least two possible infinitesimal extensions: one given by setting the infinitesimal extension of a map to be zero and another given by setting the infinitesimal extension of a map to be itself.

\begin{lem} For any Cartesian left additive category $\mathbb{X}$,
\begin{enumerate}[(\roman{enumi}).,ref={\thelem.\roman{enumi}}]
 \item Setting $\varepsilon(f) = 0$ defines an infinitesimal extension on $\mathbb{X}$ and therefore in this case, $\oplus_A = \pi_0$ and $f \oplus g = f$.
\item Setting $\varepsilon(f) = f$ defines an infinitesimal extension on $\mathbb{X}$ and therefore in this case, $\oplus_A = +_A$ and $f \oplus g = f + g$.
\end{enumerate}
\end{lem}
\begin{proof} These are both straightforward to check and we leave it as an exercise. \hfill
\end{proof}

We note that while these examples of infinitesimal extensions may seem trivial, they are both very important as they will give rise to key examples of Cartesian difference categories.

\subsection{Cartesian Difference Categories}\label{CdCsec} Here we introduce Cartesian difference categories, the main novel contribution of this paper.

\begin{defi} A \textbf{Cartesian difference category} is a Cartesian left additive category with an infinitesimal extension $\varepsilon$ which is equipped with a \textbf{difference combinator} $\dd$ of the form:
\[ \frac{f : A \to B}{\dd[f]: A \times A \to B} \]
 verifying the following coherence conditions:
 \begin{enumerate}[{\CdCax{\arabic*}},ref={\CdCax{\arabic*}},align=left]
 % \begin{enumerate}[\ref{CdCax1}]
 \setcounter{enumi}{-1}
 \item $f \circ (x + \varepsilon(y)) = f \circ x + \varepsilon\left( \dd[f] \circ \langle x, y \rangle \right)$\label{CdCax0}
\item $\dd[f+g] = \dd[f] + \dd[g]$, $\dd[0] = 0$, and $\dd[\varepsilon(f)] = \varepsilon(\dd[f])$\label{CdCax1}
 \item $\dd[f] \circ \langle x, y + z \rangle = \dd[f] \circ \langle x, y \rangle  + \dd[f] \circ \langle x + \varepsilon(y), z \rangle$ and $\dd[f] \circ \langle x, 0 \rangle = 0$\label{CdCax2}
 \item $\dd[1_A] = \pi_1$ and $\dd[\pi_0] = \pi_0 \circ \pi_1$ and $\dd[\pi_1] = \pi_1 \circ \pi_1$\label{CdCax3}
  \item $\dd[\langle f, g \rangle] = \langle \dd[f], \dd[g] \rangle$\label{CdCax4}
 \item $\dd[g \circ f] = \dd[g] \circ \langle f \circ \pi_0, \dd[f] \rangle$\label{CdCax5}
\item $\dd\left[\dd[f] \right] \circ \left \langle \langle  x, y \rangle, \langle 0, z \rangle \right \rangle= \dd[f] \circ  \langle x + \varepsilon(y), z \rangle$\label{CdCax6}
\item  $ \dd\left[\dd[f] \right] \circ \left \langle \langle x, y \rangle, \langle z, 0 \rangle \right \rangle=  \dd\left[\dd[f] \right] \circ \left \langle \langle x, z \rangle, \langle y, 0 \rangle \right \rangle$\label{CdCax7}
\end{enumerate}
We say that $\dd[f]$ is the derivative of $f$.
\end{defi}

Before giving some intuition on the axioms~\ref{CdCax0} to~\ref{CdCax7}, we first observe that one could have used change action notation $\oplus$ to express~\ref{CdCax0},~\ref{CdCax2}, and~\ref{CdCax6}:
 \begin{description}[leftmargin=*, widest=a]
 \item[\ref{CdCax0}] $f \circ (x \oplus y) = f \circ x \oplus \dd[f] \circ \langle x, y \rangle$
 \item[\ref{CdCax2}]  $\dd[f] \circ \langle x, y + z \rangle = \dd[f] \circ \langle x, y \rangle + \dd[f] \circ \langle x \oplus y, z \rangle$ and $\dd[f] \circ \langle x, 0 \rangle = 0$
 \item[\ref{CdCax6}]  $\dd\left[\dd[f] \right] \circ \left \langle \langle  x, y \rangle, \langle 0, z \rangle \right \rangle= \dd[f] \circ  \langle x \oplus y, z \rangle$
\end{description}
And also, just like Cartesian differential categories,~\ref{CdCax6} and~\ref{CdCax7} have alternative equivalent expressions.
\begin{lem}%
    \label{lem:cdc6a}
    In the presence of the other axioms,~\ref{CdCax6} and~\ref{CdCax7} are equivalent to:
  \begin{enumerate}[\CdCax{\arabic*.{a}},ref={\CdCax{\arabic*.{a}}},align=left]
 \setcounter{enumi}{5}
\item $\dd\left[\dd[f] \right] \circ \left \langle \langle  x, 0 \rangle, \langle 0, y \rangle \right \rangle = \dd[f] \circ  \langle x, y \rangle$\label{CdCax6a}
\item $\dd\left[\dd[f] \right] \circ \left \langle \langle x, y \rangle, \langle z, w \rangle \right \rangle= \dd\left[\dd[f] \right] \circ \left \langle \langle x, z \rangle, \langle y, w \rangle \right \rangle$\label{CdCax7a}
\end{enumerate}
\end{lem}
\begin{proof} The proof is essentially the same as~\cite[Proposition 4.2]{cockett2014differential}. Assume that the other axioms~\ref{CdCax0} to~\ref{CdCax5} hold. Suppose that~\ref{CdCax6} and~\ref{CdCax7} also hold. We first compute~\ref{CdCax6a} using~\ref{CdCax6}:
\begin{align*}
    \dd\left[\dd[f] \right] \circ \left \langle \langle  x, 0 \rangle, \langle 0, y \rangle \right \rangle &=~ \dd[f] \circ  \langle x + \varepsilon(0), y \rangle \tag*{\ref{CdCax6}} \\
    &=~ \dd[f] \circ  \langle x + 0, y \rangle \tag*{\ref{Eax1}} \\
    &=~ \dd[f] \circ  \langle x, y \rangle
\end{align*}
Next we compute~\ref{CdCax7a} using~\ref{CdCax2},~\ref{CdCax6}, and~\ref{CdCax7}:
\begin{align*} &\dd\left[\dd[f] \right] \circ \left \langle \langle x, y \rangle, \langle z, w \rangle \right \rangle =~ \dd\left[\dd[f] \right] \circ \pair{\pair{x}{y}}{\pair{z}{0} + \pair{0}{w}} \\
&=~ \dd\left[\dd[f] \right] \circ \left \langle \langle x, y \rangle, \langle z, 0 \rangle\right \rangle  + \dd\left[\dd[f] \right] \circ \left \langle \langle x, y \rangle + \varepsilon(\pair{z}{0}), \langle 0, w \rangle\right \rangle \tag*{\ref{CdCax2}} \\
&=~ \dd\left[\dd[f] \right] \circ \left \langle \langle x, 0 \rangle, \langle y, 0 \rangle\right \rangle  + \dd\left[\dd[f] \right] \circ \left \langle \langle x, y \rangle + \pair{\varepsilon(z)}{\varepsilon(0)} , \langle 0, w \rangle\right \rangle \tag{Lemma~\ref{lem:ep-pair}} \\
&=~ \dd\left[\dd[f] \right] \circ \left \langle \langle x, 0 \rangle, \langle y, 0 \rangle\right \rangle  + \dd\left[\dd[f] \right] \circ \left \langle \langle x, y \rangle + \pair{\varepsilon(z)}{0} , \langle 0, w \rangle\right \rangle \tag*{\ref{Eax1}} \\
&=~ \dd\left[\dd[f] \right] \circ \left \langle \langle x, 0 \rangle, \langle y, 0 \rangle\right \rangle  + \dd\left[\dd[f] \right] \circ \left \langle \langle x + \varepsilon(z), y \rangle , \langle 0, w \rangle\right \rangle \\
&=~ \dd\left[\dd[f] \right] \circ \left \langle \langle x, z \rangle, \langle y, 0 \rangle\right \rangle  + \dd[f]  \circ \pair{x + \varepsilon(z)  + \varepsilon(y)}{w} \tag*{\ref{CdCax7} +~\ref{CdCax6}} \\
&=~ \dd\left[\dd[f] \right] \circ \left \langle \langle x, z \rangle, \langle y, 0 \rangle\right \rangle  + \dd[f]  \circ \pair{x + \varepsilon(y)  + \varepsilon(z)}{w}  \\
&=~ \dd\left[\dd[f] \right] \circ \left \langle \langle x, z \rangle, \langle y, 0 \rangle\right \rangle  +  \dd\left[\dd[f] \right] \circ \left \langle \langle x + \varepsilon(y), z \rangle, \langle 0, w \rangle\right \rangle \tag*{\ref{CdCax6}} \\
&=~ \dd\left[\dd[f] \right] \circ \left \langle \langle x, z \rangle, \langle y, 0 \rangle\right \rangle  +  \dd\left[\dd[f] \right] \circ \left \langle \langle x + \varepsilon(y), z + 0 \rangle, \langle 0, w \rangle\right \rangle \\
&=~ \dd\left[\dd[f] \right] \circ \left \langle \langle x, z \rangle, \langle y, 0 \rangle\right \rangle  +  \dd\left[\dd[f] \right] \circ \left \langle \langle x + \varepsilon(y), z + \varepsilon(0) \rangle, \langle 0, w \rangle\right \rangle  \tag*{\ref{Eax1}} \\
&=~ \dd\left[\dd[f] \right] \circ \left \langle \langle x, z \rangle, \langle y, 0 \rangle\right \rangle  +  \dd\left[\dd[f] \right] \circ \left \langle \langle x , z  \rangle + \langle \varepsilon(y), \varepsilon(0) \rangle, \langle 0, w \rangle\right \rangle   \\
&=~ \dd\left[\dd[f] \right] \circ \left \langle \langle x, z \rangle, \langle y, 0 \rangle\right \rangle  +  \dd\left[\dd[f] \right] \circ \left \langle \langle x , z \rangle + \varepsilon(\langle y, 0 \rangle), \langle 0, w \rangle\right \rangle  \tag{Lemma~\ref{lem:ep-pair}} \\
&=~ \dd\left[\dd[f] \right] \circ \pair{\pair{x}{z}}{\pair{y}{0} + \pair{0}{w}}  \tag*{\ref{CdCax2}} \\
&=~ \dd\left[\dd[f] \right] \circ \left \langle \langle x, y \rangle, \langle z, w \rangle \right \rangle
    \tag*{\ref{CdCax7}}
\end{align*}
On the other, suppose instead that~\ref{CdCax6a} and~\ref{CdCax7a} hold. Then~\ref{CdCax7} follows immediately from~\ref{CdCax7a} by setting the last term $w=0$.
\begin{align*}
\dd\left[\dd[f] \right] \circ \left \langle \langle x, y \rangle, \langle z, 0 \rangle \right \rangle=  \dd\left[\dd[f] \right] \circ \left \langle \langle x, z \rangle, \langle y, 0 \rangle \right \rangle \tag*{\ref{CdCax7a}}
\end{align*}
Next we compute~\ref{CdCax6} using~\ref{CdCax2},~\ref{CdCax6a}, and~\ref{CdCax7a}.
\begin{align*}
   & \dd\left[\dd[f] \right] \circ \left \langle \langle  x, y \rangle, \langle 0, z \rangle \right \rangle =~     \dd\left[\dd[f] \right] \circ \left \langle \langle  x, 0 \rangle, \langle y, z \rangle \right \rangle \tag*{\ref{CdCax7a}} \\
    &=~ \dd\left[\dd[f] \right] \circ \pair{\pair{x}{0}}{\pair{y}{0} + \pair{0}{z}} \\
&=~ \dd\left[\dd[f] \right] \circ \left \langle \langle x, 0 \rangle, \langle y, 0 \rangle\right \rangle  + \dd\left[\dd[f] \right] \circ \left \langle \langle x, 0 \rangle + \varepsilon(\pair{y}{0}) , \langle 0, z \rangle\right \rangle \tag*{\ref{CdCax2}} \\
&=~ \dd\left[\dd[f] \right] \circ \left \langle \langle x, y \rangle, \langle 0, 0 \rangle\right \rangle  + \dd\left[\dd[f] \right] \circ \left \langle \langle x, 0 \rangle + \varepsilon(\pair{y}{0}) , \langle 0, z \rangle\right \rangle \tag*{\ref{CdCax7a}} \\
&=~ 0  + \dd\left[\dd[f] \right] \circ \left \langle \langle x, 0 \rangle + \varepsilon(\pair{y}{0}) , \langle 0, z \rangle\right \rangle \tag*{\ref{CdCax2}} \\
&=~ \dd\left[\dd[f] \right] \circ \left \langle \langle x, 0 \rangle + \pair{\varepsilon(y)}{\varepsilon(0)} , \langle 0, z \rangle\right \rangle \tag{Lemma~\ref{lem:ep-pair}} \\
&=~ \dd\left[\dd[f] \right] \circ \left \langle \langle x, 0 \rangle + \pair{\varepsilon(y)}{0} , \langle 0, z \rangle\right \rangle \tag*{\ref{Eax1}} \\
&=~ \dd\left[\dd[f] \right] \circ \left \langle \langle x + \varepsilon(y), 0 \rangle , \langle 0, z \rangle\right \rangle \\
&=~ \dd[f] \circ  \langle x + \varepsilon(y), z \rangle \tag*{\ref{CdCax6a}}
\end{align*}
So we conclude that~\ref{CdCax6} and~\ref{CdCax7} are equivalent to~\ref{CdCax6a} and~\ref{CdCax7a}.
\hfill
\end{proof}

The keen eyed reader will notice that the axioms of a Cartesian difference category are very similar to the axioms of a Cartesian differential category. Indeed,~\ref{CdCax1},~\ref{CdCax3},~\ref{CdCax4}, \ref{CdCax5}, and~\ref{CdCax7} are the same as their Cartesian differential category counterpart. The axioms which are different are~\ref{CdCax2} and~\ref{CdCax6} where the infinitesimal extension $\varepsilon$ is now included, and also there is the new extra axiom~\ref{CdCax0}. On the other hand, interestingly enough,~\ref{CDCax6a} is the same as~\ref{CdCax6a}. We also point out that writing out~\ref{CdCax0} and~\ref{CdCax2} using change action notion, we see that these axioms are precisely~\ref{CADax1} and~\ref{CADax2} respectively. To better understand~\ref{CdCax0} to~\ref{CdCax7} it may be useful to write them out using element-like notation. In element-like notation,~\ref{CdCax0} is written as: % chktex 2
\[
    f(x + \varepsilon(y)) = f(x) + \varepsilon\left(\dd[f](x,y) \right)
\]
This condition can be read as a generalization of the Kock-Lawvere axiom that characterizes the derivative from synthetic differential geometry~\cite{kock2006synthetic}. Broadly speaking, the Kock-Lawvere axiom states that, for any map $f : \mathcal{R} \to \mathcal{R}$ and any $x \in \mathcal{R}$ and $d \in \mathcal{D}$, there exists a unique $f'(x) \in \mathcal{R}$ verifying $f(x + d) = f(x) + d \cdot f'(x)$, where $\mathcal{D}$ is the subset of $\mathcal{R}$ consisting of infinitesimal elements. It is by analogy with the Kock-Lawvere axiom that we refer to $\varepsilon$ as an ``infinitesimal extension'' as it can be thought of as embedding the space $A$ into a subspace $\varepsilon(A)$ of infinitesimal elements.

\ref{CdCax1} states that the differential of a sum of maps is the sum of differentials, and similarly for zero maps and the infinitesimal extension of a map. % chktex 2
%-- as we shall see later, it can also be thought of as stating that the maps $+_A, 0_A$ and $\varepsilon_A$ are linear.
\ref{CdCax2} is the first crucial difference between a Cartesian difference category and a Cartesian differential category. In a Cartesian differential category, the differential of a map is assumed to be additive in its second argument. In a Cartesian difference category, just as derivatives for change actions, while the differential is still required to preserve zeros in its second argument, it is only additive ``up to a small perturbation'', that is: % chktex 2
\[\dd[f](x,y+z) = \dd[f](x,y) + \dd[f](x + \varepsilon(y), z)\]
\ref{CdCax3} tells us what the differential of the identity and projection maps are, while~\ref{CdCax4} says that the differential of a pairing of maps is the pairing of their differentials.~\ref{CdCax5} is the chain rule which expresses what the differential of a composition of maps is: % chktex 2
\[ \dd[g \circ f](x,y) = \dd[g](f(x), \dd[f](x,y))\]
The last two axioms,~\ref{CdCax6} and~\ref{CdCax7} tell us how to work with second order differentials.~\ref{CdCax6} is expressed as follows:
\[\dd \left[ \dd[f] \right]\left( x,y, 0,z \right) =  \dd[f](x + \varepsilon(y), z) \]
and finally~\ref{CdCax7} is expressed as:
\[\dd \left[ \dd[f] \right]\left( x,y, z, 0 \right) =  \dd \left[ \dd[f] \right]\left( x, z, y, 0 \right) \]
It is interesting to note that while~\ref{CdCax6} is different from~\ref{CDCax6}, its alternative version~\ref{CdCax6a} is the same as~\ref{CDCax6a}.
\[\dd \left[ \dd[f] \right]\left( (x,0), (0,y) \right) =  \dd[f](x, z) \]
The interplay between~\ref{CdCax0},~\ref{CdCax2}, and~\ref{CdCax6} gives rise to some remarkable and
counter-intuitive consequences.

\begin{lem}%
  \label{lem:d-epsilon} In a Cartesian difference category $\mathbb{X}$, for any map $f : A \to B$, the following equalities hold for any suitable maps:
  \begin{enumerate}[(\roman{enumi}).,ref={\thelem.\roman{enumi}}]
    \item $\dd[f] \circ \pair{x}{\varepsilon(y)} = \varepsilon(\dd[f]) \circ
    \pair{x}{y}$\label{lem:d-epsilon-i}
%    \item
%    $\dd[f] \circ \pair{x}{u + v}
%    = \dd[f] \circ \pair{x}{u} + \dd[f] \circ \pair{x + \varepsilon^2(u)}{v}$
%    \label{lem:d-epsilon-ii}
%    \item $\varepsilon(\dd^2[f]) \circ \four{x}{u}{v}{0}
%    = \varepsilon^2(\dd^2[f]) \circ \four{x}{u}{v}{0}$
%    \label{lem:d-epsilon-iii}
%  \end{enumerate}
%\end{lem}
    \item $\varepsilon(\dd[f]) \circ \pair{x + \varepsilon(y)}{z}
    = \varepsilon(\dd[f]) \circ \pair{x + \varepsilon^2(y)}{z}$\label{lem:d-epsilon-ii}
    \item $\varepsilon^2\left(\dd\left[ \dd[f] \right] \right) \circ \four{x}{y}{z}{0}
    = \varepsilon^3\left(\dd\left[ \dd[f] \right] \right) \circ \four{x}{y}{z}{0}$\label{lem:d-epsilon-iii}
    \end{enumerate}
\end{lem}
\begin{proof} These are mostly straightforward calculations. We start by computing (i):
 \begin{align*}
      \dd[f] \circ \pair{x}{\varepsilon(y)} &=~    \dd[f] \circ \left( \pair{x}{0} + \pair{0}{\varepsilon(y)} \right) \\
      &=~    \dd[f] \circ \left( \pair{x}{0} + \pair{\varepsilon(0)}{\varepsilon(y)} \right) \tag*{\ref{Eax1}} \\
      &=~    \dd[f] \circ \left( \pair{x}{0} + \varepsilon(\pair{0}{y}) \right) \tag{Lemma~\ref{lem:ep-pair}} \\
      &=~ \dd[f]\circ \pair{x}{0} + \varepsilon\left(\dd^2[f] \circ  \four{x}{0}{0}{y} \right)    \tag*{\ref{CdCax0}} \\
      &=~0 + \varepsilon(\dd[f] \circ \pair{x}{y}) \tag*{\ref{CdCax2} +~\ref{CdCax6a}} \\
      &=~ \varepsilon(\dd[f]) \circ \pair{x}{y} \tag*{\ref{Eax2}} \\
    \end{align*}
Next, we use (i) to compute (ii):
 \begin{align*}
    \varepsilon(\dd[f]) \circ \pair{x + \varepsilon(y)}{z} &=~ \varepsilon\left(  \dd[f] \circ \pair{x + \varepsilon(y)}{z}  \right) \tag*{\ref{Eax2}} \\
    &=~ \varepsilon\left( \dd\left[ \dd[f] \right] \circ \four{x}{y}{0}{z} \right)  \tag*{\ref{CdCax6}} \\
        &=~ \varepsilon\left( \dd\left[ \dd[f] \right] \circ \four{x}{0}{y}{z} \right)  \tag*{\ref{CdCax7a}} \\
         &=~ \varepsilon\left( \dd\left[ \dd[f] \right] \right) \circ \four{x}{0}{y}{z} \tag*{\ref{Eax2}} \\
      &=~ \dd\left[ \dd[f] \right] \circ \pair{\pair{x}{0}}{\varepsilon(\pair{y}{z})}  \tag{Lemma~\ref{lem:d-epsilon-i}} \\
      &=~ \dd\left[ \dd[f] \right] \circ \four{x}{0}{\varepsilon(y)}{\varepsilon(z)}  \tag{Lemma~\ref{lem:ep-pair}} \\
      &=~ \dd\left[ \dd[f] \right] \circ \four{x}{\varepsilon(y)}{0}{\varepsilon(z)}   \tag*{\ref{CdCax7a}} \\
&=~ \dd[f] \circ \pair{x + \varepsilon^2(y)}{\varepsilon(z)}    \tag*{\ref{CdCax6}} \\
  &=~ \varepsilon(\dd[f]) \circ \pair{x + \varepsilon^2(y)}{z}
    \tag{Lemma~\ref{lem:d-epsilon-i}}
    \end{align*}
Lastly for (iii), we use (i) and (ii):
 \begin{align*}
\varepsilon^2\left(\dd\left[ \dd[f] \right] \right) \circ \four{x}{y}{z}{0} &=~ \dd\left[ \dd[\varepsilon^2(f)] \right] \circ \four{x}{y}{z}{0} \tag*{\ref{CdCax1}} \\
&=~ \dd\left[ \dd[\varepsilon^2(f)] \right] \circ \four{x}{z}{y}{0} \tag*{\ref{CdCax7a}} \\
&=~ \varepsilon \left( \dd\left[ \dd[\varepsilon(f)] \right] \right) \circ \four{x}{z}{y}{0} \tag*{\ref{CdCax1}} \\
&=~ \dd\left[ \dd[\varepsilon(f)] \right] \circ  \pair{\pair{x}{z}}{\varepsilon(\pair{y}{0})}   \tag{Lemma~\ref{lem:d-epsilon-i}} \\
&=~ \dd\left[ \dd[\varepsilon(f)] \right] \circ  \four{x}{z}{\varepsilon(y)}{\varepsilon(0)}  \tag{Lemma~\ref{lem:ep-pair}} \\
&=~ \dd\left[ \dd[\varepsilon(f)] \right] \circ  \four{x}{z}{\varepsilon(y)}{0}  \tag*{\ref{Eax2}} \\
&=~  \dd\left[ \dd[\varepsilon(f)] \right] \circ  \four{x}{\varepsilon(y)}{z}{0} \tag*{\ref{CdCax7a}} \\
&=~ \varepsilon\left(\dd\left[ \dd[f] \right] \right) \circ \four{x}{\varepsilon(y)}{z}{0} \tag*{\ref{CdCax1}} \\
&=~ \varepsilon\left(\dd\left[ \dd[f] \right] \right) \circ \pair{\pair{x}{0} + \pair{0}{\varepsilon(y)}}{\pair{z}{0}} \\
&=~ \varepsilon\left(\dd\left[ \dd[f] \right] \right) \circ \pair{\pair{x}{0} + \pair{\varepsilon(0)}{\varepsilon(y)}}{\pair{z}{0}} \tag*{\ref{Eax1}} \\
&=~ \varepsilon\left(\dd\left[ \dd[f] \right] \right) \circ \pair{\pair{x}{0} + \varepsilon(\pair{0}{y})}{\pair{z}{0}} \tag{Lemma~\ref{lem:ep-pair}} \\
&=~ \varepsilon\left(\dd\left[ \dd[f] \right] \right) \circ \pair{\pair{x}{0} + \varepsilon^2(\pair{0}{y})}{\pair{z}{0}} \tag{Lemma~\ref{lem:d-epsilon-ii}} \\
&=~  \varepsilon\left(\dd\left[ \dd[f] \right] \right) \circ \pair{\pair{x}{0} + \pair{\varepsilon^2(0)}{\varepsilon^2(y)}}{\pair{z}{0}}  \tag{Lemma~\ref{lem:ep-pair}} \\
&=~ \varepsilon\left(\dd\left[ \dd[f] \right] \right) \circ \pair{\pair{x}{0} + \pair{0}{\varepsilon^2(y)}}{\pair{z}{0}}\tag*{\ref{Eax1}} \\
&=~ \varepsilon\left(\dd\left[ \dd[f] \right] \right) \circ \four{x}{\varepsilon^2(y)}{z}{0} \\
&=~ \dd\left[ \dd[\varepsilon(f)] \right] \circ \four{x}{\varepsilon^2(y)}{z}{0} \tag*{\ref{CdCax1}} \\
&=~ \dd\left[ \dd[\varepsilon(f)] \right] \circ \four{x}{z}{\varepsilon^2(y)}{0} \tag*{\ref{CdCax7a}} \\
&=~ \dd\left[ \dd[\varepsilon(f)] \right] \circ \four{x}{z}{\varepsilon^2(y)}{\varepsilon^2(0)} \tag*{\ref{Eax1}} \\
&=~ \dd\left[ \dd[\varepsilon(f)] \right] \circ  \pair{\pair{x}{z}}{\varepsilon^2(\pair{y}{0})} \tag{Lemma~\ref{lem:ep-pair}} \\
&=~ \varepsilon^2\left(\dd\left[ \dd[\varepsilon(f)] \right] \right) \circ \four{x}{z}{y}{0} \tag{Lemma~\ref{lem:d-epsilon-i}} \\
&=~ \varepsilon^2\left(\dd\left[ \dd[\varepsilon(f)] \right] \circ \four{x}{z}{y}{0} \right) \tag*{\ref{Eax2}} \\
&=~ \varepsilon^2\left(\dd\left[ \dd[\varepsilon(f)] \right] \circ \four{x}{y}{z}{0} \right) \tag*{\ref{CdCax7a}} \\
&=~ \varepsilon^2\left(\dd\left[ \dd[\varepsilon(f)] \right] \right) \circ \four{x}{y}{z}{0} \tag*{\ref{Eax2}} \\
&=~ \varepsilon^3\left(\dd\left[ \dd[f] \right] \right) \circ \four{x}{y}{z}{0} \tag*{\ref{CdCax1} \qedhere}
    \end{align*}
\end{proof}

A desirable identity we would like to hold is a slightly stronger version of Lemma~\ref{lem:d-epsilon-iii}:
\begin{equation}\label{lem:d-epsilon-iii'}\begin{gathered} \varepsilon(\dd\left[ \dd[f] \right]) \circ \four{x}{y}{z}{0}
        = \varepsilon^2(\dd\left[ \dd[f] \right]) \circ \four{x}{y}{z}{0}
 \end{gathered}\end{equation}
Note that if $\varepsilon$ is injective (i.e $\varepsilon(f) = \varepsilon(g)$ implies that $f=g$) or if $\varepsilon$ is idempotent (i.e. $\varepsilon^2(f)=\varepsilon(f)$), then (\ref{lem:d-epsilon-iii'}) holds. In fact, in all of our examples of Cartesian difference categories in Section~\ref{EXsec}, (\ref{lem:d-epsilon-iii'}) holds (even for those where $\varepsilon$ is neither injective or idempotent). It is currently unclear if (\ref{lem:d-epsilon-iii'}) holds in any Cartesian difference category, as we do not have a counterexample. If (\ref{lem:d-epsilon-iii'}) holds and the infinitesimal extension is nilpotent, then it turns out that~\ref{CdCax2} is the same as~\ref{CDCax2}.

\begin{lem}\label{epsilon-nilponent}
  Let $\mathbb{X}$ be a Cartesian difference category with a nilpotent
  infinitesimal extension, that is, for every map $f: A \to B$ there is some $k \in \mathbb{N}$ such that
  $\varepsilon^k(f) = 0$, and suppose that (\ref{lem:d-epsilon-iii'}) holds.
Then~\ref{CDCax2} holds, or in other words, every derivative $\dd[f]$ is additive in its second argument, that is, $\dd[f] \circ \langle x, y + z \rangle = \dd[f] \circ \langle x, y \rangle  + \dd[f] \circ \langle x, z \rangle$ and $\dd[f] \circ \langle x, 0 \rangle = 0$.
\end{lem}
\begin{proof} By~\ref{CdCax2}, it already holds that $\dd[f] \circ \langle x, 0 \rangle = 0$. Therefore it remains to show that $\dd[f] \circ \langle x, y + z \rangle = \dd[f] \circ \langle x, y \rangle  + \dd[f] \circ \langle x, z \rangle$.  So suppose that $\varepsilon^k(f) = 0$ for some $k \in \mathbb{N}$. Then using Lemma~\ref{lem:d-epsilon-iii'}, we compute that:
  \begin{align*}
&\dd[f] \circ \langle x, y + z \rangle =~\dd[f] \circ \langle x, y \rangle  + \dd[f] \circ \langle x + \varepsilon(y), z \rangle  \tag*{\ref{CdCax2}} \\
&=~ \dd[f] \circ \langle x, y \rangle  + \dd[f] \circ \left( \langle x, z \rangle + \langle \varepsilon(y), 0 \rangle \right)      \\
&=~ \dd[f] \circ \langle x, y \rangle  + \dd[f] \circ \left( \langle x, z \rangle + \langle \varepsilon(y), \varepsilon(0) \rangle \right)      \tag*{\ref{Eax1}} \\
&=~ \dd[f] \circ \langle x, y \rangle  + \dd[f] \circ \left( \langle x, z \rangle + \varepsilon(\langle y, 0 \rangle) \right)  \tag{Lemma~\ref{lem:ep-pair}} \\
&=~ \dd[f] \circ \langle x, y \rangle  + \dd[f] \circ \langle x , z \rangle + \varepsilon \left( \dd\left[ \dd[f] \right] \circ \four{x}{z}{y}{0} \right)  \tag*{\ref{CdCax0}} \\
&=~ \dd[f] \circ \langle x, y \rangle  + \dd[f] \circ \langle x , z \rangle + \varepsilon \left( \dd\left[ \dd[f] \right] \right)  \circ \four{x}{z}{y}{0}  \tag*{\ref{Eax2}} \\
&=~ \dd[f] \circ \langle x, y \rangle  + \dd[f] \circ \langle x , z \rangle + \varepsilon^k \left( \dd\left[ \dd[f] \right] \right) \circ \four{x}{z}{y}{0}  \tag{Iterating (\ref{lem:d-epsilon-iii'})} \\
&=~ \dd[f] \circ \langle x, y \rangle  + \dd[f] \circ \langle x , z \rangle +  \dd\left[ \dd[\varepsilon^k(f)] \right]   \circ \four{x}{z}{y}{0}  \tag*{\ref{CdCax1}} \\
&=~ \dd[f] \circ \langle x, y \rangle  + \dd[f] \circ \langle x , z \rangle + \dd\left[ \dd[0] \right]  \circ \four{x}{z}{y}{0} \tag{$\varepsilon$ is nilpotent} \\
&=~ \dd[f] \circ \langle x, y \rangle  + \dd[f] \circ \langle x , z \rangle + 0  \circ \four{x}{z}{y}{0}  \tag*{\ref{CdCax1}} \\
&=~ \dd[f] \circ \langle x, y \rangle  + \dd[f] \circ \langle x , z \rangle + 0 \\
&=~ \dd[f] \circ \langle x, y \rangle  + \dd[f] \circ \langle x , z \rangle
\end{align*}
We conclude that $\dd[f]$ is additive in its second argument, and so~\ref{CDCax2} holds.
\end{proof}

\subsection{Another look at Cartesian Differential Categories}\label{CDCisCdCsec}

Here we explain how every Cartesian differential category is a Cartesian difference category where the infinitesimal extension is given by zero.

\begin{prop}\label{CDtoCd} Every Cartesian differential category $\mathbb{X}$ with differential combinator $\mathsf{D}$ is a Cartesian difference category where the infinitesimal extension is defined as $\varepsilon(f) = 0$ and the difference combinator is defined to be the differential combinator, $\dd = \mathsf{D}$.
\end{prop}
\begin{proof} As noted before, the first two parts of the~\ref{CdCax1}, the second part of~\ref{CdCax2},~\ref{CdCax3}, \ref{CdCax4},~\ref{CdCax5}, and~\ref{CdCax7} are precisely the same as their Cartesian differential axiom counterparts. On the other hand, since $\varepsilon(f) =0$,~\ref{CdCax0} and the third part of~\ref{CdCax1} trivially state that $0=0$, while the first part of~\ref{CdCax2} and~\ref{CdCax6} end up being precisely the first part of~\ref{CDCax2} and~\ref{CDCax6}. Therefore, the differential combinator satisfies the Cartesian difference axioms and we conclude that a Cartesian differential category is a Cartesian difference category. % chktex 2
\hfill
\end{proof}

Conversely, one can always build a Cartesian differential category from a Cartesian difference category by considering the objects for which the infinitesimal extension of the identity map is the zero map. We call such objects $\varepsilon$-vanishing.

\begin{defi}\label{def:ep-vanishing} In a Cartesian left additive category $\mathbb{X}$ with infinitesimal extension $\varepsilon$, an object $A$ is said to be \textbf{$\varepsilon$-vanishing} if $\varepsilon_A = \varepsilon(1_A) = 0$. We denote $\mathbb{X}_{\varepsilon\text{-van}}$ to be the full subcategory of $\varepsilon$-vanishing objects of $\mathbb{X}$.
\end{defi}

The full subcategory of $\varepsilon$-vanishing objects always forms a Cartesian left additive category.

\begin{lem}\label{lem:ep-van-prod} In a Cartesian left additive category $\mathbb{X}$ with infinitesimal extension $\varepsilon$,
 \begin{enumerate}[(\roman{enumi}).,ref={\thelem.\roman{enumi}}]
\item The terminal object $\top$ is $\varepsilon$-vanishing;
\item If $A$ and $B$ are $\varepsilon$-vanishing then their product $A \times B$ is $\varepsilon$-vanishing;
  \end{enumerate}
Therefore, $\mathbb{X}_{\varepsilon\text{-van}}$ is a Cartesian left additive category with the same Cartesian left additive structure as $\mathbb{X}$. Furthermore:
 \begin{enumerate}[(\roman{enumi}).,ref={\thelem.\roman{enumi}}]
 \setcounter{enumi}{2}
\item If $B$ is $\varepsilon$-vanishing, then for any map $f: A \to B$, $\varepsilon(f) = 0$.
  \end{enumerate}
\end{lem}
\begin{proof} For (i), recall that for the terminal object, $1_\top = 0$. Therefore by~\ref{Eax1} we easily see that $\varepsilon(1_\top) = \varepsilon(0) = 0$. So the terminal object $\top$ is $\varepsilon$-vanishing. For (ii), first recall that by Lemma~\ref{lem:epbij}, for any pair of objects $A$ and $B$, $\varepsilon_{A \times B} = \varepsilon_A \times \varepsilon_B$. Therefore, if both $A$ and $B$ are $\varepsilon$-vanishing, we have that $\varepsilon_{A \times B} = \varepsilon_A \times \varepsilon_B = 0 \times 0 = 0$. So there product of $\varepsilon$-vanishing objects is again $\varepsilon$-vanishing. Thus $\mathbb{X}_{\varepsilon\text{-van}}$ is closed under finite products and therefore we conclude that $\mathbb{X}_{\varepsilon\text{-van}}$ is also a Cartesian left additive category. For (iii), suppose that $B$ is $\varepsilon$-vanishing. Then by~\ref{Eax2}, we have that $\varepsilon(f) = \varepsilon(1_B) \circ f = 0 \circ f = 0$. \end{proof}

For a Cartesian difference category, its subcategory of $\varepsilon$-vanishing objects is a Cartesian differential category.

\begin{prop}\label{CdtoCD} For a Cartesian difference category $\mathbb{X}$ with infinitesimal extension $\varepsilon$ and difference combinator $\dd$, then $\mathbb{X}_{\varepsilon\text{-van}}$ is a Cartesian differential category where the differential combinator is defined to be the difference combinator, $\mathsf{D} = \dd$.
\end{prop}
\begin{proof} In Lemma~\ref{lem:ep-van-prod} we already explained why $\mathbb{X}_{\varepsilon\text{-van}}$ is a Cartesian left additive category. So it remains to explain why $\dd$ is a differential combinator. However, by Lemma~\ref{lem:ep-van-prod}.iii, every map in $\mathbb{X}_{\varepsilon\text{-van}}$ satisfies $\varepsilon(f) =0$. Therefore, for maps between objects in $\mathbb{X}_{\varepsilon\text{-van}}$, the Cartesian difference axioms are precisely the Cartesian differential axioms. So we conclude that the difference combinator is a differential combinator for this subcategory, and thus $\mathbb{X}_{\varepsilon\text{-van}}$ is a Cartesian differential category. \hfill
\end{proof}

Note that by Lemma~\ref{lem:ep-van-prod}.i, for any Cartesian difference category $\mathbb{X}$, the terminal object $\top$ is always $\varepsilon$-vanishing, and so therefore, $\mathbb{X}_{\varepsilon\text{-van}}$ is never empty. On the other hand, applying Proposition~\ref{CdtoCD} to a Cartesian differential category results in the entire category, since every object is $\varepsilon$-vanishing by definition. It is also important to note that the above two propositions do not imply that if a difference combinator is a differential combinator then the infinitesimal extension must be zero. In Section~\ref{moduleex}, we provide such an example of a Cartesian differential category that comes equipped with a non-zero infinitesimal extension such that the differential combinator is a difference combinator with respect to this non-zero infinitesimal extension.
%It is also important to mention that Proposition~\ref{CDtoCd} and Proposition~\ref{CdtoCD} result in an adjunction between the category of Cartesian difference categories and the category of Cartesian differential categories. This will be explained in greater detail in an upcoming journal version of this paper.

\subsection{Cartesian Difference Categories as Change Action Models}\label{CdCisCAsec}

In this section, we show how every Cartesian difference category is a particularly well-behaved change action model, and conversely how every change action model contains a Cartesian difference category.

\begin{prop}\label{CdtoCA} Let $\mathbb{X}$ be a Cartesian difference category with infinitesimal extension $\varepsilon$ and difference combinator $\dd$. Define the functor $\alpha: \mathbb{X} \to \mathsf{CAct}(\mathbb{X})$ on objects as $\alpha(A) = (A, A, \oplus_A, +_A, 0_A)$ (as defined in Lemma~\ref{caelem}) and on maps as $\alpha(f) = (f, \dd[f])$. Then $(\mathbb{X}, \alpha : \mathbb{X} \to \mathsf{CAct}(\mathbb{X}))$ is a change action model.
\end{prop}
\begin{proof} By Lemma~\ref{caelem}, $(A, A, \oplus_A, +_A, 0_A)$ is a change action and so $\alpha$ is well-defined on objects. While for a map $f$, $\dd[f]$ is a derivative of $f$ in the change action sense since~\ref{CdCax0} and~\ref{CdCax2} are precisely~\ref{CADax1} and~\ref{CADax2}, and so $\alpha$ is well-defined on maps.
Next we show that $\alpha$ preserves identities and composition, which  follows from~\ref{CdCax3} and~\ref{CdCax5} respectively:
\begin{align*}
\alpha(1_A) &=~ (1_A, \dd[1_A]) \\
&=~ (1_A, \pi_1) \tag*{\ref{CdCax3}} \\ \\
\alpha(g \circ f) &=~(g \circ f, \dd[g \circ f]) \\
&=~ \left( g \circ f, \dd[g] \circ \langle f \circ \pi_0, \dd[f] \rangle \right) \tag*{\ref{CdCax5}} \\
&=~(g, \dd[g]) \circ (f, \dd[g]) \tag{by def. of comp. in $\mathsf{CAct}(\mathbb{X})$} \\ % chktex 36
&=~\alpha(g) \circ \alpha(f)
\end{align*}
So $\alpha$ is a functor. Next we check that $\alpha$ preserves projections and pairings, which will follow from~\ref{CdCax3} and~\ref{CdCax4}:
\begin{align*}
\alpha(\pi_i) &=~ (\pi_i, \dd[\pi_i]) \\
&=~ (\pi_i, \pi_i \circ \pi_1) \tag*{\ref{CdCax3}} \\ \\
\alpha(\pair{f}{g}) &=~ (\pair{f}{g}, \dd[\pair{f}{g}]) \\
&=~ (\pair{f}{g}, \pair{\dd[f]}{\dd[g]} ) \tag*{\ref{CdCax4}} \\
&=~ \pair{(f,\dd[f])}{(g, \dd[g])} \\
&=~ \pair{\alpha(f)}{\alpha(g)}
\end{align*}
So $\alpha$ preserves finite products. Lastly, it is clear that $\alpha$ is a section of the forgetful functor, and therefore we conclude that $(\mathbb{X}, \alpha)$ is a change action model.
\hfill
\end{proof}

Not every change action model is a Cartesian difference category. For example, change action models do not require the addition to be commutative. On the other hand, it can be shown that every change action model contains a Cartesian difference category as a full subcategory. We call the objects and maps in this subcategory \emph{flat}. It is important to note that the definition here is a slight generalisation of the one in~\cite{alvarez2020cartesian}. The original suggested definition implied that $\oplus = +$ and as a result $\varepsilon(f) = f$. While this indeed resulted in a Cartesian difference category (specifically one that generalized the Abelian group example from Section~\ref{discreteex}), we found that it was a bit too restrictive of a construction. We correct this by generalizing the definition slightly, while still keeping the construction's overall core ideas the same.

\begin{defi}\label{def:flat}
    Let $(\mathbb{X}, \alpha : \mathbb{X} \to \mathsf{CAct}(\mathbb{X}))$ be a change action model. An object $A$ is \textbf{flat} whenever the following equations hold:
    \begin{enumerate}[{\bf [F.1]},ref={{\bf [F.{\arabic*}]}},align=left]
    \item $\Delta A = A$\label{F1}
    \item $\alpha(\oplus_A) = (\oplus_A, \oplus_A \circ \pi_1)$ and $\alpha(+_A) = (+_A, +_A \circ \pi_1)$\label{F2}
    %\item $0 \oplus_A (0 \oplus_A f) = 0 \oplus_A f$ for any $f : U \to A$
    %    \label{F3}
    \item $\oplus_A$ is right-injective, that is, if $\oplus_A \circ \langle f, g \rangle = \oplus_A \circ \langle f, h \rangle$ then $g=h$.\label{F4}
    \end{enumerate}
Let $f : A \to B$ be a map between flat objects in the above change action model, with $\alpha(f) = (f, \dd[f])$. Then $f$ is \textbf{flat}
    whenever the following condition holds:
    \begin{enumerate}[{\bf [F.1]},ref={{\bf [F.{\arabic*}]}},align=left]
        \setcounter{enumi}{3}
        \item $0_B \oplus_B \partial[f] \circ \pair{x}{y} = \partial[f] \circ \pair{x}{0_A \oplus_A y}$\label{Fd}
    \end{enumerate}
    We define $\mathsf{Flat}_\alpha$ to be the subcategory of $\mathbb{X}$ containing all flat objects and all flat morphisms
    between them.
\end{defi}
Intuitively, the condition~\ref{F1} simply states that a flat object is similar to an Euclidean space
in the sense that its ``tangent space'' is equal to itself.~\ref{F2} can be read as stating that
the action $\oplus_A$ and the addition of changes $+_A$ are linear. In this sense, flat objects behave like Euclidean spaces.
%\ref{F3} says that $0 \oplus_A -$ is idempotent, which will result in an idempotent infinitesimal extension.
\ref{F4} tells us that we may cancel the left argument of $\oplus_A$, which will imply that the result infinitesimal extension is injective. This cancellation ability that~\ref{F4} provides will be used throughout most of the proofs in this section. Lastly, as we will see below,~\ref{Fd} is a generalization of Lemma~\ref{lem:d-epsilon-i}. % chktex 2

We would like to show that for any change action model $(\mathbb{X}, \alpha)$, $\mathsf{Flat}_\alpha$ is a Cartesian difference category. We first explain the finite product structure of $\mathsf{Flat}_\alpha$. % First we remark that every linear map (i.e. those maps $f$ satisfying $f = \dd[f] \circ \pi_1$) is flat.

\begin{lem}  Let $(\mathbb{X}, \alpha : \mathbb{X} \to \mathsf{CAct}(\mathbb{X}))$ be a change action model. Then:
 \begin{enumerate}[(\roman{enumi}).,ref={\thelem.\roman{enumi}}]
\item The terminal object $\top$ is flat;
\item If $A$ and $B$ are flat then their product $A \times B$ is flat;
\item If $A$ is a flat object, then the identity $1_A: A \to A$ is flat;
\item If $A$, $B$, and $C$ are flat objects, and $f: A \to B$ and $g: B \to C$ are flat maps, then their composite $g \circ f: A \to C$ is flat;
%\item If $A$ is a flat object, then the action $\oplus_A: A \times A \to A$ and addition $+_A: A \times A \to A$ are flat;
%\item If $A$ and $B$ are flat objects, then the zero map $0: A \to B$ is flat;
\item If $A$ and $B$ are flat objects, then the projection maps $\pi_0: A \times B \to A$ and ${\pi_1: A \times B \to B}$ are flat;
\item If $A$, $B$, and $C$ are flat objects, and $f: C \to A$ and $g: C \to B$ are flat maps, then their pairing $\langle f, g \rangle: C \to A \times B$ is flat.
  \end{enumerate}
Furthermore, $\mathsf{Flat}_\alpha$ is a Cartesian category with the same Cartesian structure as $\mathbb{X}$.
\end{lem}
\begin{proof} These are straightforward and we leave it as an exercise for the reader to check for themselves. (i), (ii), (v), and (vi)  follows from the fact that $\alpha$ preserves finite products, while (iii) and (v) follows from the functoriality of $\alpha$.
\end{proof}

As an immediate consequence, we note that for any change action model $(\mathbb{X}, \alpha)$, since the terminal object is always flat, $\mathsf{Flat}_\alpha$ is never empty. Next we discuss the additive structure of $\mathsf{Flat}_\alpha$. The sum of maps $f: A \to B$ and $g: A \to B$ in $\mathsf{Flat}_\alpha$ is defined using the change action structure $f+g := f +_B g$, while the zero map $0: A \to B$ is $0 := 0_B \circ{} !_A$. And so we obtain that:

\begin{lem}\label{EUCLCLAC} Let $(\mathbb{X}, \alpha : \mathbb{X} \to \mathsf{CAct}(\mathbb{X}))$ be a change action model. Then:
\begin{enumerate}[(\roman{enumi}).,ref={\thelem.\roman{enumi}}]
\item If $A$ and $B$ are flat objects, then the zero map $0: A \to B$ is flat;
\item If $A$ and $B$ are flat objects, and $f: A \to B$ and $g: A \to B$ are flat morphisms, then their sum $f + g: A \to B$ is flat.
  \end{enumerate}
Furthermore, $\mathsf{Flat}_\alpha$ is a Cartesian left additive category.
\end{lem}
\begin{proof} Most of the Cartesian left additive structure is straightforward. However, since the addition is not required to be commutative for arbitrary change actions, we will show that the addition is commutative for flat objects. Using that $\oplus_B$ is an action as in Lemma~\ref{lem:CAadd}, that by~\ref{F2} we have that $\oplus_B \circ \pi_1$ is a derivative for $\oplus_B$, and~\ref{CADax1}, we obtain that:
    \begin{align*}
        0_B \oplus_B (f + g)
        &= (0_B \oplus_B f) \oplus_B g
         \tag{Lemma~\ref{lem:CAadd.+}} \\
             &= (0_B \oplus_B f) \oplus_B (g \oplus_B 0)
          \tag{Lemma~\ref{lem:CAadd.0}} \\
    &= \oplus_B \circ \pair{0_B \oplus_B f}{g \oplus_B 0} \\
    &= \oplus_B \circ \left( \pair{0_B}{g} \oplus_B \pair{f}{0} \right) \tag{Lemma~\ref{lem:CACTprod.oplus}} \\
    &= \left( \oplus_B \circ \pair{0_B}{g} \right) \oplus_B \left(\dd[\oplus_B]\four{0_B}{g}{f}{0} \right)   \tag*{\ref{CADax1}} \\
        &= (0_B \oplus_B g) \oplus_B \dd[\oplus_B]\four{0_B}{g}{f}{0} \\
        &= (0_B \oplus_B g) \oplus_B \left( \oplus_B \circ \pi_1 \circ \four{0_B}{g}{f}{0} \right)  \tag*{\ref{F2}} \\
              &= (0_B \oplus_B g) \oplus_B \left( \oplus_B \circ \pair{f}{0} \right) \\
        &= (0_B \oplus_B g) \oplus_B (f \oplus_B 0)
      \\
        &= (0_B \oplus_B g) \oplus_B f
     \tag{Lemma~\ref{lem:CAadd.0}} \\
        &= 0_B \oplus_B (g + f)
     \tag{Lemma~\ref{lem:CAadd.+}}
    \end{align*}
    By~\ref{F4}, $\oplus_B$ is right-injective and we conclude that $f + g = g + f$. \hfill
\end{proof}

We use the action of the change action structure to define the infinitesimal extension. So for a map $f: A \to B$ in $\mathsf{Flat}_\alpha$, define $\varepsilon(f): A \to B$ as follows:
\[ \varepsilon(f) = \oplus_B \circ \pair{0_B \circ{} !_A}{f} = 0 \oplus_B f \]
As such, we may rewrite~\ref{Fd} as follows:
\[ \dd[f] \circ \pair{x}{\varepsilon(y)} = \varepsilon(\dd[f]) \circ
    \pair{x}{y} \]
which is Lemma~\ref{lem:d-epsilon-i}.

\begin{lem} $\varepsilon$ is an infinitesimal extension for $\mathsf{Flat}_\alpha$.
\end{lem}
\begin{proof} We show that $\varepsilon$ satisfies~\ref{Eax1},~\ref{Eax2},~\ref{Eax3}. Starting with~\ref{Eax1}, we compute that:
\begin{align*}
\varepsilon(0) &=~ 0 \oplus_B 0 \\
&=~ 0   \tag{Lemma~\ref{lem:CAadd.0}}
\end{align*}
So $\varepsilon(0)= 0$. Next following the same idea as in the proof of Lemma~\ref{EUCLCLAC}, we obtain the following:
    \begin{align*}
0_B \oplus_B \varepsilon(f + g) &= 0_B \oplus_B (0 \oplus_B (f + g))\\
&= 0_B \oplus_B ((0 \oplus_B f) \oplus_B g)  \tag{Lemma~\ref{lem:CAadd.+}} \\
&= 0_B \oplus_B (\varepsilon(f) \oplus_B g)  \tag{Lemma~\ref{lem:CAadd.+}} \\
        &= (0_B \oplus_B 0) \oplus_B (\varepsilon(f) \oplus_B g)  \tag{Lemma~\ref{lem:CAadd.0}} \\
        &= \oplus_B \circ \pair{0_B \oplus_B 0}{\varepsilon(f) \oplus_B g} \\
        &= \oplus_B \circ \left( \pair{0_B}{\varepsilon(f)} \oplus_B \pair{0}{g} \right)  \tag{Lemma~\ref{lem:CACTprod.oplus}} \\
    &= \left( \oplus_B \circ  \pair{0_B}{\varepsilon(f)} \right) \oplus_B \left( \partial[\oplus_B] \circ \four{0_B}{\varepsilon(f)}{0}{g} \right)   \tag*{\ref{CADax1}} \\
    &= \left( 0_B \oplus \varepsilon(f) \right) \oplus_B  \left( \partial[\oplus_B] \circ \four{0_B}{\varepsilon(f)}{0}{g} \right) \\
    &=  \left( 0_B \oplus \varepsilon(f) \right) \oplus_B \left( \oplus_B \circ \pi_1 \circ \four{0_B}{\varepsilon(f)}{0}{g} \right)  \tag*{\ref{F2}}\\
        &=  \left( 0_B \oplus \varepsilon(f) \right) \oplus_B \left( \oplus_B \circ \pair{0}{g} \right)  \tag*{\ref{F2}}\\
&=     \left( 0_B \oplus \varepsilon(f) \right) \oplus_B (0 \oplus_B g) \\
        &= (0_B \oplus_B \varepsilon(f)) \oplus_B \varepsilon(g)\\
        &= 0_B \oplus_B (\varepsilon(f) +_B \varepsilon(g))  \tag{Lemma~\ref{lem:CAadd.+}}
    \end{align*}
    Then by~\ref{F4}, it follows that $\varepsilon(f + g) = \varepsilon(f) + \varepsilon(g)$. So~\ref{Eax1} holds. Next, it is easy to show that $\varepsilon$ is compatible with composition:
    \begin{align*}
        \varepsilon(g \circ f) &= 0 \oplus_C (g \circ f) \\
        &= (0 \circ f) \oplus_C (g \circ f) \\
        &= (0 \oplus_C g) \circ f \tag{Lemma~\ref{lem:CAadd.circ}} \\
        &= \varepsilon(g) \circ f
    \end{align*}
So~\ref{Eax2} holds. Finally, for the projections we compute:
    \begin{align*}
        \varepsilon(\pi_i) &= 0 \oplus \pi_i \\
        &= \oplus_A \circ \pair{0_A \circ !_{A \times B}}{\pi_i}\\
         &= \oplus_A \circ \pair{\pi_i \circ 0_{A \times B} \circ !_{A \times B}}{\pi_i} \tag{Lemma~\ref{lem:CACTprod.0}} \\
        &= \oplus_A \circ (\pi_i \times \pi_i) \circ \pair{0_{A \times B} \circ !_{A \times B}}{1_{A \times B}}\\
        &= \pi_i \circ \oplus_{A \times B} \circ \pair{0_{A \times B} \circ !_{A \times B}}{1_{A \times B}}\\
        &= \pi_i \circ \left( 0 \oplus 1_{A \times B} \right) \\
        &= \pi_i \circ \varepsilon(1_{A \times B})
    \end{align*}
So~\ref{Eax3} holds. Thus we conclude that $\varepsilon$ is an infinitesimal extension. \hfill
\end{proof}

Next, we observe that for a flat object, the action can be expressed in terms of the infinitesimal extension.

\begin{lem}%
\label{oplus-epsilon} For any maps $f, g : A \to B$ in $\mathsf{Flat}_\alpha$, the following equality holds:
  \begin{align*}
    f \oplus_B g = f + \varepsilon(g)
  \end{align*}
\end{lem}
\begin{proof}
    The proof is a straightforward consequence of linearity of addition:
    \begin{align*}
       f + \varepsilon(g) &= f + (0 \oplus_B g) \\
       &= (f \oplus_B 0) + (0 \oplus_B g)
\tag{Lemma~\ref{lem:CAadd.0}}\\
&= +_B \circ \left( \pair{f \oplus_B 0}{0 \oplus_B g} \right) \\
&= +_B \circ \left( \pair{f}{0} \oplus_{B \times B} \pair{0}{g} \right) \tag{Lemma~\ref{lem:CACTprod.oplus}} \\
&= \left(+_B \circ \pair{f}{0} \right) \oplus_B \left( \dd[+_B] \circ \four{f}{0}{0}{g} \right)
        \tag*{\ref{CADax1}} \\
        &= (f + 0) \oplus_B \left( \dd[+_B] \circ \four{f}{0}{0}{g} \right)
        \\
            &= f \oplus_B \left( \dd[+_B] \circ \four{f}{0}{0}{g} \right) \\
        &= f \oplus_B \left( +_B \circ \pi_1 \circ \four{f}{0}{0}{g} \right)
        \tag*{\ref{F2}}
        \\
              &= f \oplus_B \left( +_B \circ \pair{0}{g} \right)
        \\
        &= f \oplus_B (0 + g) \\
        &= f \oplus_B g
    \end{align*}
Thus we conclude that the desired equality holds. \end{proof}

The difference combinator for $\mathsf{Flat}_\alpha$ is defined in the obvious way, that is, $\dd[f]$ is defined as the second component of $\alpha(f)$.

\begin{prop}\label{CAtoCd} Let $(\mathbb{X}, \alpha : \mathbb{X} \to \mathsf{CAct}(\mathbb{X}))$ be a change action model. Then $\mathsf{Flat}_\alpha$ is a Cartesian difference category.
\end{prop}
\begin{proof}
\ref{CdCax0} and~\ref{CdCax2} are simply a restatement of~\ref{CADax1} and~\ref{CADax2}.~\ref{CdCax3} and~\ref{CdCax4} follow immediately from the fact that % chktex 2
  $\alpha$ preserves finite products and from the structure of products in
  $\mathsf{CAct}(\mathbb{X})$, while~\ref{CdCax5}
  follows from the definition of composition in $\mathsf{CAct}(\mathbb{X})$. So it remains to show that~\ref{CdCax1},~\ref{CdCax6}, and~\ref{CdCax7} hold. We start by proving~\ref{CdCax1}, for which it suffices to calculate and apply~\ref{F2}:
    \begin{align*}
    \dd[f + g] &=~ \dd[+ \circ \pair{f}{g}] \\
    &=~ \dd[+_{B}] \circ \pair{\pair{f}{g} \circ \pi_0}{ \pair{\dd[f]}{\dd[g]}} \tag*{\ref{CdCax5}} \\
    &=~ + \circ \pi_1 \circ \pair{\pair{f}{g} \circ \pi_0}{ \pair{\dd[f]}{\dd[g]}} \tag*{\ref{CdCax2}} \\
    &=~ + \circ \pair{\dd[f]}{\dd[g]} \\
    &=~ \dd[f] + \dd[g]
  \end{align*}
It is not hard to show that $\dd[0] = 0$ and $\dd[\varepsilon(f)] = \varepsilon(\dd[f])$. For the first property, on one hand we have:
 \begin{align*}
    0 \circ \oplus \circ \pair{x}{y} &= 0
    = \oplus \circ \pair{0 \circ x}{0 \circ \pair{x}{y}}
 \end{align*}
 On the other hand, applying~\ref{CADax1}:
 \begin{align*}
    0 \circ \oplus \circ \pair{x}{y}
    = \oplus \circ \pair{0 \circ x}{\dd[0] \circ \pair{x}{y}}
 \end{align*}
 Hence by~\ref{F4} we have $\dd[0] \circ \pair{x}{y} = 0 \circ \pair{x}{y}$ for any choice of $x$ and $y$.
 In particular, $\dd[0] = \dd[0] \circ \pair{\pi_0}{\pi_1} = 0 \circ \pair{\pi_0}{\pi_1} = 0$ as desired.
 For the infinitesimal extension, we simply apply the chain rule and the equation $\dd[0] = 0$:
 \begin{align*}
    \dd[\varepsilon(f)]
    &= \dd[\oplus_B \circ \pair{0}{f}]\\
    &= \dd[\oplus_B] \circ \pair{\pair{0}{f} \circ \pi_0}{\dd[\pair{0}{f}]} \tag*{\ref{CdCax5}}\\
    &= \oplus_B \circ \dd[\pair{0}{f}]
    \tag*{\ref{F2}}
    \\
    &= \oplus_B \circ \pair{\dd[0]}{\dd[f]} \tag*{\ref{CdCax4}} \\
    &= \oplus_B \circ \pair{0}{\dd[f]} \tag{$\dd[0] = 0$}
    \\
    &= \varepsilon(\dd[f])
 \end{align*}
 So we conclude that~\ref{CdCax1} holds.

 We proceed to prove axioms~\ref{CdCax6a} and~\ref{CdCax7a} --- which, per Lemma~\ref{lem:cdc6a},
  is equivalent to, and easier to prove than~\ref{CdCax6} and~\ref{CdCax7}. Starting with~\ref{CdCax6a}, as before, we pick arbitrary
  $x, y : A \to B$ and calculate:
  \begin{align*}
    0 \oplus \dd^2[f] \circ \four{x}{0}{0}{y}
    &= \left( \dd[f] \circ \pair{x}{0} \right) \oplus \left( \dd^2[f] \circ \four{x}{0}{0}{y} \right)
    \tag*{\ref{CdCax2}}
    \\
&=      \dd[f] \circ \left( \pair{x}{0} \oplus \pair{0}{y} \right)  \tag*{\ref{CADax1}} \\
&= \dd[f] \circ \left( \pair{x \oplus 0}{0 \oplus y} \right) \tag{Lemma~\ref{lem:CACTprod.oplus}} \\
    &= \dd[f] \circ \pair{x}{0 \oplus y} \tag{Lemma~\ref{lem:CAadd.0}}\\
    &= 0 \oplus \dd[f] \circ \pair{x}{y}
    \tag*{\ref{Fd}}
  \end{align*}
  Hence by~\ref{F4}, $\dd^2[f]\circ \four{x}{0}{0}{y} = \dd[f] \circ \pair{x}{y}$ as desired.

  Finally, for~\ref{CdCax7a}, first observe that by Lemma~\ref{oplus-epsilon} we can compute the following equality for any suitable maps:
 \begin{align*}
(f \oplus g) \oplus (h \oplus k) &=~  (f +\varepsilon(g) ) \oplus (h + \varepsilon(k) )  \tag{Lemma~\ref{oplus-epsilon}} \\
    &= f +\varepsilon(g)  + \varepsilon\left(h + \varepsilon(k) \right)   \tag{Lemma~\ref{oplus-epsilon}} \\
        &= f +\varepsilon(g)  + \varepsilon(h) + \varepsilon^2(k)    \tag*{\ref{Eax1}} \\
    &=f +\varepsilon(h)  + \varepsilon(g) + \varepsilon^2(k)   \tag{by commutativity of $+$} \\
&=    f +\varepsilon(h)  + \varepsilon\left(g + \varepsilon(k) \right)  \tag*{\ref{Eax1}} \\
&=  (f +\varepsilon(h) ) \oplus (g + \varepsilon(k) )   \tag{Lemma~\ref{oplus-epsilon}} \\
&= (f \oplus h) \oplus (g \oplus k)
 \end{align*}
 Thus we have the following identity:
 \begin{equation}\label{swapoplus}\begin{gathered} (f \oplus g) \oplus (h \oplus k)  = (f \oplus h) \oplus (g \oplus k)
 \end{gathered}\end{equation}
Alternatively, we could have computed the following for any suitable maps:
 \begin{align*}
(f \oplus g) \oplus (h \oplus k) &=~  (f +\varepsilon(g) ) \oplus (h + \varepsilon(k) )  \tag{Lemma~\ref{oplus-epsilon}} \\
    &= f +\varepsilon(g)  + \varepsilon\left(h + \varepsilon(k) \right)   \tag{Lemma~\ref{oplus-epsilon}} \\
        &= f +\varepsilon(g)  + \varepsilon(h) + \varepsilon^2(k)    \tag*{\ref{Eax1}} \\
&= f + \varepsilon\left( g +h + \varepsilon(k) \right)     \tag*{\ref{Eax1}} \\
&= f \oplus \left( g +h + \varepsilon(k) \right)  \tag{Lemma~\ref{oplus-epsilon}} \\
&= f \oplus \left( (g+h) \oplus k \right) \tag{Lemma~\ref{oplus-epsilon}}
 \end{align*}
Thus we also have the following identity:
 \begin{equation}\label{swapoplus2}\begin{gathered} (f \oplus g) \oplus (h \oplus k)  = f \oplus \left( (g+h) \oplus k \right)
 \end{gathered}\end{equation}
Then we compute the following for suitable maps:
  \begin{align*}
&f \circ \left((x \oplus y) \oplus (z \oplus w) \right) = \left( f \circ \left( x \oplus y \right) \right) \oplus \left( \dd[f]\circ \langle x \oplus y, z \oplus w \rangle \right)
    \tag*{\ref{CADax1}} \\
    &= \left( (f \circ x) \oplus \left( \dd[f]\circ \langle x, y  \rangle \right) \right) \oplus \left( \dd[f] \circ \langle x \oplus y, z \oplus w \rangle \right)
    \tag*{\ref{CADax1}}
    \\
&= \left( (f \circ x) \oplus \left( \dd[f]\circ \langle x, y  \rangle \right) \right) \oplus \left( \dd[f] \circ \left( \pair{x}{z} \oplus \pair{y}{w} \right)  \right)
    \\
&= \left( (f \circ x) \oplus \left( \dd[f]\circ \langle x, y  \rangle \right) \right) \oplus \left( \left( \dd[f] \circ \langle x, z \rangle \right) \oplus \left( \dd^2[f] \circ \four{x}{z}{y}{w} \right) \right)   \tag*{\ref{CADax1}} \\
&= (f \circ x) \oplus \left( \left( \left( \dd[f]\circ \langle x, y  \rangle \right) + \left( \dd[f] \circ \langle x, z \rangle \right) \right)  \oplus \left( \dd^2[f] \circ \four{x}{z}{y}{w} \right) \right) \tag{\ref{swapoplus2}}
  \end{align*}
So we have the following identity:
 \begin{equation}\label{swapoplus3}\begin{gathered} f \circ \left((x \oplus y) \oplus (z \oplus w) \right) = \\ (f \circ x) \oplus \left( \left( \left( \dd[f]\circ \langle x, y  \rangle \right) + \left( \dd[f] \circ \langle x, z \rangle \right) \right)  \oplus \left( \dd^2[f] \circ \four{x}{z}{y}{w} \right) \right)
 \end{gathered}\end{equation}
However, by swapping the middle two arguments on the left hand side using (\ref{swapoplus}) we compute:
\begin{align*}
&(f \circ x) \oplus \left( \left( \left( \dd[f]\circ \langle x, y  \rangle \right) + \left( \dd[f] \circ \langle x, z \rangle \right) \right)  \oplus \left( \dd^2[f] \circ \four{x}{z}{y}{w} \right) \right) \\
&= f \circ \left((x \oplus y) \oplus (z \oplus w) \right) \tag{\ref{swapoplus3}} \\
&= f \circ \left((x \oplus z) \oplus (y \oplus w) \right) \tag{\ref{swapoplus}} \\
&=(f \circ x) \oplus \left( \left( \left( \dd[f]\circ \langle x, z  \rangle \right) + \left( \dd[f] \circ \langle x, y \rangle \right) \right)  \oplus \left( \dd^2[f] \circ \four{x}{y}{z}{w} \right) \right)\tag{\ref{swapoplus3}} \\
&= (f \circ x) \oplus \left( \left( \left( \dd[f]\circ \langle x, y  \rangle \right) + \left( \dd[f] \circ \langle x, z \rangle \right) \right)  \oplus \left( \dd^2[f] \circ \four{x}{y}{z}{w} \right) \right) \tag{by com. of $+$}
\end{align*}
So we have that:
\begin{align*}
    &(f \circ x) \oplus \left( \left( \left( \dd[f]\circ \langle x, y  \rangle \right) + \left( \dd[f] \circ \langle x, z \rangle \right) \right)  \oplus \left( \dd^2[f] \circ \four{x}{z}{y}{w} \right) \right) \\
    &=  (f \circ x) \oplus \left( \left( \left( \dd[f]\circ \langle x, y  \rangle \right) + \left( \dd[f] \circ \langle x, z \rangle \right) \right)  \oplus \left( \dd^2[f] \circ \four{x}{y}{z}{w} \right) \right)
\end{align*}
By applying~\ref{F4}, we obtain that:
\begin{align*} &\left( \left( \dd[f]\circ \langle x, y  \rangle \right) + \left( \dd[f] \circ \langle x, z \rangle \right) \right)  \oplus \left( \dd^2[f] \circ \four{x}{z}{y}{w} \right) \\& = \left( \left( \dd[f]\circ \langle x, y  \rangle \right) + \left( \dd[f] \circ \langle x, z \rangle \right) \right)  \oplus \left( \dd^2[f] \circ \four{x}{y}{z}{w} \right)  \end{align*}
By applying~\ref{F4} again we finally obtain $\dd^2[f]\left  \langle\langle x, y \rangle, \langle z, w \rangle \right \rangle= \dd^2[f]\left  \langle\langle x, z \rangle, \langle y, w \rangle \right \rangle$ as desired. So we conclude that $\mathsf{Flat}_\alpha$ is a Cartesian difference category.
    \hfill
\end{proof}

%Once again, it is important to mention that Proposition~\ref{CdtoCA} and Proposition~\ref{CAtoCd} result in an adjunction between the category of change action models and the category of Cartesian difference categories (which we will discuss in the upcoming journal version of this paper).

\subsection{Linear Maps and \texorpdfstring{$\varepsilon$}{Epsilon}-Linear Maps}

An important subclass of maps in a Cartesian differential category is the subclass of \emph{linear maps}~\cite[Definition 2.2.1]{blute2009cartesian}. One can also define linear maps in a Cartesian difference category by using the same definition.

\begin{defi}%
\label{def:linearity} In a Cartesian difference category, a map $f$ is \textbf{linear} if the following equality holds: $\dd[f] = f \circ \pi_1$.
\end{defi}

Using element-like notation, a map $f$ is linear if $\dd[f](x,y) = f(y)$. Linear maps in a Cartesian difference category satisfy many of the same properties found in~\cite[Lemma 2.2.2]{blute2009cartesian}.

\begin{lem}\label{lem:linear} In a Cartesian difference category,
  \begin{enumerate}[(\roman{enumi}),ref={\thelem.\roman{enumi}}]
  \item\label{epsilon-linear} If $f: A \to B$ is linear then $\varepsilon(f) = f \circ
  \varepsilon(1_{A})$.
  \item\label{add-linear} If $f: A \to B$ is linear, then $f$ is additive.
  \item\label{id-linear} Identity maps, projection maps, and zero maps are linear.
  \item\label{comp-linear} The composite, sum, pairing, and product of linear maps are linear.
  \item\label{chain-linear} If $f: A \to B$ and $k: C \to D$ are linear, then for any map $g: B \to
  C$:
  \[ \dd[k \circ g \circ f] = k \circ \dd[g] \circ (f \times f) \]
  \item If an isomorphism is linear, then its inverse is linear.
  \item For any object $A$, $\oplus_A$ and $+_A$ are linear.\label{struct-linear}
\end{enumerate}

\end{lem}
\begin{proof}
  Most of the above results are either immediate or admit a similar proof to the
  ones in~\cite[Lemma~2.2.2]{blute2009cartesian}. We will prove
  i, as it differs from the differential setting:
  \begin{align*}
f \circ \varepsilon(1_{A}) &=~  f \circ (0 + \varepsilon(1_{A})) \\
&=~ f \circ 0 + \varepsilon(\dd[f] \circ \pair{0}{1_A})   \tag*{\ref{CdCax0}}\\
&=~ f \circ \pi_1 \circ \pair{0}{0} + \varepsilon(\dd[f] \circ \pair{0}{1_A}) \\
&=~ \dd[f] \circ \pair{0}{0} + \varepsilon(f \circ \pi_1 \circ \pair{0}{1_A}) \tag{$f$ is linear} \\
&=~ 0 + \varepsilon(f \circ 1_A) \tag*{\ref{CdCax2}} \\
&=~ \varepsilon(f)
  \end{align*}
So $\varepsilon(f) =  f \circ \varepsilon(1_{A})$.
\end{proof}

Using element-like notation, the first point of the above lemma says that if $f$ is linear then $f(\varepsilon(x)) = \varepsilon(f(x))$. It is important to note that while all linear maps are additive, the converse is not necessarily true, see~\cite[Corollary 2.3.4]{blute2009cartesian}. That said, an immediate consequence of the above lemma is that the subcategory of linear maps of a Cartesian difference category has finite biproducts. %We also note a particular omission from~\cite[Lemma 2.2.2]{blute2009cartesian} for an arbitrary Cartesian difference category due to \textbf{[C$\dd$.6]}, which is that $\dd[f] \circ \langle 0, 1_A \rangle$ is no longer necessarily a linear map.

Another interesting subclass of maps is the subclass of $\varepsilon$-linear maps, which are maps whose infinitesimal extension is linear.

\begin{defi} In a Cartesian difference category, a map $f$ is \textbf{$\varepsilon$-linear} if $\varepsilon(f)$ is linear.
\end{defi}

\begin{lem} In a Cartesian difference category,
  \begin{enumerate}[(\roman{enumi}),ref={\thelem.\roman{enumi}}]
\item If $f: A \to B$ is $\varepsilon$-linear then $f \circ (x + \varepsilon(y)) = f \circ x + \varepsilon(f) \circ y$;
    \item Every linear map is $\varepsilon$-linear;
    \item The composite, sum, and pairing of $\varepsilon$-linear maps is $\varepsilon$-linear;
    \item If an isomorphism is $\varepsilon$-linear, then its inverse is again $\varepsilon$-linear.
    \item If $\varepsilon$ is nilpotent (.i.e for every $f$ there exists a $k \in \mathbb{N}$ such that $\varepsilon^k(f)= 0$), then for every map $f$, $\dd[f] \circ \pair{0}{1}$ is linear.\label{nilpotent-linear}
\end{enumerate}
\end{lem}
\begin{proof} The first four (i) through (iv) are straightforward and so we leave them as an exercise for the reader.
For (v), suppose that $\varepsilon$ is nilpotent. We need to show that $\dd[\varepsilon(\dd[f] \circ \pair{0}{1})]
  = \varepsilon(\dd[f] \circ \pair{0}{1}) \circ \pi_1$. So we compute:
\begin{align*}
    \dd[\varepsilon(\dd[f] \circ \pair{0}{1})] &=~ \varepsilon\left( \dd\left[ \dd[f] \circ \pair{0}{1} \right] \right) \tag*{\ref{CdCax1}} \\
    &=~ \varepsilon\left( \dd\left[\dd[f] \right] \circ \left( \pair{0}{1} \times \pair{0}{1} \right) \right)\tag{Lemma~\ref{chain-linear}} \\
       &=~ \varepsilon\left( \dd\left[\dd[f] \right] \circ \four{0}{\pi_0}{0}{\pi_1} \right) \\
       &=~ \varepsilon\left(\dd[f] \circ \pair{0 + \varepsilon(\pi_0)}{\pi_1}\right)   \tag*{\ref{CdCax6}} \\
      &=~  \varepsilon\left(\dd[f] \circ \pair{\varepsilon(\pi_0)}{\pi_1}\right) \\
      &=~ \varepsilon\left(\dd[f] \circ \left( \pair{0}{\pi_1}+ \pair{\varepsilon(\pi_0)}{0} \right)\right) \\
            &=~ \varepsilon\left(\dd[f] \circ \left( \pair{0}{\pi_1}+ \pair{\varepsilon(\pi_0)}{\varepsilon(0)} \right)\right) \tag*{\ref{Eax1}} \\
 &=~ \varepsilon\left(\dd[f] \circ \left( \pair{0}{\pi_1}+ \varepsilon(\pair{\pi_0}{0}) \right)\right)  \tag{Lemma~\ref{lem:ep-pair}} \\
  &=~ \varepsilon\left(\dd[f] \circ \pair{0}{\pi_1} + \varepsilon\left( \dd[\dd[f]] \circ \four{0}{\pi_1}{\pi_0}{0} \right) \right)   \tag*{\ref{CdCax0}} \\
&=~ \varepsilon\left(\dd[f] \circ \pair{0}{\pi_1} \right) + \varepsilon^2\left( \dd[\dd[f]] \circ \four{0}{\pi_1}{\pi_0}{0} \right) \tag*{\ref{Eax1}} \\
&=~ \varepsilon\left(\dd[f] \right)  \circ \pair{0}{\pi_1} + \varepsilon^2\left( \dd[\dd[f]] \right)  \circ \four{0}{\pi_1}{\pi_0}{0} \tag*{\ref{Eax2}} \\
&=~ \varepsilon\left(\dd[f] \right)  \circ \pair{0}{\pi_1} + \varepsilon^k\left( \dd[\dd[f]] \right)  \circ \four{0}{\pi_1}{\pi_0}{0} \tag{Iterating Lemma~\ref{lem:d-epsilon-iii}}     \\
&=~ \varepsilon\left(\dd[f] \right)  \circ \pair{0}{\pi_1} + \dd[\dd[\varepsilon^k\left( f \right) ]]  \circ \four{0}{\pi_1}{\pi_0}{0}   \tag*{(\ref{CdCax1})} \\
&=~ \varepsilon\left(\dd[f] \right)  \circ \pair{0}{\pi_1} + \dd[\dd[0]]  \circ \four{0}{\pi_1}{\pi_0}{0}    \tag{$\varepsilon$ is nilpotent} \\
&=~ \varepsilon\left(\dd[f] \right)  \circ \pair{0}{\pi_1} + 0  \circ \four{0}{\pi_1}{\pi_0}{0} \tag*{\ref{CdCax1}} \\
&=~ \varepsilon\left(\dd[f] \right)  \circ \pair{0}{\pi_1} + 0 \\
    &=~ \varepsilon(\dd[f]) \circ \pair{0}{1} \circ \pi_1 \\
&=~ \varepsilon(\dd[f] \circ \pair{0}{1}) \circ \pi_1 \tag*{\ref{Eax2}}
\end{align*}
So we conclude that $\varepsilon(\dd[f]) \circ \pair{0}{1}$ is linear and hence $\dd[f] \circ \pair{0}{1}$ is $\varepsilon$-linear.
\end{proof}

Using element-like notation, the first point of the above lemma says that if $f$ is $\varepsilon$-linear then $f(x + \varepsilon(y)) = f(x) + \varepsilon(f(y))$. So $\varepsilon$-linear maps are additive on ``infinitesimal'' elements (i.e.\ those of the form $\varepsilon(y)$). For a Cartesian differential category, linear maps in the Cartesian difference category sense are precisely the same as the Cartesian differential category sense~\cite[Definition 2.2.1]{blute2009cartesian}, while every map is $\varepsilon$-linear since $\varepsilon =0$.

\section{Examples of Cartesian Difference Categories}\label{EXsec}

\subsection{Smooth Functions}\label{smoothex}

As we have shown in Proposition~\ref{CDtoCd}, every Cartesian differential category is a Cartesian difference category where the infinitesimal extension is zero. As a particular example, we consider the category of real smooth functions
which, as mentioned above, can be considered to be the canonical (and motivating) example of a Cartesian differential category.

\begin{defi}
Let $\mathbb{R}$ denote the set of real numbers. The category $\mathsf{SMOOTH}$ is the category whose objects are Euclidean
spaces $\mathbb{R}^n$ (including the point $\mathbb{R}^0 = \lbrace \ast \rbrace$), and whose maps are smooth functions $F:
\mathbb{R}^n \to \mathbb{R}^m$.
\end{defi}

$\mathsf{SMOOTH}$ is a Cartesian left additive category where the product structure is given by the standard Cartesian product of Euclidean spaces and where the additive structure is defined by point-wise addition, $(F+G)(\vec x) = F(\vec x) + G(\vec x)$ and $0(\vec x) = (0, \hdots, 0)$, where $\vec x \in \mathbb{R}^n$. $\mathsf{SMOOTH}$ is a Cartesian differential category where the differential combinator is defined by the directional derivative of smooth functions. Explicitly, for a smooth function $F: \mathbb{R}^n \to \mathbb{R}^m$, which is in fact a tuple of smooth functions $F= (f_1, \hdots, f_n)$ where $f_i: \mathbb{R}^n \to \mathbb{R}$, $\mathsf{D}[F]: \mathbb{R}^n \times \mathbb{R}^n \to \mathbb{R}^m$ is defined as follows:
\[\mathsf{D}[F]\left(\vec x, \vec y \right) := \left( \sum \limits^n_{i=1} \frac{\partial f_1}{\partial u_i}(\vec x) y_i, \hdots, \sum \limits^n_{i=1} \frac{\partial f_n}{\partial u_i}(\vec x) y_i  \right)\]
where $\vec x = (x_1, \hdots, x_n), \vec y = (y_1, \hdots, y_n) \in \mathbb{R}^n$. Alternatively, $\mathsf{D}[F]$ can also be defined in terms of the Jacobian matrix of $F$. Therefore $\mathsf{SMOOTH}$ is a Cartesian difference category with infinitesimal extesion $\varepsilon =0$ and with difference combinator $\mathsf{D}$. Since $\varepsilon = 0$, the induced action is simply $\vec x \oplus \vec y = \vec x$. A smooth function is linear in the Cartesian difference category sense precisely if it is $\mathbb{R}$-linear in the classical sense, and every smooth function is $\varepsilon$-linear since $\varepsilon(f)=0$ is linear.

\subsection{Calculus of Finite Differences}\label{discreteex}

The calculus of finite differences~\cite{jordan1965calculus, gleich2005finite} is a field which aims to apply methods from
differential calculus to discrete settings. It does so by introducing a ``discrete derivative'' or ``finite difference''
operator $\Delta$ which, when applied to an integer-valued function $f : \mathbb{Z} \to \mathbb{Z}$, gives the function
$\Delta(f)$ defined by
\[
    \Delta(f) (x) = f(x + 1) - f(x)
\]

Here we show that a generalization of the finite difference operator gives an example of a Cartesian
difference category (but \emph{not} a Cartesian differential category). This example was the main motivating example for
developing Cartesian difference categories. The behaviour of the calculus of finite differences, as well as some of its
generalizations (notably, the Boolean differential calculus~\cite{thayse1981boolean, steinbach2009boolean}) is captured by
the category of abelian groups and arbitrary set functions between them, equipped with a suitable Cartesian difference
structure.

\begin{defi}
    $\overline{\mathsf{Ab}}$ is the category whose objects are abelian groups $G$ (where we use additive notation for group structure) and where a map $f: G \to H$ is simply an arbitrary function between them (and therefore does not necessarily preserve the group structure).
\end{defi}

\begin{prop}
    $\overline{\mathsf{Ab}}$ is a Cartesian left additive category where the product structure is given by the standard Cartesian product of sets and where the additive structure is again given by point-wise addition, $(f+g)(x)=f(x) + g(x)$ and $0(x)=0$. Furthermore, $\overline{\mathsf{Ab}}$ is a Cartesian difference category where the infinitesimal extension is given by the identity, that is, $\varepsilon(f)=f$, and and where the difference combinator $\dd$ is defined as follows for a map $f: G \to H$:
 \[\dd[f](x,y) = f(x + y) - f(x)\]
\end{prop}

On the other hand, note that $\dd$ is not a differential combinator for $\overline{\mathsf{Ab}}$ since it does not satisfy~\ref{CDCax6} and part of~\ref{CDCax2}. Indeed, since $f$ is not necessarily a group homomorphism,~\ref{CDCax2} fails to
hold as:
\[\dd[f](x,y +z) = f(x + y +z) - f(x)\]
is not necessarily equal to:
\[\dd[f](x,y) + \dd[f](x,z)= f(x+y) - f(x) + f(x+z) - f(x)\]
$\dd$ does satisfy~\ref{CdCax2} and~\ref{CdCax6}, as well as~\ref{CdCax0}.\ which in this case are respectively:
\[\dd[f](x,y+z) = \dd[f](x,y) + \dd[f](x + y, z)\]
\[\dd \left[ \dd[f] \right]\left( (x,y), (0,z) \right) =  \dd[f](x +y, z) \]
\[ f(x + y) = f(x) + \dd[f](x,y) \]
However, as noted in~\cite{FMCS2018}, it is interesting to note that $\dd$ does satisfy~\ref{CDCax1}, the second part of~\ref{CDCax2},~\ref{CDCax3},~\ref{CDCax4},~\ref{CDCax5},~\ref{CDCax7}, and~\ref{CdCax6a}. It is worth emphasizing that in~\cite{FMCS2018}, the goal was to drop the addition and develop a ``non-additive'' version of Cartesian differential
categories, whereas the current presentation keeps the additive structure while suitably relaxing the differential
combinator to a difference combinator.

In $\overline{\mathsf{Ab}}$, since the infinitesimal operator is given by the identity, the induced action is simply the
addition, $x \oplus y = x + y$. On the other hand, the linear maps in $\overline{\mathsf{Ab}}$ are precisely the group
homomorphisms. Indeed, $f$ is linear if $\dd[f](x,y) = f(y)$. But by~\ref{CdCax0} and~\ref{CdCax2}, we get that:
\[f(x + y) = f(x) + \dd[f](x,y)= f(x) + f(y) \quad \quad \quad f(0) =  \dd[f](x,0) = 0 \]
So $f$ is a group homomorphism. Conversely, whenever $f$ is a group homomorphism:
\[\dd[f](x,y) = f(x+y) - f(x) = f(x) + f(y) - f(x) = f(y)\]
So $f$ is linear. Since $\varepsilon(f)=f$, the $\varepsilon$-linear maps are precisely the linear maps.

\subsection{Module Morphisms}\label{moduleex}

Here we provide a simple example of a Cartesian difference category whose difference combinator is also a differential combinator, but where the infinitesimal extension is neither zero nor the identity.

\begin{defi}
    Let $R$ be a commutative semiring. We define the category $\mathsf{MOD}_R$ as the category whose objects are all
    $R$-modules and whose maps are all the $R$-linear maps between them.
\end{defi}

$\mathsf{MOD}_R$ has finite biproducts and is, therefore, a Cartesian additive (in particular, left additive) category.
Every $r \in R$ induces an infinitesimal extension $\varepsilon^r$ defined by scalar multiplication,
$\varepsilon^r(f)(m) = r f(m)$. For any choice of $r$, the category $\mathsf{MOD}_R$ is a Cartesian difference category
with the infinitesimal extension $\varepsilon^r$ and its difference combinator $\dd$ defined as:
\[\dd[f](m,n)=f(n)\]
$R$-linearity of $f$ assures that~\ref{CdCax0} holds, while the remaining Cartesian difference axioms hold trivially. In
fact, $\dd$ is also a differential combinator and therefore $\mathsf{MOD}_R$ is also a Cartesian differential category,
but note that the Cartesian difference structure of $\mathsf{MOD}_R$ is (whenever $R$ is non-zero) different than the
Cartesian difference structure that corresponds to this differential structure. The
induced action is given by $m \oplus n = m + rn$. By definition of $\dd$, every map in $\mathsf{MOD}_R$ is linear, and by
definition of $\varepsilon^r$ and $R$-linearity, every map is also $\varepsilon$-linear.

\newcommand{\seq}[1]{\left[ {#1} \right]}
\newcommand{\Ab}[0]{\overline{\mathsf{Ab}}}
\newcommand{\z}[0]{\mathbf{z}}
\subsection{Stream calculus}\label{streamex}

It is common knowledge that streams, i.e.\ infinite sequences of values, can be studied using methods from differential
calculus. For example, in~\cite{rutten2005coinductive}, a notion of \emph{stream derivative} operator is introduced,
and streams are characterized as the solutions of stream differential equations involving stream derivatives. More recent work in the setting of causal functions between streams of real numbers~\cite{sprunger2019differential, sprunger2019differentiable} has focused on extending the ``classical'' notion of the derivative of a real-valued
function to stream-valued functions.

We seek now to show that causal functions between streams are indeed endowed with the structure of a Cartesian
difference category, the corresponding difference combinator capturing the change of the stream over time. For this we will introduce an idempotent infinitesimal extension on streams that plays a similar role in this setting as Rutten's stream derivative operator~\cite{rutten2005coinductive}, which is given by discarding the head of the stream. On the other hand, our work is more closely related to Sprunger et al.'s work~\cite{sprunger2019differential, sprunger2019differentiable} as it focuses on the differentiation of functions between streams, rather than the description of single streams in terms of differential equations.

For a set $A$, let $A^\omega$ denote the set of infinite sequences of elements of $A$. We write $\seq{a_i}$ for the infinite sequence $\seq{a_i} = (a_0, a_1, a_2, \hdots)$ and $a_{i:\omega}$ for the (infinite) subsequence $(a_i, a_{i + 1}, \hdots)$. A function $f : A^\omega \to B^\omega$ is \textbf{causal} whenever the
$n$-th element ${f\left(\seq{a_i}\right)}_n$ of the output sequence only depends on the first $n$ elements of $\seq{a_i}$. More formally, the function $f$ is causal if and only if, whenever $a_{0:n} = b_{0:n}$,
then ${f\left(\seq{a_i}\right)}_{0:n} = {f\left(\seq{b_i}\right)}_{0:n}$.

We will restrict ourselves to considering streams over abelian groups\footnote{
A similar approach to the one in~\cite{sprunger2019differentiable} is possible where we consider streams on arbitrary difference categories, and lift the difference operator of the underlying category to its category of streams, although it would complicate the presentation of this section without gaining clarity.
}, so let $\Ab^\omega$ be the category whose objects are all the abelian groups and where a morphism from $G$ to $H$ in $\Ab^\omega$ is a causal map from $G^\omega$ to $H^\omega$. $\Ab^\omega$ is a Cartesian left-additive category, where the product is given by the standard product of abelian groups and where the additive structure is lifted point-wise from the structure of $\Ab$.

More concretely, whenever $G, H$ are abelian groups, the hom-set $\Ab^\omega(G, H)$ is endowed with the additive structure that comes from setting
\begin{align*}
{(f+g)\left(\seq{a_i}\right)}_n = {f\left(\seq{a_i}\right)}_n + {g\left(\seq{a_i}\right)}_n && {0\left(\seq{a_i}\right)}_n = 0
\end{align*}

At this point, one might think that the ``natural'' choice of an infinitesimal extension in this setting would
be something akin to Rutten's stream derivative operator, which is given by dropping the first element of the stream, that is, $(\seq{a_i})' = {(\seq{a_i})}_{1:\omega}$. This, however, is not a causal function and so it
does not exist in the category $\Ab^\omega$. That said, it is possible to construct a larger category where this operator is used for differentiation, it however fails to satisfy~\ref{CdCax6}. We instead an define infinitesimal extension for $\Ab^\omega$ as the \emph{truncation} operator $\z$.

\begin{defi} For an abelian group $A$, define the \textbf{truncation operator} $\z_A : A^\omega \to A^\omega$ as follows:
  \begin{align*}
    {(\z \seq{a_i})}_0 = 0 &&  {(\z \seq{a_i})}_{j + 1} = a_{j + 1}
  \end{align*}
\end{defi}

  Note that $\z_A$ is a monoid homomorphism according to the pointwise monoid structure on $A^\omega$. Thus it is straightforward to see that we obtain an infinitesimal extension $\z$ for $\Ab^\omega$.

\begin{thm}
    The category $\Ab^\omega$ is a Cartesian difference category, with the infinitesimal extension $\varepsilon(f) = \z \circ f$ and a difference operator defined as:
    \begin{align*}
           {\dd[f]\left(\seq{a_i}, \seq{b_i}\right)}_0 &= {f\left(\seq{a_i} + \seq{b_i}\right)}_0 -  {f\left(\seq{a_i}\right)}_0 \\
                {\dd[f]\left(\seq{a_i}, \seq{b_i}\right)}_{n + 1} &= {f\left(\seq{a_i} + \z(\seq{b_i}) \right)}_{n + 1} -  {f\left(\seq{a_i}\right)}_{n+1}
    \end{align*}
\end{thm}
\begin{proof} We leave it to the reader to check for themselves that $\z$ is an infinitesimal extension.~\ref{CdCax0} is satisfied because of causality. Indeed, since $f$ is causal, $f(\seq{a_i} + \z(\seq{b_i}))$ depends
    only on ${\left(\seq{a_i} + \z(\seq{b_i})\right)}_0$, but note that, by definition of $\z$, this term is precisely
    equal to $a_0$. Hence we obtain the following:
    \begin{align*}
        {\left(f(\seq{a_i}) + \z\left(\dd[f](\seq{a_i}, \seq{b_i})\right)\right)}_0
        = {f\left(\seq{a_i}\right)}_0
        = {f\left(\seq{a_i} + \z(\seq{b_i})\right)}_0
    \end{align*}
    For any other index, we have that:
    \begin{align*}
        {\left(f(\seq{a_i}) + \z\left(\dd[f](\seq{a_i}, \seq{b_i})\right)\right)}_{i+1}
        &= {f(\seq{a_i})}_{i+1} + {\z\left(\dd[f](\seq{a_i}, \seq{b_i})\right)}_{i+1}\\
        &= {f(\seq{a_i})}_{i+1} + {\dd[f]\left(\seq{a_i}, \seq{b_i}\right)}_{i+1}\\
        &= {f(\seq{a_i})}_{i+1} + {f\left(\seq{a_i} + \z(\seq{b_i})\right)}_{i+1} - {f(\seq{a_i})}_{i+1}\\
        &= {f\left(\seq{a_i} + \z(\seq{b_i})\right)}_{i+1}
    \end{align*}

    We also explicitly prove the part of~\ref{CdCax1} concerning the infinitesimal extension $\z$:
    \begin{align*}
        {\dd[\z]\left(\seq{a_i}, \seq{b_i}\right)}_0 &= {\z\left(\seq{a_i} + \seq{b_i}\right)}_0 - {\z(\seq{a_i})}_0 = 0 = {\z(\seq{b_i})}_0\\
        {\dd[\z]\left(\seq{a_i}, \seq{b_i}\right)}_{n+1} &= {\z\left(\seq{a_i} + \z(\seq{b_i})\right)}_{n+1} - {\z(\seq{a_i})}_{n+1} = b_{n+1} = {\z(\seq{b_i})}_{n+1}
    \end{align*}

    Finally, we show that~\ref{CdCax3} holds for the identity map, which is a matter of simple
    calculation as well
    (linearity of the projection maps follows by an almost identical argument). However, as we remark
    later, this property would fail to hold if we had chosen a more ``natural'' candidate for the
    difference operator.
    \begin{align*}
        {\dd[1]\left(\seq{a_i}, \seq{b_i}\right)}_0 &= a_0 + b_0 - a_0 = b_0\\
        {\dd[1]\left(\seq{a_i}, \seq{b_i}\right)}_{n+1} &= {\left(\seq{a_i} + \z(\seq{b_i})\right)}_{n+1} - a_{n+1} = b_{n+1}
    \end{align*}
    The remaining axioms can be shown to hold by similar pointwise reasoning; the corresponding calculations are very similar to the case of $\Ab$.
\end{proof}

\begin{rem}
    One might expect the difference operator in $\Ab^\omega$ to be given by the simpler expression
    \[
        \dd[f](\seq{a_i}, \seq{b_i}) = f(\seq{a_i} + \z(\seq{b_i})) - f(\seq{a_i})
    \]
    While this satisfies some of the Cartesian difference axioms (notably~\ref{CdCax0})
    it does not satisfy all of them: for example,~\ref{CdCax3} fails to hold since
    $\dd[1]\left( \seq{a_i}, \seq{b_i}\right) = \z(\seq{b_i}) \neq \seq{b_i}$.
\end{rem}

Note the similarities between the difference combinator on $\Ab$ and that on $\Ab^\omega$. The induced action can be computed out to be:
  \begin{align*}
   {\left(\seq{a_i} \oplus \seq{b_i}\right)}_0 = a_0 && {\left(\seq{a_i} \oplus \seq{b_i}\right)}_{n+1} = a_{n+1} + b_{n+1}~(\equiv~a_{n+1} \oplus b_{n+1})
\end{align*}

The linear maps (in the Cartesian difference category sense) in $\Ab^\omega$ are precisely those maps $f$ that are group
homomorphisms (when the set of streams $G^\omega$ is equipped with the structure lifted pointwise from the group $G$)
satisfying the additional property that whenever $\seq{a_i}_{1:\omega} = \seq{b_i}_{1:\omega}$ then ${f(\seq{a_i})}_{n+1}
= {f(\seq{b_i})}_{n+1}$, but this is far from evident. The reader can easily verify that any such homomorphism is linear,
we prove here the converse.

\begin{prop}
    Any linear map $f$ in $\Ab^\omega$ is a group homomorphism. Furthermore, whenever $\seq{a_i}_{1:\omega} =
    \seq{b_i}_{1:\omega}$, the map $f$ satisfies ${f(\seq{a_i})}_{n+1} = {f(\seq{b_i})}_{n+1}$.
\end{prop}
\begin{proof}
    The first part of the proposition is simply a corollary of Lemma~\ref{add-linear}, according to which
    every linear map in $\Ab^\omega$ is additive and, therefore, a group homomorphism. For the second property, since $f$ is linear, we have:
    \begin{align*}
        {f\left(\seq{b_i}\right)}_{n+1}
        = {\dd[f]\left(\seq{a_i}, \seq{b_i}\right)}_{n+1}
        = {f\left(\seq{a_i} + \z(\seq{b_i})\right)}_{n+1} - {f(\seq{a_i})}_{n+1}
    \end{align*}
    Therefore ${f\left(\seq{a_i} + \z(\seq{b_i})\right)}_{n+1} = {f\left(\seq{a_i}\right)}_{n+1} + {f\left(\seq{b_i}\right)}_{n+1}$.
    By setting $\seq{a_i} = \seq{0}$ in the above equation, and since $f$ preserves the identity element,
    we establish ${f\left(\z(\seq{b_i})\right)}_{n+1} = {f\left(\seq{b_i}\right)}_{n+1}$, from which the
    second part of the desired property follows as an immediate corollary.
\end{proof}

On the other hand, since every map of the form $f \circ \z$ verifies this second property of ``insensitivity'' to the initial element of the stream, it follows that for a map $f$ to be $\varepsilon$-linear it is sufficient (and necessary)
that $f \circ \varepsilon$ be a group homomorphism.

\newcommand{\T}[0]{\mathsf{T}}
\section{Tangent Bundles in Cartesian Difference Categories}%
\label{monadsec}

In this section, we show that the difference combinator of a Cartesian difference category induces a monad, called the \emph{tangent bundle monad}. This construction is a generalization of the tangent bundle monad for Cartesian differential categories~\cite{cockett2014differential,manzyuk2012tangent}. Furthermore, we show that a full subcategory of the Eilenberg-Moore category and the Kleisli category of the tangent bundle monad is again a Cartesian difference category. The general intuition of these categories for the tangent bundle monad are more or less the same as explained in~\cite[Section 3.2]{cockett2014differential}. The maps of the Kleisli category are generalized vector fields, while the considered full subcategory of the Eilenberg-Moore category consists of objects equipped with a linear map which associates tangent vectors to points.

\subsection{The Tangent Bundle Monad} If only to introduce notation, recall that a monad on a category $\mathbb{X}$ is a triple $(\T, \mu, \eta)$ consisting of an endofunctor $\T: \mathbb{X} \to \mathbb{X}$, and two natural transformations $\mu: \mathsf{T}^2 \Rightarrow \mathsf{T}$ and $\eta: \mathsf{1}_\mathbb{X} \Rightarrow \mathsf{T}$ (where $\mathsf{1}_\mathbb{X}: \mathbb{X} \to \mathbb{X}$ is the identity functor), such that the following equalities hold for all objects $A \in \mathbb{X}$:
\begin{equation}\label{monaddef}\begin{gathered} \mu_A \circ \eta_{\T(A)} = 1_{\T(A)} = \mu_A \circ \T(\eta_A) \quad \quad \quad \mu_A \circ \mu_{\T(A)} = \mu_A \circ \T(\mu_A)
 \end{gathered}\end{equation}
Now let $\mathbb{X}$ be a Cartesian difference category with infinitesimal extension $\varepsilon$ and difference combinator $\dd$. Define the functor $\mathsf{T}: \mathbb{X} \to \mathbb{X}$ as follows:
\[ \mathsf{T}(A) = A \times A \quad \quad \quad \mathsf{T}(f) = \langle f \circ \pi_0, \dd[f] \rangle \]
and define the natural transformations $\eta: \mathsf{1}_\mathbb{X} \Rightarrow \mathsf{T}$ and $\mu: \mathsf{T}^2 \Rightarrow \mathsf{T}$ as follows:
\[\eta_A := \langle 1_A, 0 \rangle \quad \quad \quad \mu_A := \left \langle \pi_{00}, \pi_{10} + \pi_{01} + \varepsilon(\pi_{11}) \right \rangle \]
where $\pi_{ij}= \pi_i \circ \pi_j$. Before providing the proof that $(\mathsf{T}, \mu, \eta)$ is indeed a monad for any Cartesian difference category, let us provide some examples for intuition.

\begin{exa}\label{ex:tsmooth}
For a Cartesian differential category, since $\varepsilon = 0$, the induced monad is precisely the monad induced by its tangent category structure~\cite{cockett2014differential,manzyuk2012tangent}. For example, in the Cartesian differential category $\mathsf{SMOOTH}$ (as defined in Section~\ref{smoothex}), one has that $\mathsf{T}(\mathbb{R}^n) = \mathbb{R}^n \times \mathbb{R}^n$, which is in fact the classical tangent bundle over $\mathbb{R}^n$ from differential geometry, and also that $\mathsf{T}(F)(\vec x, \vec y) = (F(\vec x), \mathsf{D}[F](\vec x, \vec y))$, $\eta_{\mathbb{R}^n}(\vec x) = (\vec x,\vec 0)$, and $\mu_{\mathbb{R}^n}((\vec x,\vec y), (\vec z, \vec w)) = (\vec x, \vec y + \vec z)$.
\end{exa}

\begin{exa}\label{ex:tab}
In the Cartesian difference category $\overline{\mathsf{Ab}}$ (as defined in Section~\ref{discreteex}), the monad is given by $\mathsf{T}(G) = G \times G$ $\mathsf{T}(f)(x,y) = (f(x), f(x+y) - f(x))$, $\eta_G(x) = (x,0)$, and $\mu_G((x,y),(z,w)) = (x, y + z + w)$.
\end{exa}

%Referee suggests that the proof could be simpler by showing Kleisli triple axiom (only 3 equations) instead of monad axioms (1 for functoriarity of T, 2 for naturality + 3 for monad laws). What do you think?

We now show that the above construction is indeed a monad.

\begin{prop} Let $\mathbb{X}$ be a Cartesian difference category with infinitesimal extension $\varepsilon$ and difference combinator $\dd$. Then $(\mathsf{T}, \mu, \eta)$, as defined above, is a monad on $\mathbb{X}$.
\end{prop}
\begin{proof} That $\mathsf{T}$ is a functor follows from the fact that the change action model $\alpha : \mathbb{X} \to \mathsf{CAct}(\mathbb{X})$ from Proposition~\ref{CdtoCA} is a functor. Specifically, $\mathsf{T}$ is the second component of $\alpha$. Next we show the naturality of $\eta$ and $\mu$. For $\eta$ we have:
  \begin{align*}
  \T(f) \circ \eta_A &=~ \pair{f \circ \pi_0}{\dd[f]} \circ \eta_A \\
  &=~  \pair{f \circ \pi_0 \circ \eta_A}{
    \dd[f] \circ \eta_A} \\
  &=~ \pair{f \circ \pi_0 \circ \pair{1_{A}}{0}}{
    \dd[f] \circ \pair{1_{A}}{0}} \\
    &=~  \pair{f}{0} \tag*{\ref{CdCax2}} \\
    &=~ \pair{1_{B}}{0} \circ f \\
  &=~ \eta_B \circ f
  \end{align*}
  For the naturality of $\mu$, first it is straightforward to check that:
  \begin{align*}
    \T^2(f) = \four{f \circ \pi_{00}}
    {\dd[f] \circ \pi_0}
    {\dd[f] \circ (\pi_0 \times \pi_0)}
    {\dd^2[f]}
  \end{align*}
and therefore we compute:
  \begin{align*}
  &\mu_B \circ \T^2(f) =~ \left \langle \pi_{00}, \pi_{10} + \pi_{01} + \varepsilon(\pi_{11}) \right \rangle \circ \T^2(f) \\
  &=~ \left \langle \pi_{00} \circ \T^2(f) , \left(\pi_{10} + \pi_{01} + \varepsilon(\pi_{11}) \right) \circ \T^2(f)   \right \rangle \\
  &=~ \left \langle \pi_{00} \circ \T^2(f) , \pi_{10}  \circ \T^2(f) + \pi_{01}  \circ \T^2(f) + \varepsilon(\pi_{11})  \circ \T^2(f)   \right \rangle \\
    &=~ \left \langle \pi_{00} \circ \T^2(f) , \pi_{10}  \circ \T^2(f) + \pi_{01}  \circ \T^2(f) + \varepsilon(\pi_{11}  \circ \T^2(f) )   \right \rangle  \tag*{\ref{Eax2}} \\
        &=~ \left \langle f \circ \pi_{00} , \dd[f] \circ \pi_0 + \dd[f] \circ (\pi_0 \times \pi_0)+ \varepsilon(\dd^2[f] )   \right \rangle \\
  & =~
  \pair{f \circ \pi_{00}}{
    \dd[f] \circ (\pi_0 \times \pi_0)
    + \left(
      \dd[f] \circ \pi_0 + \varepsilon(\dd^2[f])
    \right)
  } \\
    &=~ \pair{
      f \circ \pi_{00}
    }{
      \dd[f] \circ \pair{\pi_{00}}{\pi_{01}}
      + \left(\dd[f] \circ \pair{\pi_{00}}{\pi_{10}}
      + \varepsilon\left(\dd^2[f] \circ \four{\pi_{00}}{\pi_{10}}{\pi_{01}}{\pi_{11}} \right) \right)
    }\\
      &=~ \pair{
      f \circ \pi_{00}
    }{
      \dd[f] \circ \pair{\pi_{00}}{\pi_{01}}
      + \dd[f] \circ \pair{\pi_{00} + \varepsilon(\pi_{01})}{\pi_{10} + \varepsilon(\pi_{11})}
    }
     \tag*{\ref{CdCax0}} \\
        &=~ \pair{
      f \circ \pi_{00}
    }{
      \dd[f] \circ \pair{\pi_{00}}{\pi_{01} + \pi_{10} + \varepsilon(\pi_{11})}
    }
   \tag*{\ref{CdCax2}} \\
   &=~ \pair{f \circ \pi_0}{
      \dd[f]} \circ \pair{\pi_{00}}{\pi_{01} + \pi_{10} + \varepsilon(\pi_{11})
    }
    \\
&=~    \T(f) \circ \mu_A
  \end{align*}
 So $\eta$ and $\mu$ are natural transformations. Now we show that $\mu$ and $\eta$ satisfy the monad identities (\ref{monaddef}). Starting with $\mu_{A} \circ \eta_A = 1_{\T(A)}$:
  \begin{align*}
    \mu_A \circ \eta_{\T(A)}
    &=~  \left \langle \pi_{00}, \pi_{10} + \pi_{01} + \varepsilon(\pi_{11}) \right \rangle \circ  \eta_{\T(A)} \\
    &=~  \left \langle \pi_{00} \circ  \eta_{\T(A)} , \left(\pi_{10} + \pi_{01} + \varepsilon(\pi_{11}) \right) \circ  \eta_{\T(A)}  \right \rangle \\
    &=~ \left \langle \pi_{00} \circ  \eta_{\T(A)} , \pi_{10} \circ  \eta_{\T(A)} + \pi_{01} \circ  \eta_{\T(A)} + \varepsilon(\pi_{11}) \circ  \eta_{\T(A)}  \right \rangle \\
        &=~ \left \langle \pi_{00} \circ  \eta_{\T(A)} , \pi_{10} \circ  \eta_{\T(A)} + \pi_{01} \circ  \eta_{\T(A)} + \varepsilon(\pi_{11} \circ  \eta_{\T(A)} )   \right \rangle  \tag*{\ref{Eax2}} \\
        &=~ \left \langle \pi_{00} \circ  \pair{1_{T(A)}}{0} , \pi_{10} \circ  \pair{1_{T(A)}}{0} + \pi_{01} \circ \pair{1_{T(A)}}{0} + \varepsilon\left(\pi_{11} \circ \pair{1_{T(A)}}{0} \right)   \right \rangle  \\
    &=~ \pair{\pi_0}{\pi_1 + 0 + \varepsilon(0)}
    \\
    &=~ \pair{\pi_0}{\pi_1 + 0} \tag*{\ref{Eax1}} \\
    &=~ \pair{\pi_0}{\pi_1}\\
    &= 1_{\T(A)}
  \end{align*}
Next we check that $\mu_{A} \circ \T(\eta_A) = 1_{\T(A)}$. First note that $\T(\eta_A) = \four{\pi_0}{0}{\pi_1}{0}$. Then we compute:
\begin{align*}
\mu_A \circ\T(\eta_{A}) &=~  \pair{\pi_{00}}{\pi_{01} + \pi_{10} + \varepsilon(\pi_{11})}
    \circ \T(\eta_{A})
      \\
  &=~ \pair{\pi_{00}     \circ\T(\eta_{A}) }{\left(\pi_{01} + \pi_{10} + \varepsilon(\pi_{11}) \right)     \circ\T(\eta_{A})  } \\
    &=~ \pair{\pi_{00}     \circ\T(\eta_{A})}{\pi_{01}\circ\T(\eta_{A}) + \pi_{10}\circ\T(\eta_{A}) + \varepsilon(\pi_{11}) \circ\T(\eta_{A})   } \\
    &=~ \pair{\pi_{00}     \circ\T(\eta_{A}) }{\pi_{01}\circ\T(\eta_{A})+ \pi_{10}\circ\T(\eta_{A}) + \varepsilon\left(\pi_{11} \circ\T(\eta_{A})\right)}  \tag*{\ref{Eax2}} \\
        &=~ \pair{\pi_0}{\pi_1 + 0 + \varepsilon(0)}
    \\
 &=~ \pair{\pi_0}{\pi_1 + 0} \tag*{\ref{Eax1}} \\
    &=~ \pair{\pi_0}{\pi_1} \\
    &=~ 1_{\T(A)}
\end{align*}
For the last of the monad laws, we first note that, since $\mu$ is linear,
  it follows that $\T(\mu) = \mu \times \mu$. Then it suffices to compute:
  \begin{align*}
   & \mu_A \circ \T(\mu_A)
    = \mu_A \circ (\mu_A \times \mu_A)
    \\
    &= \pair{\pi_{00}}{\pi_{10} + \pi_{01} + \varepsilon(\pi_{11})}
    \circ (\mu_A \times \mu_A)
    \\
        &= \pair{\pi_{00}   \circ (\mu_A \times \mu_A)}{\left( \pi_{10} + \pi_{01} + \varepsilon(\pi_{11}) \right)    \circ (\mu_A \times \mu_A)}
    \\
            &= \pair{\pi_{00}   \circ (\mu_A \times \mu_A)}{\pi_{10} \circ (\mu_A \times \mu_A)+ \pi_{01} \circ (\mu_A \times \mu_A)+ \varepsilon(\pi_{11})   \circ (\mu_A \times \mu_A)}
    \\
                &= \pair{\pi_{00}   \circ (\mu_A \times \mu_A)}{\pi_{10} \circ (\mu_A \times \mu_A)+ \pi_{01} \circ (\mu_A \times \mu_A)+ \varepsilon(\pi_{11} \circ (\mu_A \times \mu_A))   }
\tag*{\ref{Eax2}} \\
    &=
    \pair{\pi_0 \circ \mu_A \circ \pi_0}{\pi_1 \circ \mu_A \circ \pi_0
    + \pi_0 \circ \mu_A \circ \pi_1 + \varepsilon(\pi_1 \circ \mu_A \circ \pi_1)}
    \\
    &=
    \pair{\pi_{000}}{
      \pi_{100} + \pi_{010} + \varepsilon(\pi_{110})
      + \pi_{001}
      + \varepsilon \pa{\pi_{101} + \pi_{011} + \varepsilon(\pi_{111})}
    }
    \\
    &=
    \pair{\pi_0 \circ \pi_{00}}{
      \pi_{1} \circ \pi_{00}
      + \pi_0 \circ \pa{\pi_{10} + \pi_{01} + \varepsilon(\pi_{11})}
      + \varepsilon\pa{\pi_1 \circ \pa{\pi_{10} + \pi_{01} + \varepsilon(\pi_{11})}}
    }
    \\
    &=
    \pair{\pi_{00}}{\pi_{10} + \pi_{01} + \varepsilon(\pi_{11}) }
    \circ \pair{\pi_{00}}{\pi_{10} + \pi_{01} + \varepsilon(\pi_{11}) }
    \\
    &= \mu_A \circ \mu_{\T(A)}
  \end{align*}
  So we conclude that $(\T, \mu, \eta)$ is a monad.
\end{proof}

Those familiar with monads may wonder if the Kleisli triple approach might have simplified the above proof. Recall that for a category $\mathbb{X}$, a Kleisli triple is a triple $(\T, \eta, {(\_)}^\sharp)$ consisting of a function on object $\T$, $A \mapsto \T(A)$, a family of maps indexed by the objects of $\mathbb{X}$, $\eta = \lbrace \eta_A: A \to \T(A) \vert~ A \in \mathbb{X} \rbrace$, and an operator ${(\_)}^\sharp$ which for any map $f: A \to \T(B)$ results in a map $f^\sharp: \T(A) \to \T(B)$, and such that the following equalities hold:
\[ f^\sharp \circ \eta_A = f \quad \quad \quad \eta_A^\sharp = 1_{\T(A)} \quad \quad \quad {(g^\sharp \circ f)}^\sharp = g^\sharp \circ f^\sharp \]
There is a bijective correspondence between monads and Kleisli triples. So for the tangent bundle monad, we could have instead defined its associated Kleisli triple and prove 3 equations instead of the functoriality of $\T$, the naturality of $\mu$ and $\eta$, and the 3 monad identities. However, we find the operator ${(\_)}^\sharp$ slightly more complicated to work with. As such, we elected to work out the monad identities directly, since while there are more identities to prove, we find the computations simpler and the proof easier to follow. That said, it is still interesting to work out the ${(\_)}^\sharp$ for the tangent bundle monad. In general, for a map $f: A \to \T(B)$, $f^\sharp: \T(A) \to \T(B)$ is equal to the composite $\T(f) = \mu_B \circ \T(f)$. In the case of the tangent bundle monad, note that a $f: A \to \T(B)$ would be a pair $f = \langle f_0, f_1 \rangle$, and so we obtain that (which we leave an excercise for the reader to compute for themselves):
\begin{align*}
    f^\sharp = \pair{f_0 \circ \pi_0}{ f_1 \circ \pi_0 + \partial[f_0]+ \varepsilon\left( \partial[f_1] \right) }
\end{align*}

%\begin{exa} \normalfont In the Cartesian difference category $\mathsf{MOD}_R$ (as defined in Section~\ref{moduleex}) with infinitesimal extension $\varepsilon^r$, for $r \in R$, $\mathsf{T}(f)(m,n) = (f(m), f(n))$, $\eta_M(m) = (m,0)$, and $\mu_M((m,n),(p,q)) = (m, n + p + rq)$.
%\end{exa}

We next observe that the tangent bundle functor $\T$ preserves finite products up to isomorphism. Indeed, note that $\T(A \times B) \cong \T(A) \times \T(B)$ via the canonical natural isomorphism:
\begin{equation}\label{phidef1}\begin{gathered} \phi_{A,B}: \T(A \times B)  \to \T(A) \times \T(B) \quad \quad \quad \quad \phi_{A,B} := \pair{T(\pi_0)}{T(\pi_1)}
 \end{gathered}\end{equation}
By~\ref{CdCax3}, it follows that $\phi_{A,B} = \pair{\pi_0 \times \pi_0}{\pi_1 \times \pi_1}$ and its inverse $\phi^{-1}_{A,B}: \T(A) \times \T(B) \to \T(A \times B)$ is defined in the same way, that is, $\phi^{-1}_{A,B} = \pair{\pi_0 \times \pi_0}{\pi_1 \times \pi_1}$. Expanding this out, we can see that $\phi_{A,B}: (A \times B) \times (A \times B) \to (A \times A) \times (B \times B)$ swaps the middle two arguments, and so does $\phi^{-1}_{A,B}$. The following lemma will be extremely useful in many proofs throughout the remainder of this section.

\begin{lem}\label{lem:t-prod} In a Cartesian difference category:
  \begin{enumerate}[(\roman{enumi}),ref={\thelem.\roman{enumi}}]
\item $\mu$, $\eta$, and $\phi$ are linear;
\item\label{T-sum} $\T(f + g) = \T(f) + \T(g)$ and $\T(0) = 0$;
\item If $f$ is linear then $\T(f)$ is linear and $\T(f) = f \times f$;
\item $\T(\pi_i) = \pi_i \circ \phi$ and $\phi \circ \T(\pair{f}{g}) = \pair{\T(f)}{\T(g)}$ and $\phi \circ \T(h \times k) = (\T(h) \times \T(k)) \circ \phi$;
\item\label{T-diff} $\T(\dd[f]) = \dd[\T(f)] \circ \phi$ and $\T(\varepsilon(f)) = \varepsilon(\T(f))$;
\item\label{lem:mu-ep} $\varepsilon(\mu_A) = \mu_A \circ \varepsilon(1_{\T(A)})$, $\varepsilon(\eta_A) = \eta_A \circ \varepsilon(1_A)$, and $\varepsilon(\phi_{A,B}) = \phi_{A,B} \circ \varepsilon(1_{\T(A \times B)})$
\end{enumerate}
\end{lem}
\begin{proof} The first three are mostly straightforward. Indeed, (i) follows immediately from Lemma~\ref{lem:linear} and by construction of $\mu$, $\eta$, and $\phi$, while (ii) follows from~\ref{CdCax1}, and (iii) follows from the definition of $\T$ and linear maps. Next, we compute the identities of (iv), the first two follow from~\ref{CdCax3} and~\ref{CdCax4}. We first show $\T(\pi_i) = \pi_i \circ \phi$:
\begin{align*}
    \T(\pi_i) &=~ \pair{\pi_i \circ \pi_0}{\dd[\pi_i]} \\
    &=~ \pair{\pi_i \circ \pi_0}{\pi_i \circ \pi_1} \tag*{\ref{CdCax3}} \\
    &=~ \pi_i \times \pi_i \\
    &=~ \pi_i \circ \phi \end{align*}
    Next we compute that $\phi \circ \T(\pair{f}{g}) = \pair{\T(f)}{\T(g)}$:
    \begin{align*}
 &\phi \circ \T(\pair{f}{g}) =~ \phi \circ \pair{\pair{f}{g} \circ \pi_0}{\dd[\pair{f}{g}]} \\
 &=~ \phi \circ \four{f \circ \pi_0}{g \circ \pi_0}{\dd[f]}{\dd[g]} \tag*{\ref{CdCax4}} \\
 &=~ \pair{\pi_0 \times \pi_0}{\pi_1 \times \pi_1} \circ \four{f \circ \pi_0}{g \circ \pi_0}{\dd[f]}{\dd[g]} \\
 &=~ \pair{(\pi_0 \times \pi_0) \circ \four{f \circ \pi_0}{g \circ \pi_0}{\dd[f]}{\dd[g]}}{(\pi_1 \times \pi_1) \circ \four{f \circ \pi_0}{g \circ \pi_0}{\dd[f]}{\dd[g]}} \\
 &=~ \pair{\pair{\pi_0 \circ \pair{f \circ \pi_0}{g \circ \pi_0}}{\pi_0 \circ \pair{\dd[f]}{\dd[g]}}}{\pair{\pi_1 \circ \pair{f \circ \pi_0}{g \circ \pi_0}}{\pi_1 \circ\pair{ \dd[f]}{\dd[g]}}} \\
 &=~ \four{f \circ\pi_0}{\dd[f]}{g \circ \pi_0}{\dd[g]} \\
 &=~  \pair{\T(f)}{\T(g)}
\end{align*}
For the remaining identity, $\phi \circ \T(h \times k) = (\T(h) \times \T(k)) \circ \phi$, we use the previous one:
\begin{align*}
    \phi \circ \T(h \times k) &=~ \phi \circ \T(\pair{h \circ \pi_0}{k \circ \pi_1}) \\
    &=~ \pair{\T(h \circ \pi_0)}{\T(k \circ \pi_1)} \tag{Lem.\ref{lem:t-prod}.iv} \\
    &=~ \pair{\T(h) \circ \T(\pi_0)}{\T(k) \circ \T(\pi_1)} \tag{$\T$ is a functor} \\
    &=~ \pair{\T(h) \circ (\pi_0 \times \pi_0)}{\T(k) \circ (\pi_1 \times \pi_1)} \tag{Lem.\ref{lem:t-prod}.iii} \\
    &=~ \pair{\T(h) \circ \pi_0 \circ \phi}{\T(k) \circ \pi_1 \circ \phi} \\
    &=~ \pair{\T(h) \circ \pi_0}{ \T(k) \circ \pi_1} \circ \phi \\
    &=~ (\T(h) \times \T(k)) \circ \phi
\end{align*}
Next we compute the two identities of (v). First we show $\T(\dd[f]) = \dd[\T(f)] \circ \phi$:
\begin{align*}
&\dd[\T(f)] \circ \phi =~ \dd[\pair{f \circ \pi_0}{\dd[f]}] \circ \phi \\
&=~ \pair{\dd[f \circ \pi_0]}{\dd^2[f]} \circ \phi \tag*{\ref{CdCax4}} \\
&=~ \pair{\dd[f] \circ (\pi_0 \times \pi_0)}{\dd^2[f]} \circ \phi \tag{Lem.\ref{lem:linear}} \\
&=~ \pair{\dd[f] \circ (\pi_0 \times \pi_0) \circ \phi}{\dd^2[f] \circ \phi} \\
&=~ \pair{\dd[f] \circ (\pi_0 \times \pi_0) \circ \pair{\pi_0 \times \pi_0}{\pi_1 \times \pi_1}}{\dd^2[f] \circ \pair{\pi_0 \times \pi_0}{\pi_1 \times \pi_1}} \\
&=~ \pair{\dd[f] \circ \pair{\pi_0 (\pi_0 \times \pi_0)}{\pi_0 \circ (\pi_1 \times \pi_1)}}{\dd^2[f] \circ \four{\pi_0 \circ \pi_0}{\pi_0 \circ \pi_1}{\pi_1 \circ \pi_0}{\pi_1 \circ \pi_1}} \\
&=~ \pair{\dd[f] \circ \pair{\pi_1 \circ \pi_0}{\pi_1 \circ \pi_0}}{\dd^2[f] \circ \four{\pi_0 \circ \pi_0}{\pi_0 \circ \pi_1}{\pi_1 \circ \pi_0}{\pi_1 \circ \pi_1}} \\
&=~ \pair{\dd[f] \circ \pair{\pi_0 \circ \pi_0}{\pi_1 \circ \pi_0}}{\dd^2[f] \circ \four{\pi_0 \circ \pi_0}{\pi_1 \circ \pi_0}{\pi_0 \circ \pi_1}{\pi_1 \circ \pi_1}} \tag*{\ref{CdCax7a}}\\
&=~ \pair{\dd[f] \circ \pi_0}{\dd^2[f]} \\
&=~ \T(\dd[f])
\end{align*}
Next we show $\T(\varepsilon(f)) = \varepsilon(\T(f))$:
\begin{align*}
    \T(\varepsilon(f)) &=~ \pair{\varepsilon(f) \circ \pi_0}{\dd[\varepsilon(f)]} \\
    &=~ \pair{\varepsilon(f \circ \pi_0)}{\varepsilon(\dd[f])} \tag*{\ref{Eax2} +~\ref{CdCax1}} \\
    &=~\varepsilon(\pair{f \circ \pi_0}{\dd[f]}) \tag{Lem.\ref{lem:ep-pair}} \\
    &=~ \varepsilon(\T(f))
\end{align*}
Lastly (vi) follows from the fact we have shown that $\mu$, $\eta$, and $\phi$ are linear and so by Lemma~\ref{epsilon-linear}, the desired equalities hold.
\end{proof}

\subsection{The Eilenberg-Moore Category of \texorpdfstring{$\mathsf{T}$}{T}}%
\label{EMsec}

Recall that a $\T$-algebra of the monad $(\mathsf{T}, \mu, \eta)$ is a pair $(A, \nu)$ consisting of an object $A$ and a map $\nu: \T(A) \to A$ such that $\nu \circ \eta_A = 1_A$ and $\nu \circ \T(\nu) = \nu \circ \mu_A$, and that a $\T$-algebra morphism $f: (A, \nu) \to (B, \omega)$ is a map $f: A \to B$ such that $\omega \circ\T(f) = f \circ \nu$. The Eilenberg-Moore category of $(\mathsf{T}, \mu, \eta)$ is the category of $\T$-algebras, that is, the category $\mathbb{X}^\mathsf{T}$ whose objects are $\T$-algebras and whose maps are $\T$-algebra morphisms. It is well known that for any monad on a category with finite products, the Eilenberg-Moore category also has finite products. Indeed, in this case, for $\T$-algebras $(A, \nu)$ and $(B, \omega)$, their product is defined as:
\[ (A, \nu) \times (B, \omega) := (A \times B, (\nu \times \omega) \circ \phi_{A,B}) \]
while the projection maps and the pairing of maps are the same as in the base category. The terminal object in $\mathbb{X}^\mathsf{T}$ is defined as $(\top, 0)$.

Therefore, it may be tempting to think that the Eilenberg-Moore category of the tangent bundle monad is also a Cartesian difference category. Unfortunately, there are two problems with this: the sum of a $\T$-algebra morphism is not necessarily a $\T$-algebra morphism and the differential of a $\T$-algebra may not necessarily be a $\T$-algebra. Indeed, let $f,g: (A, \nu) \to (B, \omega)$ be $\T$-algebra morphism. On the one hand, since $\T(f + g) = \T(f) + \T(g)$ we have that:
\[ \omega \circ \T(f+ g) = \omega \circ (\T(f) + \T(g)) = \omega \circ \T(f) + \omega \circ \T(g) = f \circ \nu + g \circ \nu  \]
However, if $\nu$ is not additive then $f \circ \nu + g \circ \nu$ may not be equal to $(f + g) \circ \nu$. As such, one solution could be to consider the subcategory of additive $\T$-algebras. While this full subcategory of additive $\T$-algebras will be a Cartesian left additive category, this does not solve the problem of the differential of a $\T$-algebra morphism. For a $\T$-algebra morphism $f: (A, \nu) \to (B, \omega)$, its derivative $\dd[f]: A \times A \to B$ should be a $\T$-algebra morphism of type $(A, \nu) \times (A, \nu) \to (B, \omega)$. However one cannot get very far in trying to show that $\dd[f]$ is a $\T$-algebra morphism. The solution to this problem is instead to consider the full subcategory of \emph{linear} $\T$-algebras.

Define $\mathbb{X}^{\T}_{lin}$ as the category of \emph{linear} $\T$-algebra, that is, the category whose objects are $\T$-algebras $(A, \nu)$ such that $\nu$ is linear and whose maps are arbitrary $\T$-algebra morphisms between them. In particular, note that a map in $\mathbb{X}^{\T}_{lin}$ need not be linear. Like the Eilenberg-Moore category, there is an obvious forgetful functor $\mathsf{U}: \mathbb{X}^{\T}_{lin} \to \mathbb{X}$ defined as $\mathsf{U}(A, \nu) = A$ and $\mathsf{U}(f) = f$. A linear $\T$-algebra structure should be interpreted as a map which linearly modifies a point by a tangent vector. Unfortunately, as mentioned in~\cite{cockett2014differential}, the Eilenberg-Moore category of the tangent bundle monad has yet to be studied in full, even in classical differential geometry.

\begin{exa}
For the tangent bundle monad on the Cartesian differential category $\mathsf{SMOOTH}$ (Example~\ref{ex:tsmooth}), a linear $\T$-algebra is a pair $(\mathbb{R}^n, \nu)$ such that $\nu(\vec x, \vec y) = \vec x + t \vec y$ for some fixed $t \in \mathbb{R}$. Therefore, $\mathsf{SMOOTH}^{\T}_{lin}$ is equivalent to the category whose objects are pairs $(\mathbb{R}^n, t)$, with $t \in \mathbb{R}$, where a map $F: (\mathbb{R}^n, t) \to (\mathbb{R}^m, s)$ is a smooth function $F: \mathbb{R}^n \to \mathbb{R}^m$ such that $F(\vec x + t \vec y) = F(\vec x) + s \mathsf{D}[F](\vec x, \vec y)$.
\end{exa}

\begin{exa}
For the tangent bundle monad on the Cartesian differential category $\overline{\mathsf{Ab}}$ (Example~\ref{ex:tab}), a linear $\T$-algebra is a pair $(G, \nu)$ such that $\nu(x,y) = x + e(y)$ for some group endomorphism $e: G \to G$. Therefore, $\overline{\mathsf{Ab}}$ is equivalent to the category whose objects are pairs $(G,e)$, with $G$ an abelian group and $e: G \to G$ a group endomorphism, where a map $f: (G, e) \to (H, e^\prime)$ is a function $f: G \to H$ such that \[f( x + e(y) ) = f(x) + e^\prime \left( f(x+y) \right) - e^\prime(f(x)).\]
\end{exa}

We will now explain how $\mathbb{X}^{\T}_{lin}$ is a Cartesian difference category. Simply put, the Cartesian difference structure of $\mathbb{X}^{\T}_{lin}$ is the same as $\mathbb{X}$, and therefore we say that the forgetful functor preserves the Cartesian difference structure strictly. Starting with the Cartesian left additive structure: the finite product structure of $\mathbb{X}^{\T}_{lin}$ is defined in the same way as $\mathbb{X}^{\T}$, while the sum of $\T$-algebra morphisms and zero $\T$-algebra morphisms are defined as the sum and zero maps in $\mathbb{X}$. Of course, we must check that this is well-defined.

\begin{lem} $\mathbb{X}^{\T}_{lin}$ is a Cartesian left additive category.
\end{lem}
\begin{proof} We must first check that if $(A, \nu)$ and $(B, \omega)$ are linear $\T$-algebras, then so is their product $(A, \nu) \times (B, \omega)$, that is, we must show that $(\nu \times \omega) \circ \phi_{A,B}$ is linear. However, since $\nu$, $\omega$, and $\phi_{A,B}$ are linear, it follows from Lemma~\ref{comp-linear} that $(\nu \times \omega) \circ \phi_{A,B}$ is also linear. Therefore, the product of linear $\T$-algebras is also a linear $\T$-algebra. Next, since zero maps are linear, $(\top, 0)$ is also a linear $\T$-algebra. Therefore $\mathbb{X}^{\T}_{lin}$ has finite products since it is a full subcategory of $\mathbb{X}^\mathsf{T}$. Next, we show that if $f: (A, \nu) \to (B, \omega)$ and $g: (A, \nu) \to (B, \omega)$ are $\T$-algebra morphisms between linear $\T$-algebras, then their sum $f +g: (A, \nu) \to (B, \omega)$ is also a $\T$-algebra morphism. This follows from the fact that linear maps are additive (Lemma~\ref{add-linear}):
\begin{align*}
\omega \circ \T(f+ g) &=~ \omega \circ (\T(f) + \T(g))  \tag{by Prop.\ref{T-sum}} \\
&=~ \omega \circ \T(f) + \omega \circ \T(g)\\
&=~ f \circ \nu + g \circ \nu \tag{by $\T$-alg. morph. def.} \\
&=~ (f + g) \circ \nu  \tag{$\nu$ is additive}
\end{align*}
So we have that $f+g$ is a $\T$-algebra morphism. Similarly, we must show that zero maps between linear $\T$-algebras $(A, \nu)$ and $(B, \omega)$ are $\T$-algebra morphisms:
\begin{align*}
\omega \circ \T(0) &=~ \omega \circ 0  \tag{by Prop.\ref{T-sum}} \\
&=~ 0 \\
&=~ 0 \circ \nu \tag{$\nu$ is additive}
\end{align*}
Therefore $0: (A, \nu) \to (B, \omega)$ is a $\T$-algebra morphism. And clearly, the remaining Cartesian left additive structure axioms hold since composition in $\mathbb{X}^{\T}_{lin}$ is the same as in $\mathbb{X}$. So we conclude that $\mathbb{X}^{\T}_{lin}$ is a Cartesian left additive category.
\end{proof}

Similarly, the infinitesimal extension and the difference combinator of $\mathbb{X}^{\T}_{lin}$ are defined in the same way as in $\mathbb{X}$. Once again, while the infinitesimal extension and difference combinator axioms will automatically hold, we will have to check that this is all well-defined.

 \begin{prop} For a Cartesian difference category $\mathbb{X}$, the category of linear $\T$-algebras $\mathbb{X}^{\T}_{lin}$ is a Cartesian difference category such that the obvious forgetful functor $\mathsf{U}: \mathbb{X}^{\T}_{lin} \to \mathbb{X}$ preserves the Cartesian difference structure strictly.
\end{prop}
\begin{proof} We first show that if $f: (A, \nu) \to (B, \omega)$ is a $\T$-algebra morphism between linear $\T$-algebras, then so is $\varepsilon(f)$:
\begin{align*}
\omega \circ \T(\varepsilon(f)) &=~\omega \circ \varepsilon(\T(f))  \tag{by Prop.\ref{T-diff}} \\
&=~ \varepsilon(\omega \circ \T(f)) \tag{by Lem.\ref{epsilon-linear}} \\
&=~ \varepsilon(f \circ \nu)  \tag{by $\T$-alg. morph. def.} \\
&=~ \varepsilon(f) \circ \nu \tag{by $\varepsilon$ def.}
\end{align*}
So we have that $\varepsilon(f): (A, \nu) \to (B, \omega)$ is a $\T$-algebra morphism. Next, we check that the derivative $\dd[f]: A \times A \to B$ is a $\T$-algebra morphism of type $(A, \nu) \times (A, \nu) \to (B, \omega)$.
\begin{align*}
\omega \circ \T(\dd[f]) &=~\omega \circ \dd[\T(f)] \circ \phi_{A,A}  \tag{by Prop.\ref{T-diff}} \\
&=~ \dd[\omega \circ \T(f)] \circ \phi_{A,A} \tag{by Lem.\ref{chain-linear} since $\omega$ is linear} \\
&=~ \dd[f \circ \nu] \circ \phi_{A,A}  \tag{by $\T$-alg. morph. def.} \\
&=~ \dd[f] \circ (\nu \times \nu) \circ \phi_{A,A} \tag{by Lem.\ref{chain-linear} since $\nu$ is linear}
\end{align*}
So we have that $\dd[f]: (A, \nu) \times (A, \nu) \to (B, \omega)$ is a $\T$-algebra morphism. Since composition and the Cartesian left additive structure of $\mathbb{X}^{\T}_{lin}$ is the same as $\mathbb{X}$, it automatically follows that $\varepsilon$ is an infinitesimal extension and $\dd$ is a difference combinator on $\mathbb{X}^{\T}_{lin}$. Therefore, we conclude that $\mathbb{X}^{\T}_{lin}$ is a Cartesian difference category, and clearly the forgetful functor preserves the Cartesian difference structure strictly.
\end{proof}

\subsection{The Kleisli Category of \texorpdfstring{$\mathsf{T}$}{T}}%
\label{Kleislisec}
The construction found here is different from the one found in the conference paper~\cite{alvarez2020cartesian}. Indeed, the proposed infinitesimal extension and difference combinator in~\cite{alvarez2020cartesian} were based on the ones that appeared in~\cite{alvarez2019change}. Unfortunately, we have found that said proposed infinitesimal extension and difference combinator failed to satisfy {\bf [C$\dd$.2]} and therefore both of the aforementioned results in~\cite{alvarez2020cartesian,alvarez2019change} are incorrect. We rectify this mistake here by changing the infinitesimal extension to the correct one, while keeping the difference combinator the same.

Recall that the Kleisli category of the monad $(\mathsf{T}, \mu, \eta)$ is defined as the category $\mathbb{X}_\mathsf{T}$ whose objects are the objects of $\mathbb{X}$, and where a map $A \to B$ in $\mathbb{X}_\mathsf{T}$ is a map $f: A \to \mathsf{T}(B)$ in $\mathbb{X}$, which would be a pair $f = \langle f_0, f_1 \rangle$ where $f_j: A \to B$. The identity map in $\mathbb{X}_\mathsf{T}$ is the monad unit $\eta_A: A \to \mathsf{T}(A)$, while composition of Kleisli maps $f: A \to \mathsf{T}(B)$ and $g: B \to \mathsf{T}(C)$ is defined as the composite $\mu_C \circ \mathsf{T}(g) \circ f$. To distinguish between composition in $\mathbb{X}$ and $\mathbb{X}_\mathsf{T}$, we denote the Kleisli composition as follows:
\begin{equation}\label{Kcomp1}\begin{gathered}  g \circ^\T f = \mu_C \circ \mathsf{T}(g) \circ f
 \end{gathered}\end{equation}
If $f=\langle f_0, f_1 \rangle$ and $g=\langle g_0, g_1 \rangle$, then their Kleisli composition can be worked out to be:
\begin{equation}\label{Kcomp2}\begin{gathered}  g \circ^\T f = \langle g_0, g_1 \rangle \circ^\T \langle f_0, f_1 \rangle = \left \langle g_0 \circ f_0, \dd[g_0] \circ \langle f_0, f_1 \rangle + g_1 \circ (f_0 \oplus f_1) \right \rangle  \end{gathered}\end{equation}
Kleisli maps can be understood as ``generalized'' vector fields. Indeed, $\T(A)$ should be thought of as the tangent bundle over $A$, and therefore a vector field would be a map $\langle 1, f \rangle: A \to \T(A)$, which is of course also a Kleisli map. For more details on the intuition behind this Kleisli category see~\cite{cockett2014differential}. Furthermore, in general, the Kleisli category of a monad is equivalently to a full subcategory of the Eilenberg-Moore consisting of the free algebras. In this case, note that every free $\T$-algebra $(\T(A), \mu_A)$ is also an object in $\mathbb{X}^{\T}_{lin}$, since $\mu_A$ is linear. Therefore, $\mathbb{X}_\mathsf{T}$ is equivalent to the full subcategory of free $\T$-algebras of $\mathbb{X}^{\T}_{lin}$.

\begin{exa}
For the tangent bundle monad on the Cartesian differential category $\mathsf{SMOOTH}$ (Example~\ref{ex:tsmooth}), a Kleisli map is a smooth function $F: \mathbb{R}^n \to \mathbb{R}^m \times \mathbb{R}^m$, which is interpreted as a pair of smooth functions $F = \langle F_0, F_1 \rangle$ where $F_i: \mathbb{R}^n \to \mathbb{R}^m$. The composition of Kleisli maps $F = \langle F_0, F_1 \rangle$ and $G = \langle G_0, G_1 \rangle$ is computed out to be $(G \circ^\T F)(\vec x) = ( G_0(F_0(\vec x)), \mathsf{D}[G_0](F_0(\vec x), F_1(\vec x)) + G_1(F_0(\vec x))$. Among these Kleisli maps are the vector fields on $\mathbb{R}^n$, which are precisely the Kleisli maps of the form $F = \langle 1_{\mathbb{R}^n}, f \rangle$, for some smooth function $f: \mathbb{R}^n \to \mathbb{R}^n$.
\end{exa}

\begin{exa}
For the tangent bundle monad on the Cartesian differential category $\overline{\mathsf{Ab}}$ (Example~\ref{ex:tab}), a Kleisli map is a function $f: G \to H \times H$, which is interpreted as a pair of functions $f = \langle f_0, f_1 \rangle$ where $f_i: G \to H$. In this case, the composition of Kleisli maps $f = \langle f_0, f_1 \rangle$ and $g = \langle g_0, g_1 \rangle$ is $(g \circ^\T g)(x) = ( g_0(g_0(x)), g_0\left( f_0(x) + f_1(x) \right) - g_0(f_0(x)) + g_1(f_0(x))$.
\end{exa}

We now wish to explain how the Kleisli category $\mathbb{X}_\mathsf{T}$ is again a Cartesian difference category. We begin by exhibiting the Cartesian left additive structure of the Kleisli category $\mathbb{X}_\mathsf{T}$. Generally, the Kleisli category does not automatically inherit the product structure of the base category, even if the Eilenberg-Moore category does. However, since $\T$ preserves finite products up to isomorphism, it follows that its Kleisli category has finite products. As such, the product of objects in $\mathbb{X}_\mathsf{T}$ is defined as $A \times B$ with projections $\pi^{\mathsf{T}}_0: A \times B \to \mathsf{T}(A)$ and $\pi^{\mathsf{T}}_1: A \times B \to \mathsf{T}(B)$ defined respectively as $\pi^{\mathsf{T}}_0 = \langle \pi_0, 0 \rangle$ and $\pi^{\mathsf{T}}_1 = \langle \pi_1, 0 \rangle$, and the pairing of Kleisli maps $f=\langle f_0, f_1 \rangle$ and $g=\langle g_0, g_1 \rangle$ is defined as:
\begin{equation}\label{Kpair}\begin{gathered}  \langle f, g \rangle^\mathsf{T} = \phi^{-1} \circ \pair{f}{g} = \four{f_0}{g_0}{f_1}{g_1} \end{gathered}\end{equation}
where recall $\phi^{-1}:\T(A) \times\T(B) \to\T(A \times B)$ is the inverse of $\phi$ as defined in (\ref{phidef1}). The terminal object is again $\top$ and where the unique map to the terminal object is $!^{\mathsf{T}}_A = 0$. The sum of Kleisli maps $f=\langle f_0, f_1 \rangle$ and $g=\langle g_0, g_1 \rangle$ is defined as:
\[f +^\mathsf{T} g = f + g = \langle f_0 + g_0, f_1 + g_1 \rangle\]
and the zero Kleisli maps is simply $0^\T = 0 = \langle 0, 0 \rangle$. Therefore we conclude that the Kleisli category of the tangent monad is a Cartesian left additive category.

\begin{lem} $\mathbb{X}_\mathsf{T}$ is a Cartesian left additive category.
\end{lem}
\begin{proof} As explained above, by Lemma~\ref{lem:t-prod}, $\T$ preserves the finite product structure, and thus it follows that the Kleisli category $\mathbb{X}_\mathsf{T}$ is a Cartesian category. So it remains to show that $\mathbb{X}_\mathsf{T}$ is a left additive category and that the projection maps are additive. We start by showing that the proposed additive structure is compatible with Kleisli composition. This follows from the fact that $\T$ preserves the additive structure by Lemma~\ref{lem:t-prod}.(ii) and Lemma~\ref{add-linear}, that $\mu$ is linear and therefore also additive. % chktex 36
\begin{align*}
(f +^\T g) \circ^\T x &=~ (f+g) \circ^\T x \\
&=~ \mu \circ \T(f +g) \circ x \\
&=~ \mu \circ (\T(f) + \T(g) ) \circ x \tag{Lem\ref{lem:t-prod}.ii} \\
&=~ \mu \circ (\T(f) \circ x + \T(g) \circ x) \\
&=~ \mu \circ \T(f) \circ x + \mu \circ \T(g) \circ x \tag{$\mu$ is additive} \\
&=~ f \circ^\T x + g \circ^\T x \\
&=~ f \circ^\T x +^\T g \circ^\T x \\ \\
0^\T \circ x &=~ 0 \circ^\T x \\
&=~ \mu \circ\T(0) \circ x \\
&=~ \mu \circ 0 \circ x \tag{Lem\ref{lem:t-prod}.ii} \\
&=~0 \circ x \tag{$\mu$ is additive} \\
&=~ 0
\end{align*}
So we have that $\mathbb{X}_\mathsf{T}$ is a left additive category. Next we show that the projection maps are additive. Note that $\pi_i^\T = \eta \circ \pi_i$ and that by Lemma\ref{lem:t-prod}.(iii), $\T(\pi_i) = \pi_i \times \pi_i$ is linear and therefore also additive. So we have that: % chktex 36
\begin{align*}
    \pi^\T_i \circ^\T (x + y) &=~ (\eta \circ \pi_i) \circ^\T (x+y) \\
    &=~ \mu \circ \T(\eta \circ \pi_i) \circ (x+y) \\
    &=~ \mu \circ \T(\eta) \circ \T(\pi_i) \circ (x+y) \tag{$\T$ is a functor} \\
    &=~ \T(\pi_i) \circ (x+y) \tag{Monad identities} \\
    &=~ \T(\pi_i) \circ x + \T(\pi_i) \circ y \tag{$\T(\pi_i)$ is additive} \\
 &=~\mu \circ \T(\eta) \circ \T(\pi_i) \circ x + \mu \circ \T(\eta) \circ \T(\pi_i) \circ y \tag{Monad identities} \\
 &=~ \mu \circ \T(\eta \circ \pi_i) \circ x + \mu \circ \T(\eta \circ \pi_i) \circ y \tag{$\T$ is a functor} \\
 &=~ (\eta \circ \pi_i) \circ^\T x  + (\eta \circ \pi_i) \circ^\T y \\
 &=~ \pi^\T_i \circ^\T x + \pi^\T_i \circ^\T y \\ \\
 \pi^\T_i \circ^\T 0^\T &=~ (\eta \circ \pi_i) \circ^\T 0 \\
 &=~ \mu \circ \T(\eta \circ \pi_i) \circ 0 \\
 &=~ \mu \circ \T(\eta) \circ \T(\pi_i) \circ 0 \tag{$\T$ is a functor} \\
 &=~ \T(\pi_i) \circ 0 \tag{Monad identities} \\
 &=~ 0 \tag{$\T(\pi_i)$ is additive}
\end{align*}
So we conclude that $\mathbb{X}_\mathsf{T}$ is a Cartesian left additive category.
\end{proof}

The infinitesimal extension $\varepsilon^\T$ for the Kleisli category is the same as the infinitesimal extension of the base category, that is, for a Kleisli map $f= \langle f_0, f_1 \rangle$:
\[ \varepsilon^\T(f) = \varepsilon(f) = \varepsilon(\pair{f_0}{f_1}) =  \pair{\varepsilon(f_0)}{\varepsilon(f_1)}\]
where the last equality is due to Lemma~\ref{lem:ep-pair}. It is important to note that this is the major correction made from the conference paper version. Indeed, the infinitesimal extension suggested in~\cite{alvarez2020cartesian} was based on the change action structure suggested in~\cite{alvarez2019change}. However, it unfortunately turns out that a result in~\cite{alvarez2019change} was incorrect, since~\ref{CADax2} actually fails, and therefore the infinitesimal extension and difference combinator suggested in~\cite{alvarez2020cartesian} do not satisfy~\ref{CdCax2}. Luckily, the infinitesimal extension provide in this paper does work, as we carefully prove below.

\begin{lem} $\varepsilon^\T$ is an infinitesimal extension on $\mathbb{X}_\mathsf{T}$.
\end{lem}
\begin{proof} We check that $\varepsilon^\T$ satisfies~\ref{Eax1},~\ref{Eax2}, and~\ref{Eax3}. For~\ref{Eax1}, we compute:
\begin{align*}
\varepsilon^\T(f +^\T g) &=~ \varepsilon^\T(f+g) \\
&=~ \varepsilon(f+g) \\
&=~ \varepsilon(f) + \varepsilon(g) \tag*{\ref{Eax1}} \\
&=~ \varepsilon(f) +^\T \varepsilon(g) \\
&=~ \varepsilon^\T(f) +^\T \varepsilon^\T(g) \\ \\
\varepsilon^\T(0^\T) &=~ \varepsilon(0^\T) \\
&=~ \varepsilon(0) \\
&=~ 0 \tag*{\ref{Eax1}} \\
&=~ 0^\T
\end{align*}
Next for~\ref{Eax2}, we compute:
\begin{align*}
\varepsilon^\T(g \circ^\T f) &=~ \varepsilon(g \circ^\T f) \\
&=~ \varepsilon(\mu \circ \T(g) \circ f) \\
&=~ \varepsilon(\mu) \circ \T(g) \circ f \tag*{\ref{Eax2}} \\
&=~ \mu \circ \varepsilon(1) \circ \T(g) \circ f \tag{Lem.\ref{lem:t-prod}.(vi)}\\
&=~ \mu \circ \varepsilon(\T(g)) \circ f \\
&=~ \mu \circ \T(\varepsilon(g)) \circ f \tag{Lem.\ref{lem:t-prod}.(v)} \\
&=~ \varepsilon(g) \circ^\T f \\
&=~ \varepsilon^\T(g) \circ^\T f
\end{align*}
Lastly for~\ref{Eax3}, recall again that $\pi_i^\T = \eta \circ \pi_i$, so we can compute:
\begin{align*}
    \varepsilon^\T(\pi_i^\T) &=~ \varepsilon(\pi_i^\T) \\
    &=~ \varepsilon(\eta \circ \pi) \\
    &=~ \varepsilon(\T(\pi_i) \circ \eta) \tag{Nat. of $\eta$} \\
    &=~ \varepsilon(\T(\pi_i)) \circ \eta \tag*{\ref{Eax2}} \\
    &=~ \T(\varepsilon(\pi_i)) \circ \eta \tag{Lem.\ref{lem:t-prod}.(v)} \\
    &=~ \T(\pi_i \circ \varepsilon(1)) \circ \eta\tag*{\ref{Eax3}} \\
    &=~ \T(\pi_i) \circ \T(\varepsilon(1)) \circ \eta \tag{$\T$ is a functor} \\
    &=~ \T(\pi_i) \circ \eta \circ \varepsilon(1) \tag{Nat. of $\eta$} \\
    &=~ \T(\pi_i) \circ \varepsilon(\eta) \tag{Lem.\ref{lem:t-prod}.(vi)}\\
    &=~ \mu \circ \T(\eta) \circ \T(\pi_i) \circ \varepsilon(\eta) \tag{Monad identities} \\
       &=~ \mu \circ \T(\eta \circ \pi_i) \circ \varepsilon(\eta) \tag{$\T$ is a functor} \\
    &=~ (\eta \circ \pi_i) \circ^\T \varepsilon(\eta) \\
    &=~ (\eta \circ \pi_i) \circ^\T \varepsilon^\T(\eta) \\
    &=~ \pi_i^\T \circ^\T \varepsilon^\T(\eta)
\end{align*}
So we conclude that $\varepsilon^\T$ is an infinitesimal extension on $\mathbb{X}_\mathsf{T}$.
\end{proof}

To define the difference combinator for the Kleisli category, first note that difference combinators by definition do not change the codomain. That is, if $f : A \to \T(B)$ is a Kleisli map, then the type of its derivative should be $A \times A \to \T(B)$, which coincides with the type of its derivative in $\mathbb{X}$. Therefore, the difference combinator $\dd^\T$ for the Kleisli category is defined to be the difference combinator of the base category, that is, for a Kleisli map $f= \langle f_0, f_1 \rangle$, its derivative is defined as:
\begin{equation}\label{Kd}\begin{gathered} \dd^\T[f] = \dd[f] = \langle \dd[f_0], \dd[f_1] \rangle
 \end{gathered}\end{equation}
 where the last equality is due to~\ref{CdCax4}. We note that this difference combinator was the one suggested in~\cite{alvarez2020cartesian} and is the derivative underlying the change action model structure suggested in~\cite{alvarez2019change}. However, as mentioned above, the difference combinator with the infinitesimal extension or change action suggested in those papers does not satisfy~\ref{CdCax2} or~\ref{CADax2}. Luckily, we are able to correct this and still obtain a positive result by carefully proving that the proposed infinitesimal extension and difference combinator in this paper does provide a Cartesian difference structure on the Kleisli category.

 \begin{prop} For a Cartesian difference category $\mathbb{X}$, the Kleisli category $\mathbb{X}_\mathsf{T}$ is a Cartesian difference category with infinitesimal extension $\varepsilon^\T$ and difference combinator $\dd^\T$.
\end{prop}
\begin{proof}
We first note that we can easily compute the following:
  \begin{align*}
\dd^T[f] \circ^\T \pair{x}{y}^\T &=~ \mu \circ \T(\dd^T[f]) \circ \pair{x}{y}^\T \\
&=~  \mu \circ \T(\dd[f]) \circ \phi^{-1} \circ \pair{x}{y} \\
&=~ \mu \circ  \dd[{\T(f)}] \circ \phi \circ \phi^{-1} \circ \pair{x}{y}  \tag{by Prop.\ref{T-diff}} \\
&=~  \mu \circ  \dd[{\T(f)}] \circ \pair{x}{y}
\end{align*}
So we have that $\dd^T[f] \circ^\T \pair{x}{y}^\T = \mu \circ  \dd[{\T(f)}] \circ \pair{x}{y}$.
This will help simplify many of the calculations to follow, since
  $\T(\dd[f])$ appears everywhere due to the definition of Kleisli composition. We now prove the Cartesian difference category axioms.
\\\\
\noindent~\ref{CdCax0} $f \circ^\T (x +^\T \varepsilon^\T(y)) = f \circ^\T x +^\T \varepsilon^\T\left( \dd^\T[f] \circ^\T \langle x, y \rangle^\T \right)$ \\\\
    First note that $\mu$ is linear and therefore additive. Then we compute that:
     \begin{align*}
&f \circ^\T x +^\T \varepsilon^\T\left( \dd^\T[f] \circ^\T \langle x, y \rangle^\T \right) =~ f \circ^\T x + \varepsilon\left( \dd^\T[f] \circ^\T \langle x, y \rangle^\T \right) \\
&=~ \mu \circ \T(f) \circ x + \varepsilon\left( \mu \circ  \dd[{\T(f)}] \circ \pair{x}{y} \right) \\
&=~ \mu \circ \T(f) \circ x + \varepsilon\left( \mu \right) \circ  \dd[{\T(f)}] \circ \pair{x}{y} \tag*{\ref{Eax2}} \\
&=~ \mu \circ \T(f) \circ x + \mu \circ \varepsilon(1_{\T(A)}) \circ  \dd[{\T(f)}] \circ \pair{x}{y} \tag{Lemma~\ref{lem:mu-ep}}\\
&=~ \mu \circ \T(f) \circ x + \mu \circ \varepsilon\left( \dd[{\T(f)}] \circ \pair{x}{y} \right) \tag{Lemma~\ref{lem:epbij}} \\
&=~ \mu \circ \left(\T(f) \circ x + \varepsilon\left( \dd[{\T(f)}] \circ \pair{x}{y} \right) \right) \tag{$\mu$ is additive} \\
&=~\mu \circ \T(f) \circ (x + \varepsilon(y))  \tag*{\ref{CdCax0}} \\
&=~ f \circ^\T (x + \varepsilon(y)) \\
&=~f \circ^\T (x +^\T \varepsilon^\T(y))
\end{align*}
\\
\noindent~\ref{CdCax1}  $\dd^\T[f +^\T g] = \dd^\T[f] +^\T \dd^\T[g]$, $[\dd^\T[0^\T] =  \dd[0]$, and $\dd^\T[\varepsilon^\T(f)] = \varepsilon^\T\left( \dd^\T[f] \right)$ \\\\
    Since both the sum, zero maps, infinitesimal extension, and differential combinator in the Kleisli category are the same as in the base category, by~\ref{CdCax1} it easily follows that:
    \begin{align*}
\dd^\T[f +^\T g] = \dd[f+g] = \dd[f] + \dd[g] = \dd^\T[f] +^\T \dd^\T[g]
\end{align*}
\[\dd^\T[0^\T] = \dd[0] = 0 = 0^\T\]
\[\dd^\T[\varepsilon^\T(f)] = \dd[\varepsilon(f)] =  \varepsilon\left( \dd[f] \right) =  \varepsilon^\T\left( \dd^\T[f] \right) \]
\\
\noindent~\ref{CdCax2} $\dd^\T[f] \circ^\T \langle x, y +^\T z \rangle^\T =  \dd^\T[f] \circ^\T \langle x, y \rangle^\T + \dd^\T[f] \circ^\T \langle x + \varepsilon^\T(y), z \rangle^\T$\\
and ${\dd^\T[f] \circ^\T \langle x, 0^\T \rangle^\T = 0^\T}$
    \begin{align*}
\dd^\T[f] \circ^\T \pair{x}{y +^\T z}^\T &=~\mu \circ \dd\left[ \T(f) \right] \circ \pair{x}{y +^\T z} \\
&=~\mu \circ \dd\left[ \T(f) \right] \circ \pair{x}{y + z} \\
&=~\mu \circ \left( \dd\left[ \T(f) \right] \circ \pair{x}{y} +  \dd\left[ \T(f) \right] \circ \pair{x + \varepsilon(y)}{z} \right) \tag*{\ref{CdCax2}} \\
&=~\mu \circ \dd\left[ \T(f) \right] \circ \pair{x}{y} + \mu \circ \dd\left[ \T(f) \right] \circ \pair{x + \varepsilon(y)}{z} \tag{$\mu$ is additive} \\
&=~ \mu \circ \dd\left[ \T(f) \right] \circ \pair{x}{y} + \mu \circ \dd\left[ \T(f) \right] \circ \pair{x +^\T \varepsilon^\T(y)}{z} \\
&=~ \dd^\T[f] \circ^\T \langle x, y \rangle^\T + \dd^\T[f] \circ^\T \langle x + \varepsilon^\T(y), z \rangle^\T \\\\
\dd^\T[f] \circ^\T \pair{x}{0^\T}^\T &=~ \dd^\T[f] \circ^\T \pair{x}{0}^\T \\
&=~\mu \circ \dd\left[ \T(f) \right] \circ \pair{x}{0} \\
&=~ 0 \tag*{\ref{CdCax2}} \\
&=~ 0^\T
\end{align*}
\\
\noindent~\ref{CdCax3} $\dd^\T[\eta] = \pi_1^\T$ and $\dd^\T[\pi^\T_i] =\pi_i^\T\circ^\T \pi_1^\T$ \\\\
Recall that $\pi_i^\T = \eta \circ \pi_i$ and so $\pi_i^\T\circ^\T \pi_1^\T = \eta \circ \pi_i \circ \pi_1$. Then since $\eta$ is linear, we have that:
\begin{align*}
\dd^\T[\eta] &=~ \dd[\eta] \\
&=~ \eta \circ \pi_1 \tag{$\eta$ is linear} \\
&=~ \pi_1^\T \\\\
\dd^\T[\pi^\T_i] &=~ \dd[\pi^\T_i] \\
&=~ \dd[\eta \circ \pi_i] \\
&=~\eta \circ \dd[\pi_i] \tag{$\eta$ linear and Lem.\ref{chain-linear}} \\
&=~\eta \circ \pi_i \circ \pi_1 \tag*{\ref{CdCax3}}
\end{align*}
\\
\noindent~\ref{CdCax4} $\dd^\T[\langle f, g \rangle^\T] = \langle \dd^\T[f], \dd^\T[g] \rangle^\T$ \\\\
First note that since $\phi$ is linear, so it its inverse $\phi^{-1}$. Then we have that:
 \begin{align*}
\dd^\T[\langle f, g \rangle^\T] &=~ \dd[\langle f, g \rangle^\T] \\
&=~ \dd\left[\phi^{-1} \circ \langle f, g \rangle \right] \\
&=~ \phi^{-1} \circ \dd\left[ \langle f, g \rangle \right] \tag{$\phi^{-1}$ linear and Lem.\ref{chain-linear}} \\
&=~ \phi^{-1} \circ \langle \dd[f], \dd[g] \rangle  \tag*{\ref{CdCax4}} \\
&=~  \langle \dd[f], \dd[g] \rangle^\T \\
&=~ \langle \dd^\T[f], \dd^\T[g] \rangle^\T
\end{align*}
\\
\noindent~\ref{CdCax5}  $\dd^\T[g \circ^\T f] = \dd^\T[g] \circ^\T \langle f \circ^\T \pi_0^\T, \dd^\T[f] \rangle^\T$ \\\\
First note that since $\pi_0^\T = \eta \circ \pi_0$, it easily follows that $f \circ^\T \pi_0^\T = f \circ \pi_0$ (using the monad identities and the naturality of $\eta$). Therefore, we compute that:

\begin{align*}
 \dd^\T[g] \circ^\T \langle f \circ^\T \pi_0^\T, \dd^\T[f] \rangle^\T &=~ \mu \circ \dd\left[ \T(g) \right] \circ \langle f \circ^\T \pi_0^\T, \dd^\T[f] \rangle \\
 &=~\mu \circ \dd\left[ \T(g) \right] \circ \langle f \circ \pi_0, \dd[f] \rangle \\
    &=~ \mu \circ \dd[\T(g) \circ f]  \tag*{\ref{CdCax5}} \\
    &=~ \dd\left[\mu \circ\T(g) \circ f \right]  \tag{$\mu$ linear and Lem.\ref{chain-linear}} \\
    &=~\dd[g \circ^\T f] \\
    &=~ \dd^\T[g \circ^\T f]
\end{align*}
For the remaining two axioms, we will instead prove~\ref{CdCax6a} and~\ref{CdCax7a}. Before we do so, we first compute the following:
\begin{align*}
&\dd^\T\left[\dd^\T[f] \right] \circ^\T \left \langle \langle  x, y \rangle^\T, \langle z,w \rangle^\T \right \rangle^\T =~ \mu \circ \dd\left[ \T(\dd^\T[f]) \right] \circ  \left \langle \langle  x, y \rangle^\T, \langle z,w \rangle^\T \right \rangle \\
&=~ \mu \circ \dd\left[ \T(\dd[f]) \right] \circ  \left \langle \phi^{-1} \circ \langle  x, y \rangle, \phi^{-1} \circ \langle z,w \rangle \right \rangle \\
&=~ \mu \circ \dd\left[ \T(\dd[f]) \right] \circ (\phi^{-1} \times \phi^{-1}) \circ \four{x}{y}{z}{w} \\
&=~ \mu \circ \dd\left[ \T(\dd[f]) \circ \phi^{-1} \right]\circ \four{x}{y}{z}{w} \tag{$\phi^{-1}$ linear and Lem.\ref{chain-linear}} \\
&=~ \mu \circ \dd\left[ \dd[\T(f)] \circ \phi \circ \phi^{-1} \right]\circ \four{x}{y}{z}{w}  \tag{by Prop.\ref{T-diff}} \\
&=~\mu \circ \dd^2\left[ \T(f) \right] \circ \four{x}{y}{z}{w}
\end{align*}
Therefore we have that $\dd^\T\left[\dd^\T[f] \right] \circ^\T \left \langle \langle  x, y \rangle^\T, \langle z,w \rangle^\T \right \rangle^\T = \mu \circ \dd^2\left[ \T(f) \right] \circ \four{x}{y}{z}{w}$. \\\\
\noindent~\ref{CdCax6a}  $\dd^\T\left[\dd^\T[f] \right] \circ^\T \left \langle \langle  x, 0^\T \rangle^\T, \langle 0^\T, y \rangle^\T \right \rangle^\T = \dd^\T[f] \circ^\T  \langle x, y \rangle^\T$
\begin{align*}
\dd^\T\left[\dd^\T[f] \right] \circ^\T \left \langle \langle  x, 0^\T \rangle^\T, \langle 0^\T, y \rangle^\T \right \rangle^\T &=~ \mu \circ \dd^2 \left[\T(f) \right] \circ \left \langle  \langle  x, 0^\T \rangle, \langle 0^\T,y \rangle \right \rangle \\
&=~  \mu \circ \dd^2 \left[\T(f) \right] \circ \left \langle  \langle  x, 0 \rangle, \langle 0,y \rangle \right \rangle \\
&=~\mu \circ \dd[\T(f)] \circ \pair{x}{y} \tag*{\ref{CdCax6a}} \\
&=~  \dd^\T[f] \circ^\T  \langle x, y \rangle^\T
\end{align*}
\\
\noindent~\ref{CdCax7a}  $\dd^\T\left[\dd^\T[f] \right] \circ^\T \left \langle \langle x, y \rangle^\T, \langle z, w \rangle^\T \right \rangle^\T= \dd^\T\left[\dd^\T[f] \right] \circ \left \langle \langle x, z \rangle^\T, \langle y, w \rangle^\T \right \rangle^\T$
\begin{align*}
\dd^\T\left[\dd^\T[f] \right] \circ^\T \left \langle \langle x, y \rangle^\T, \langle z, w \rangle^\T \right \rangle^\T &=~\mu \circ \dd^2\left[ \T(f) \right] \circ \four{x}{y}{z}{w} \\
&=~  \mu \circ \dd \left[ \dd[\T(f)] \right] \circ \left \langle  \langle  x, z \rangle, \langle y,w \rangle \right \rangle \tag*{\ref{CdCax7a}} \\
&=~  \dd^\T\left[\dd^\T[f] \right] \circ \left \langle \langle x, z \rangle^\T, \langle y, w \rangle^\T \right \rangle^\T
\end{align*}
  So we conclude that the Kleisli category is a Cartesian difference category.
\end{proof}

We also point that in the case of a Cartesian differential category, since $\varepsilon =0$, it follows that $\varepsilon^T = 0$. Therefore we have that the Kleisli category of the tangent bundle monad of a Cartesian differential category is again a Cartesian differential category. To the knowledge of the authors, this is a novel observation.

\begin{cor}  For a Cartesian differential category $\mathbb{X}$, the Kleisli category $\mathbb{X}_\mathsf{T}$ is a Cartesian differential category with differential combinator $\D^\T = \D$.
\end{cor}

We conclude this section by briefly taking a look at the linear maps and the $\varepsilon^\T$-linear maps in the Kleisli category. A Kleisli map $f=\langle f_0, f_1 \rangle$ is linear in the Kleisli category if $\dd^\T[f] = f \circ^\T \pi^\T_1$, which amounts to requiring that:
\[ \langle \dd[f_0], \dd[f_1] \rangle = \langle f_0 \circ \pi_1, f_1 \circ \pi_1 \rangle \]
Therefore a Kleisli map is linear in the Kleisli category if and only if it is the pairing of maps which are linear in the base category. Similarly, a Kleisli map is $\varepsilon^\T$-linear if and only if is the pairing of $\varepsilon$-linear maps.

\section{Difference \texorpdfstring{$\lambda$}{Lambda}-Categories}%
\label{sec:differential-lambda-categories}

Categorical models of the differential $\lambda$-calculus~\cite{ehrhard2003differential} are known as differential $\lambda$-cate\-gories~\cite{bucciarelli2010categorical}. However, a differential $\lambda$-category is not simply a Cartesian differential category which is Cartesian closed. In a differential $\lambda$-category, both the additive structure and the differential structure must be compatible with the curry operator. The same is true of Cartesian difference categories. In this section, we introduce difference $\lambda$-categories. Briefly, a difference $\lambda$-category is a Cartesian difference category which is Cartesian closed and such that the Cartesian difference structure and the curry operations are compatible.

For a Cartesian closed category $\mathbb{X}$, we denote the exponential as $A \Rightarrow B$, the evaluation map as $\A{ev}: (A \Rightarrow B) \times A \to B$, and the curry of a map $f: A \times B \to C$ as $\Lambda(f): A \to (B \Rightarrow C)$, that is, $\Lambda(f)$ is the unique map such that $\A{ev} \circ (\Lambda(f) \times 1_B) = f$. Conversly, given a map of type $g: A \to (B \Rightarrow C)$, define $\Lambda^{-1}(g): A \times B \to C$ as $\Lambda^{-1}(g) = \A{ev} \circ (g \times 1_B)$. Of course, $\Lambda$ and $\Lambda^{-1}$ are inverses of each other, that is, $\Lambda\left( \Lambda^{-1}(g) \right) = g$ and $\Lambda^{-1}\left( \Lambda(f) \right) = f$. Another useful map will be the canonical natural isomorphism $\A{sw}: (A \times B) \times C \to (A \times C) \times B$ which swaps the last arguments, that is, $\A{sw}$ is defined as follows:
  \[
  \A{sw} :=  \pair{\pair{\pi_{00}}{\pi_1}}{\pi_{10}}
  \]
where recall that $\pi_{ij} = \pi_i \circ \pi_j$. Note that $\A{sw}$ is its own inverse, that is, $\A{sw} \circ \A{sw} = 1$.

\begin{defi} A \textbf{Cartesian closed left additive category} is a Cartesian left additive category $\mathbb{X}$ such that $\mathbb{X}$ is Cartesian closed and such that the curry operator preserves the additive structure, that is, $\Lambda(f + g) = \Lambda(f) + \Lambda(g)$ and $\Lambda(0) = 0$.
\end{defi}

\begin{defi}%
  \label{def:difference-lambda-category} A \textbf{difference $\lambda$-category} is a Cartesian difference category $\mathbb{X}$, with difference combinator $\dd$ and infinitesimal extension $\varepsilon$, such that $\mathbb{X}$ is a Cartesian closed left additive category and the following additional axioms hold:
 \begin{enumerate}[{\CdlCax{\arabic*}},ref={\CdlCax{\arabic*}},align=left]
    \item $\dd[\Lambda(f)] = \Lambda\pa{\dd[f]
      \circ \pair{\pa{\pi_0 \times 1_{}}}{\pa{\pi_1 \times 0}}
    }
    $\label{CdlCax1}
    \item $\Lambda(\varepsilon(f)) = \varepsilon\pa{\Lambda(f)}$\label{CdlCax2}
  \end{enumerate}
\end{defi}

The first axiom~\ref{CdlCax1} is identical to its differential combinator analogue in a differential $\lambda$-category. As such,  it follows the same broad intuition. First note that~\ref{CdlCax1} can also equivalently written in terms of $\A{sw}$ as follows:
  \[
    \dd[\Lambda(f)] = \Lambda\pa{
      \dd[f]
      \circ \pa{(1 \times 1) \times \pair{1}{0}}
      \circ \A{sw}
    }
  \]
which will be useful for many calculations in this section. Now given a map
$f : A \times B \to C$, we usually understand the composite:
\[\dd[f] \circ (1_{A \times B} \times \pair{1_A}{0})   \circ \A{sw} : (A \times A) \times B \to C\]
as the partial derivative of $f$ with respect to its first argument $A$. Hence,~\ref{CdlCax1} states that the
derivative of a curried function is precisely the curry of the partial derivative of the function with respect to its first argument. On the other hand,~\ref{CdlCax2} simply says that the curry of an infinitesimal extension of a function is the infinitesimal extension of the curry of the function. Using $\lambda$-calculus notation, this implies that $\lambda x . \varepsilon(f(t,y)) = \varepsilon (\lambda x . f(t,y))$.

\begin{exa}
Every differential $\lambda$-category~\cite{bucciarelli2010categorical} is precisely a difference $\lambda$-category such that $\varepsilon = 0$.
\end{exa}

\begin{exa}
  The category $\Ab$ (as defined in Section~\ref{discreteex}) is a difference $\lambda$-category. Given abelian groups $G, H$, the exponential $G \Rightarrow H$ is defined as the abelian group of all functions between $G$ and $H$, where the group structure is defined point-wise. The remaining Cartesian closed structure is defined in the standard way, that is, for an arbitrary function $f: G \times H \to K$, $\Lambda(f)(x)(y) = f(x,y)$. We leave it to the reader to check for themselves that $\Ab$ is indeed a Cartesian closed left additive category. Since $\varepsilon =1$, clearly~\ref{CdlCax2} holds automatically. So it remains to verify that the other axiom~\ref{CdlCax1} also holds:
  \begin{align*}
    \dd[\Lambda(f)](x, u)(y) &=~ \Lambda(f)(x + u)(y) - \Lambda(f)(x)(y)
    \\
    &=~ f(x + u, y) - f(x, y) \\
    &=~ f\left( (x,y) + (u,0) \right) - f(x,y) \\
    &=~\dd[f](x,y, u, 0) \\
    &=~ (\dd[f] \circ \left( (1 \times 1) \times \pair{1_{}}{0} \right) \circ \A{sw})(x,u,y) \\
    &=~ \Lambda(\dd[f] \circ \left( (1 \times 1) \times \pair{1_{}}{0} \right) \circ \A{sw})(x, u)(y)
  \end{align*}
  Therefore, $\Ab$ is indeed a difference $\lambda$-category.
\end{exa}

A central property of differential $\lambda$-categories is a deep correspondence
between differentiation and the evaluation map. As one would expect, the partial
derivative of the evaluation map gives a first-class derivative operator, see for example~\cite[Lemma~4.5]{bucciarelli2010categorical}, which provides
an interpretation for the differential substitution operator in the differential
$\lambda$-calculus. This property
still holds in difference $\lambda$-categories, although its formulation is somewhat more
involved.

\begin{lem}%
  \label{lem:lambda-d-ev} In a difference $\lambda$-category, for maps $g: A \times B \to C$ and $f: A \to B$, the following identities hold:
  \begin{enumerate}[(\roman{enumi}),ref={\thelem.\roman{enumi}}]
    \item $
    \dd[\A{ev} \circ \pair{\Lambda(g)}{f}]
    = \A{ev} \circ \pair{\dd[\Lambda(g)]}{f \circ \pi_0}
    +\dd[g] \circ \pair{\pair{\pi_0 + \varepsilon(\pi_1)}{f \circ
    \pi_0}}{\pair{0}{\dd[f]}}$
    \item $
      \dd[\A{ev} \circ \pair{\Lambda(g)}{f}]
      =
      \A{ev} \circ \pair{\dd[\Lambda(g)]}{f \circ \pi_0 + \varepsilon(\dd[f])}
      +\dd[g] \circ \pair{\pair{\pi_0}{f \circ \pi_0}}{\pair{0}{\dd[f]}}
    $
  \end{enumerate}
\end{lem}
\begin{proof} For (i), we compute:
\begin{align*}
   &\dd[\A{ev} \circ \pair{\Lambda(g)}{f}] =~ \dd[g \circ \pair{1_A}{f}] \\
   &=~ \dd[g] \circ \pair{\pair{1_A}{f} \circ \pi_0}{\dd[\pair{1_A}{f}]}  \tag*{\ref{CdCax5}} \\
   &=~ \dd[g] \circ \four{\pi_0}{f \circ \pi_0}{\pi_1}{\dd[f]} \tag*{\ref{CdCax3}+\ref{CdCax4}} \\
   &=~ \dd[g] \circ \pair{\pair{\pi_0}{f \circ \pi_0}}{ \pair{\pi_1}{0} + \pair{0}{\dd[f]} } \\
   &=~\dd[g] \circ \four{\pi_0}{f \circ \pi_0}{\pi_1}{0} + \dd[g] \circ \pair{\pair{\pi_0}{f \circ \pi_0} + \varepsilon(\pair{\pi_1}{0} )}{\pair{0}{\dd[f]} }  \tag*{\ref{CdCax2}} \\
   &=~\dd[g] \circ \four{\pi_0}{f \circ \pi_0}{\pi_1}{0} + \dd[g] \circ \pair{\pair{\pi_0}{f \circ \pi_0} + \pair{ \varepsilon(\pi_1)}{0}}{\pair{0}{\dd[f]} }  \tag{Lem~\ref{lem:ep-pair} +~\ref{Eax1}} \\
   &=~\dd[g] \circ \four{\pi_0}{f \circ \pi_0}{\pi_1}{0} + \dd[g] \circ \four{\pi_0 + \varepsilon(\pi_1)}{f \circ \pi_0}{0}{\dd[f]} \\
   &=~\dd[g] \circ \pair{\pi_0 \times 1_B}{\pi_1 \times 0} \circ \pair{1}{f \circ \pi_0} + \dd[g] \circ \four{\pi_0 + \varepsilon(\pi_1)}{f \circ \pi_0}{0}{\dd[f]} \\
&=~ \Lambda^{-1}\left( \Lambda\left( \dd[g] \circ \pair{\pi_0 \times 1_B}{\pi_1 \times 0} \right) \right) \circ \pair{1}{f \circ \pi_0} + \dd[g] \circ \four{\pi_0 + \varepsilon(\pi_1)}{f \circ \pi_0}{0}{\dd[f]} \\
&=~ \Lambda^{-1}\left( \dd[\Lambda[g]]\right) \circ \pair{1}{f \circ \pi_0} + \dd[g] \circ \four{\pi_0 + \varepsilon(\pi_1)}{f \circ \pi_0}{0}{\dd[f]} \tag{by~\ref{CdlCax1}} \\
&=~ \A{ev} \circ ( \dd[\Lambda[g]] \times 1_B) \circ \pair{1}{f \circ \pi_0} + \dd[g] \circ \four{\pi_0 + \varepsilon(\pi_1)}{f \circ \pi_0}{0}{\dd[f]} \\
&=~ \A{ev} \circ \pair{ \dd[\Lambda[g]]}{f \circ \pi_0} + \dd[g] \circ \four{\pi_0 + \varepsilon(\pi_1)}{f \circ \pi_0}{0}{\dd[f]}
\end{align*}
On the other hand, for (ii), we compute:
\begin{align*}
  & \dd[\A{ev} \circ \pair{\Lambda(g)}{f}] =~ \dd[g \circ \pair{1_A}{f}] \\
   &=~\dd[g] \circ \pair{\pair{1_A}{f} \circ \pi_0}{\dd[\pair{1_A}{f}]}  \tag*{\ref{CdCax5}} \\
   &=~ \dd[g] \circ \four{\pi_0}{f \circ \pi_0}{\pi_1}{\dd[f]} \tag*{\ref{CdCax3}+\ref{CdCax4}} \\
      &=~ \dd[g] \circ \pair{\pair{\pi_0}{f \circ \pi_0}}{\pair{0}{\dd[f]} + \pair{\pi_1}{0} } \\
      &=~\dd[g] \circ \four{\pi_0}{f \circ \pi_0}{0}{\dd[f]} + \dd[g] \circ \pair{\pair{\pi_0}{f \circ \pi_0} + \varepsilon(\pair{0}{\dd[f]})}{\pair{\pi_1}{0} }  \tag*{\ref{CdCax2}} \\
&=~\dd[g] \circ \four{\pi_0}{f \circ \pi_0}{0}{\dd[f]} + \dd[g] \circ \pair{\pair{\pi_0}{f \circ \pi_0} + \pair{0}{\varepsilon(\dd[f])} }{\pair{\pi_1}{0} }   \tag{Lem~\ref{lem:ep-pair} +~\ref{Eax1}} \\
&=~\dd[g] \circ \four{\pi_0}{f \circ \pi_0}{0}{\dd[f]}  + \dd[g] \circ \four{\pi_0}{f \circ \pi_0 + \varepsilon(\dd[f])}{\pi_1}{0} \\
&=~\dd[g] \circ \four{\pi_0}{f \circ \pi_0}{0}{\dd[f]}  + \dd[g] \circ \pair{\pi_0 \times 1_B}{\pi_1 \times 0} \circ  \pair{1}{f \circ \pi_0 + \varepsilon(\dd[f])}\\
&=~\dd[g] \circ \four{\pi_0}{f \circ \pi_0}{0}{\dd[f]}  + \Lambda^{-1}\left( \Lambda\left( \dd[g] \circ \pair{\pi_0 \times 1_B}{\pi_1 \times 0} \right) \right) \circ  \pair{1}{f \circ \pi_0 + \varepsilon(\dd[f])}\\
 &=~\dd[g] \circ \four{\pi_0}{f \circ \pi_0}{0}{\dd[f]}  +  \Lambda^{-1}\left( \dd[\Lambda[g]]\right)  \circ  \pair{1}{f \circ \pi_0 + \varepsilon(\dd[f])} \tag{by~\ref{CdlCax1}} \\
 &=~\dd[g] \circ \four{\pi_0}{f \circ \pi_0}{0}{\dd[f]}  + \A{ev} \circ ( \dd[\Lambda[g]] \times 1_B) \circ  \pair{1}{f \circ \pi_0 + \varepsilon(\dd[f])}\\
 &=~ \dd[g] \circ \four{\pi_0}{f \circ \pi_0}{0}{\dd[f]} + \A{ev} \circ \pair{\dd[\Lambda[g]]}{f \circ \pi_0 + \varepsilon(\dd[f])}\\
 &=~ \A{ev} \circ \pair{\dd[\Lambda[g]]}{f \circ \pi_0 + \varepsilon(\dd[f])} + \dd[g] \circ \four{\pi_0}{f \circ \pi_0}{0}{\dd[f]}
\end{align*}
Thus the desired identities hold. \end{proof}

\section{Conclusions and Future Work}

We have presented Cartesian difference categories, which generalize Cartesian differential categories to account for more discrete definitions of derivatives while providing additional structure that is absent in change action models. We have also exhibited important examples and shown that Cartesian difference categories arise quite naturally from considering tangent bundles in any Cartesian differential category. We claim that Cartesian difference categories can facilitate the exploration of differentiation in discrete spaces, by generalizing techniques and ideas from the study of their differential counterparts. For example, Cartesian differential categories can be extended to allow objects whose tangent space is not necessarily isomorphic to the object itself~\cite{cruttwell2017cartesian}. The same generalization could be applied to Cartesian difference categories --- with some caveats: for example, the equation defining a linear map (Definition~\ref{def:linearity}) becomes ill-typed, but the notion of $\varepsilon$-linear map remains meaningful.

Perhaps the most important topic for further research is identifying and cataloguing more instances of Cartesian difference
categories. There is a number of very natural candidates for such; among those, we would like to single out synthetic differential
geometry~\cite{kock2006synthetic} and non-standard analysis~\cite{keisler2013elementary}, both of which feature a notion
of ``infinitesimal elements''. In either case, multiplying by a suitably chosen infinitesimal $\varepsilon$ would be an obvious
choice of an infinitesimal extension. As suggested by one of the anonymous reviewers, the case of non-standard analysis based on hyperreal numbers might be of particular interest:
since they contain invertible infinitesimals (as opposed to the nilpotent infinitesimals from SDG), one could define the derivative
of a function $f$ as $\dd[f](x, y) = \varepsilon^{-1}\left( f(x + \varepsilon y) - f(x)\right)$ (assuming a choice of an
``infinitesimal unit'' $\varepsilon$). With an infinitesimal extension being defined by multiplication by the above infinitesimal
unit $\varepsilon$, most of the axioms of a Cartesian difference category can be proven by a suitable generalisation of the proofs
for the calculus of finite differences (Section~\ref{discreteex}), although more work would be needed to prove \CdCax{7}. Lie
groups, and Lie algebras, have also been suggested by one of the anonymous reviewers as a potential source of models, but work in this direction is at a less
mature stage.

Another relevant path to consider is developing the analogue of the ``tensor'' story for Cartesian difference categories. Indeed, an important source of examples of Cartesian differential categories are the coKleisli categories of a tensor differential category~\cite{blute2006differential,blute2009cartesian}. A similar result likely holds for a hypothetical ``tensor difference category'', but it is not immediately clear how these should be defined: \CdCax{2} implies that derivatives in the difference sense are non-linear and therefore their interplay with the tensor structure will be much different. As suggested by one of the anonymous reviewers, a potential leading example of a ``tensor difference category'' is $\mathsf{Ab}$, the category of abelian groups and group homomorphisms between them, which is a model of linear logic via a linear-non-linear adjunction between $\mathsf{Ab}$ and $\mathsf{SET}$, the category of sets and arbitrary set functions between them~\cite{benton1994mixed}. The induced comonad (which models the exponential modality of linear logic) on $\mathsf{Ab}$ is given by mapping an abelian group $G$ to the free abelian group over the underlying set of $G$. Even more interesting is the fact that the coKleisli category of this comonad is isomorphic to $\overline{\mathsf{Ab}}$, as defined in Section~\ref{discreteex}, which is one of our main examples of a Cartesian difference category. As such, this is a pretty convincing argument that $\mathsf{Ab}$ will indeed be a prime candidate for an example of a ``tensor difference category'' and help guide us in constructing the correct definition.

A further generalization of Cartesian differential categories, categories with tangent structure~\cite{cockett2014differential} are defined directly in terms of a tangent bundle functor rather than requiring that every tangent bundle be trivial (that is, in a tangent category it may not be the case that $\T A = A \times A$). Some preliminary research on change actions has already shown that, when generalized in this way, change actions are precisely internal categories, but the consequences of this for change action models (and, \emph{a fortiori}, Cartesian difference categories) are not understood. More recently, some work has emerged about differential equations using the language of tangent categories~\cite{cockett2017connections}. We believe similar techniques can be applied in a straightforward way to Cartesian difference categories, where they might be of use to give an abstract formalization of discrete dynamical systems and difference equations.

  %% the following bibliography is gererated manually for the sake of brevity
  %% only; please use a separate .bib file in your submission

 \bibliographystyle{alpha}
 \bibliography{references}
\end{document}